\documentclass{article} 
\usepackage[final]{neurips_2023}

\usepackage[hidelinks]{hyperref}
\usepackage{url}
\usepackage{todonotes}

\usepackage{graphicx}
\usepackage{apptools}
\usepackage[flushleft]{threeparttable}
\usepackage{array,booktabs,makecell}
\usepackage{multirow}
\usepackage{nicefrac}
\usepackage{amsthm}
\usepackage{amsmath,amsfonts,bm}
\usepackage{amssymb}
\usepackage{breqn}
\usepackage{wrapfig}
\usepackage{caption}
\usepackage{subcaption}
\usepackage{siunitx}
\usepackage{thm-restate}
\usepackage{empheq}
\usepackage{bbm}
\usepackage{tabularx}
\usepackage{pdflscape}
\usepackage{datatool}
\usepackage{listings}
\usepackage{xhfill}
\usepackage{pdfpages}

\usepackage{algorithmic}
\usepackage{algorithm}

\usepackage{color}
\usepackage{colortbl}
\definecolor{bgcolor}{rgb}{0.76,0.88,0.50}
\definecolor{bgcolor0}{rgb}{0.93,0.99,1}
\definecolor{bgcolor1}{rgb}{0.8,1,1}
\definecolor{bgcolor2}{rgb}{0.8,1,0.8}
\definecolor{bgcolor3}{rgb}{0.50,0.90,0.50}
\usepackage{tcolorbox}
\usepackage{pifont}
\definecolor{mydarkgreen}{rgb}{39,130,67}
\definecolor{mydarkred}{rgb}{192,25,25}

\newcommand{\norm}[1]{\left\| #1 \right\|}

\newcommand{\inp}[2]{\left\langle#1,#2\right\rangle} 

\newcommand{\R}{\mathbb{R}} 
\newcommand{\N}{\mathbb{N}} 

\newcommand{\Exp}[1]{{\mathbb{E}}\left[#1\right]}
\newcommand{\ExpSub}[2]{{\mathbb{E}}_{#1}\left[#2\right]}

\newcommand{\cC}{\mathcal{C}}

\newcommand{\cM}{\mathcal{M}}

\newcommand{\cO}{\mathcal{O}}

\theoremstyle{plain}
\newtheorem{theorem}{Theorem}[section]

\newtheorem{lemma}[theorem]{Lemma}

\theoremstyle{definition}
\newtheorem{definition}[theorem]{Definition}
\newtheorem{assumption}[theorem]{Assumption}
\theoremstyle{remark}
\newtheorem{remark}[theorem]{Remark}

\newcommand{\eqdef}{:=}

\makeatletter
\newcommand{\vast}{\bBigg@{4}}

\DeclareMathOperator*{\argmin}{arg\,min}

\newcommand{\squeeze}{\textstyle}

\usepackage{tikz}
\usetikzlibrary{positioning}

\hypersetup{
  colorlinks   = true, 
  urlcolor     = blue, 
  linkcolor    = {red!75!black}, 
  citecolor   = {blue!50!black} 
}

\makeatletter
\providecommand\theHALG@line{\thealgorithm.\arabic{ALG@line}}
\makeatother

\newcommand{\alglinelabel}{%
  \label
}

\newcommand{\alglinelabeladiana}{%
  \label
}

\newcommand{\alglinelabelef}{%
  \label
}

\newcounter{mynumber}

\allowdisplaybreaks

\title{\algnamebig{2Direction:} Theoretically Faster Distributed Training with Bidirectional Communication Compression}

\usepackage{tcolorbox}
\usepackage{pifont}
\definecolor{myxxx}{RGB}{130,67,39}
\definecolor{mydarkgreen}{RGB}{39,130,67}
\definecolor{mydarkorange}{RGB}{236,147,14}
\definecolor{mydarkred}{RGB}{192,47,25}
\newcommand{\green}{\color{mydarkgreen}}

\newcommand{\red}{\color{mydarkred}}

\newcommand{\algnamebig}[1]{{\sf\green #1}}
\newcommand{\algname}[1]{{\small\green\sf #1}}
\newcommand{\algnamesmall}[1]{{\scriptsize\green\sf #1}}

\author{%
  Alexander Tyurin\\
  KAUST\\
  Saudi Arabia\\
  \texttt{alexandertiurin@gmail.com} \\
  \And
  Peter Richt\'{a}rik \\
  KAUST\\
  Saudi Arabia\\
  \texttt{richtarik@gmail.com} \\
}

\begin{document}

\newread\myread
\openin\myread=result.txt
\read\myread to \formulakappaexpand
\read\myread to \formulakappaconst
\read\myread to \formulakappasol
\read\myread to \formularhoexpand
\read\myread to \formularhoconst
\read\myread to \formularhosol
\read\myread to \formulabregman
\read\myread to \formulabregmanconst
\read\myread to \formuladist
\read\myread to \formuladistconst
\read\myread to \formulafollows

\maketitle

\begin{abstract}
  We consider distributed convex optimization problems in the regime when the communication between the server and the workers is expensive in both uplink and downlink directions. We develop a new and provably accelerated method, which we call \algname{2Direction}, based on fast bidirectional compressed communication and a new bespoke error-feedback mechanism which may be of independent interest. Indeed, we find that the \algname{EF} and \algname{EF21-P} mechanisms \citep{Seide2014, gruntkowska2022ef21} that have considerable success in the design of efficient non-accelerated methods are not appropriate for accelerated methods. In particular, we prove that \algname{2Direction} improves the previous state-of-the-art communication complexity $\widetilde{\Theta}\left(K \times \left(\nicefrac{L}{\alpha \mu} + \nicefrac{L_{\max} \omega}{n \mu} + \omega\right)\right)$ \citep{gruntkowska2022ef21} to $\widetilde{\Theta}(K \times (\sqrt{\nicefrac{L (\omega + 1)}{\alpha \mu}} + \sqrt{\nicefrac{L_{\max} \omega^2}{n \mu}} + \nicefrac{1}{\alpha} + \omega))$ in the $\mu$--strongly-convex setting, where $L$ and $L_{\max}$ are smoothness constants, $n$ is \# of workers, $\omega$ and $\alpha$ are compression errors of the Rand$K$ and Top$K$ sparsifiers (as examples), $K$ is \# of coordinates/bits that the server and workers send to each other. Moreover, our method is the first that improves upon the communication complexity of the vanilla accelerated gradient descent (\algname{AGD}) method \citep{nesterov2018lectures}. We obtain similar improvements in the general convex regime as well. Finally, our theoretical findings are corroborated by experimental evidence.
\end{abstract}

\section{Introduction}

We consider convex optimization problems in the centralized distributed setting. These types of problems appear in federated learning \citep{konevcny2016federated, mcmahan2017communication} and distributed optimization \citep{ramesh2021zero}. In this setting, one of the main problems is the communication bottleneck: the connection link between the server and the workers can be very slow. We focus our attention on methods that aim to address this issue by applying {\em lossy compression} to the communicated messages \citep{alistarh2017qsgd, DIANA,gruntkowska2022ef21}.  

\subsection{The problem} Formally, we consider the  optimization problem
\begin{align}
   \label{eq:main_task}
 \squeeze \min \limits_{x \in \R^d} \left\{f(x) \eqdef \frac{1}{n} \sum\limits_{i=1}^n f_i(x)\right\},
\end{align}
where $n$ is the number of workers and $f_i \,:\, \R^d \rightarrow \R$ are smooth convex functions for all $i \in [n] \eqdef \{1, \dots, n\}.$ We consider the {\em centralized distributed optimization} setting in which
 each $i$\textsuperscript{th} worker contains the function $f_i,$ and all workers are directly connected to a server \citep{kairouz2021advances}.
In general, we want to find a (possibly random) point $\widehat{x}$ such that $\Exp{f(\widehat{x})} - f(x^*) \leq \varepsilon,$ where $x^*$ is an optimal point. In the strongly convex setup, we also want to guarantee that $\mathbb{E}[\norm{\widetilde{x} - x^*}^2] \leq \varepsilon$ for some point $\widetilde{x}.$

Virtually all other theoretical works in this genre assume that, compared to the worker-to-server (w2s) communication cost, the server-to-workers (s2w) broadcast is so fast  that it can be ignored. We lift this limitation and instead associate a relative cost $r \in [0,1]$ with the two directions of communication. If $r=0$, then s2w communication is free,  if $r=1$, then w2s communication is free, and if $r=\nicefrac{1}{2}$, then the s2w and w2s costs are equal.  All our theoretical results hold for any $r \in [0,1]$. We formalize and elaborate upon this setup in Section~\ref{sec:current}.

\subsection{Assumptions}

Throughout the paper we rely on several standard assumptions on the functions $f_i$ and $f.$

\begin{assumption}
  \label{ass:workers_lipschitz_constant}
Functions $f_i$ are $L_i$--smooth, i.e., 
  $\norm{\nabla f_i(x) - \nabla f_i(y)} \leq L_i \norm{x - y}$ for all $x, y \in \R^d,$ for all $i \in [n].$
We let $L_{\max} \eqdef \max_{i \in [n]} L_i$. Further, let $\widehat{L}>0$ be a constant such that
  $\frac{1}{n} \sum_{i=1}^n \norm{\nabla f_i(x) - \nabla f_i(y)}^2 \leq \widehat{L}^2 \norm{x - y}^2$ for all $x, y \in \R^d.$\\
  Note that if the functions $f_i$ are $L_i$--smooth for all $i \in [n],$ then $\widehat{L} \leq L_{\max}.$
\end{assumption}
\begin{assumption}
  \label{ass:lipschitz_constant}
Function $f$ is $L$--smooth, i.e., $\norm{\nabla f(x) - \nabla f(y)} \leq L \norm{x - y}$ for all $x, y \in \R^d$.
\end{assumption}
\begin{assumption}
  \label{ass:convex}
Functions $f_i$ are convex for all $i \in [n],$ and  $f$ is $\mu$-strongly convex with $\mu \geq 0$, attaining a minimum at some point $x^* \in \R^d.$    
\end{assumption}
It is known that the above smoothness constants are related in the following way.
\begin{lemma}[\citet{gruntkowska2022ef21}]
  \label{lemma:lipt_constants}
If Assumptions~\ref{ass:lipschitz_constant}, \ref{ass:workers_lipschitz_constant} and \ref{ass:convex} hold, then $\widehat{L} \leq L_{\max} \leq n L$ and $L \leq \widehat{L} \leq \sqrt{L_{\max} L}.$
\end{lemma}

\begin{table*}
  \caption{\footnotesize\textbf{Communication Rounds in the Strongly Convex Case.} The number of communication rounds and rounds costs to get an $\varepsilon$-solution ($\Exp{\norm{\widehat{x} - x^*}^2} \leq \varepsilon$) up to logarithmic factors. The table shows the most relevant bidirectional compressed methods that are ordered by the total communication complexity \# {\bf Communication Rounds} $ \times $ {\bf Round Cost} (see \eqref{eq:main_complexity} for details). \\
  \phantom{XX} i. The parameter $r$ weights the importance/speed of uplink and downlink connections. When $r = \nicefrac{1}{2},$ it means that the uplink and downlink speeds are equal. \\
   \phantom{XX}  ii. The parameters $K_{\omega}$ and $K_{\alpha}$ are the expected densities Definition~\ref{def:expected_density} of compressors $\cC^{D} \in \mathbb{U}(\omega)$ and $\cC^{P} \in \mathbb{B}(\alpha)\textsuperscript{\color{blue}(a)},$ that operate in the workers and the server accordingly. Less formally, $K_{\omega}$ and $K_{\alpha}$ are the number of coordinates/bits that the workers and the server send to each other in each communication round.}
  \label{table:strongly_convex_case}
  \centering 
   \scriptsize  
  \begin{threeparttable}
    \begin{tabular}{cccccc}
\toprule
   {\bf Method} & \# {\bf Communication Rounds} & {\bf Round Cost\textsuperscript{\color{blue}(c)}}\\
   \midrule
   \scriptsize \makecell{\algnamesmall{Dore}, \algnamesmall{Artemis}, \algnamesmall{MURANA}\textsuperscript{\color{blue}(a)} \\
      \citep{liu2020double} \\
      \citep{philippenko2020artemis} \\ \citep{condat2022murana}} & $\widetilde{\Omega}\left(\frac{\omega}{\alpha n}\frac{L_{\max}}{\mu}\right)$\textsuperscript{\color{blue}(f)} & {\tiny $(1 - r)K_{\omega} + r K_{\alpha}$}\\
      \midrule
      \scriptsize \makecell{\algnamesmall{MCM}\textsuperscript{\color{blue}(a)} \\ \citep{philippenko2021preserved}} & $\widetilde{\Omega}\left(\left(\frac{1}{\alpha^{3/2}} + \frac{\omega^{1/2}}{\alpha \sqrt{n}} + \frac{\omega}{n}\right)\frac{L_{\max}}{\mu}\right)$\textsuperscript{\color{blue}(f)} & {\tiny $(1 - r)K_{\omega} + r K_{\alpha}$}\\
    \midrule
    \scriptsize \makecell{\algnamesmall{GD} \\ \citep{nesterov2018lectures}} & $\frac{L}{\mu}$ & {\tiny $d$} \\
     \midrule
     \scriptsize \makecell{\algnamesmall{EF21-P + DIANA} \\
     \citep{gruntkowska2022ef21}} & $\squeeze \frac{L}{\alpha \mu} + \frac{L_{\max} \omega}{n \mu} + \omega$ & {\tiny $(1 - r)K_{\omega} + r K_{\alpha}$}\\
     \midrule
     \scriptsize \makecell{\algnamesmall{AGD} \\ \citep{nesterov2018lectures}} & $\sqrt{\frac{L}{\mu}}$ & {\tiny $d$} \\
    \midrule
    \cellcolor{bgcolor1} \begin{tabular}{c} \makecell{\algnamesmall{2Direction}} \\ \scriptsize (Remark~\ref{rem:realistic})\textsuperscript{\color{blue}(b)}, (Theorem~\ref{cor:realistic})\end{tabular} & \cellcolor{bgcolor1} $\sqrt{\frac{L (\omega + 1)}{\alpha \mu}} + \sqrt{\frac{L_{\max} \omega^2}{n \mu}} + \frac{1}{\alpha} + \omega$ & \cellcolor{bgcolor1} {\tiny $(1 - r)K_{\omega} + r K_{\alpha}$} \\
  \midrule
  \cellcolor{bgcolor1} \begin{tabular}{c} \makecell{\algnamesmall{2Direction}} \\ \scriptsize(Remark~\ref{rem:optimistic})\textsuperscript{\color{blue}(b)}, (Theorem~\ref{cor:optimistic}) \\ \scriptsize (requires $\nicefrac{L_{\max}}{L}$)\textsuperscript{\color{blue}(d)}\end{tabular} & \cellcolor{bgcolor1} \makecell{$\sqrt{\frac{L \max\{1, r \left(\omega + 1\right)\}}{\alpha \mu}} + \sqrt{\frac{L^{2/3} L_{\max}^{1/3} (\omega + 1)}{\alpha n^{1/3} \mu}} +$ \\
  $ \sqrt{\frac{L^{1/2} L_{\max}^{1/2} (\omega + 1)^{3/2}}{\sqrt{\alpha n} \mu}} + \sqrt{\frac{L_{\max} \omega^2}{n \mu}} + \frac{1}{\alpha} + \omega$} & \cellcolor{bgcolor1} {\tiny $(1 - r)K_{\omega} + r K_{\alpha}$} \\
    \midrule
    \bottomrule
    \end{tabular}
    \begin{tablenotes}
      \scriptsize
      \item [{\color{blue}(a)}] The \algnamesmall{Dore}, \algnamesmall{Artemis}, \algnamesmall{MURANA}, and \algnamesmall{MCM} methods do not support {\em biased} compressors for server-to-worker compression. In these methods,  the error $\alpha$ equals $\nicefrac{1}{(\omega_{\textnormal{s}} + 1)},$ where the error $\omega_{\textnormal{s}}$ is a parameter of an unbiased compressor that is used in server-to-worker compression. For these methods, we define $\nicefrac{1}{(\omega_{\textnormal{s}} + 1)}$ as $\alpha$ to make comparison easy with \algnamesmall{EF21-P + DIANA} and \algnamesmall{2Direction}.
      \item [{\color{blue}(b)}] In this table, we present the simplified iteration complexity of \algnamesmall{2Direction} assuming that $r \leq \nicefrac{1}{2}$ and $\omega + 1 = \Theta\left(\nicefrac{d}{K_{\omega}}\right).$ The full complexities are in \eqref{eq:realistic_compl} and \eqref{eq:optimistic_compl}. In Section~\ref{sec:sanity_check}, we show that \algnamesmall{2Direction} has no worse total communication complexity than \algnamesmall{EF21-P + DIANA} for all $r \in [0, 1]$ and for any choice of compressors.
      \item [{\color{blue}(c)}] We define {\bf Round Cost} of a method $\cM$ as a constant such that $\mathfrak{m}^{r}_{{\scriptscriptstyle \cM}} = $ \# {\bf Communication Rounds} $ \times $ {\bf Round Cost}, where $\mathfrak{m}^{r}_{{\scriptscriptstyle \cM}}$ is the total communication complexity \eqref{eq:main_complexity}.
      \item [{\color{blue}(d)}] \algnamesmall{2Direction} can have even better total communication complexity if the algorithm can use the ratio $\nicefrac{L_{\max}}{L}$ when selecting the parameters $\tau$ and $p$ in Algorithm~\ref{alg:bi_diana}. For instance, this is the case if we assume that  $L_{\max} = L, $ which was done by \citet{ADIANA, li2021canita}, for example.
      \item [{\color{blue}(f)}] The notation $\widetilde{\Omega}\left(\cdot\right)$ means ``at least up to logarithmic factors.''
    \end{tablenotes}
\end{threeparttable}      
\end{table*}

\section{Motivation: From Unidirectional to Bidirectional Compression} \label{sec:current}

In this work, we  distinguish between  worker-to-server (w2s=uplink) and  server-to-worker (s2w=downlink) communication cost, and define w2s and s2w communication complexities of methods in the following natural way.

\begin{definition}
  \label{def:compl}
  For a centralized distributed method $\cM$ aiming to solve problem \eqref{eq:main_task}, the communication complexity $\mathfrak{m}^{{\scriptscriptstyle \textnormal{w2s}}}_{{\scriptscriptstyle \cM}}$ is the expected number of coordinates/floats\footnote{Some works measure bits instead of coordinates. For computer systems, where coordinates are represented by 32 or 64 bits, these measures are equivalent up to the constant factors 32 or 64.} that each worker sends to the server to  solve problem \eqref{eq:main_task}. The quantity $\mathfrak{m}^{{\scriptscriptstyle \textnormal{s2w}}}_{{\scriptscriptstyle \cM}}$ is the expected number of floats/coordinates the server broadcasts to the workers to solve problem \eqref{eq:main_task}. If $\mathfrak{m}^{{\scriptscriptstyle \textnormal{w2s}}}_{{\scriptscriptstyle \cM}} = \mathfrak{m}^{{\scriptscriptstyle \textnormal{s2w}}}_{{\scriptscriptstyle \cM}},$ then we use the simplified notation $\mathfrak{m}_{{\scriptscriptstyle \cM}} \eqdef \mathfrak{m}^{{\scriptscriptstyle \textnormal{s2w}}}_{{\scriptscriptstyle \cM}} = \mathfrak{m}^{{\scriptscriptstyle \textnormal{w2s}}}_{{\scriptscriptstyle \cM}}.$
\end{definition}

Let us illustrate the above concepts on the simplest baseline: vanilla gradient descent (\algname{GD}). It is well known \citep{nesterov2018lectures} that for $L$--smooth, $\mu$--strongly convex problems, \algname{GD} returns an $\varepsilon$-solution after $\cO\left(\nicefrac{L}{\mu}\log\nicefrac{1}{\varepsilon}\right)$ iterations. In each iteration, the workers and the server communicate all $\Theta(d)$ coordinates to each other (since  no compression is applied). Therefore, the communication complexity of \algname{GD} is $\mathfrak{m}_{{\scriptscriptstyle \textnormal{GD}}} = \Theta\left(\nicefrac{d L}{\mu}\log\nicefrac{1}{\varepsilon}\right).$ The same reasoning applies to the accelerated gradient method (\algname{AGD}) \citep{nesterov2018lectures}, whose communication complexity is $\mathfrak{m}_{{\scriptscriptstyle \textnormal{AGD}}} = \Theta(d \sqrt{\nicefrac{L}{\mu}}\log\nicefrac{1}{\varepsilon}).$

\subsection{Compression mappings}

In the literature, researchers often use the following two families of compressors:
\begin{definition}
  \label{def:biased_compression}
  A (possibly) stochastic mapping $\cC\,:\,\R^d \rightarrow \R^d$ is a \textit{biased  compressor} if
  there exists $\alpha \in (0,1]$ such that
  \begin{align}
      \label{eq:biased_compressor}
      \qquad \Exp{\norm{\cC(x) - x}^2} \leq (1 - \alpha) \norm{x}^2, \qquad \forall x \in \R^d.
  \end{align}
\end{definition}

\begin{definition}
  \label{def:unbiased_compression}
  A stochastic mapping $\cC\,:\,\R^d \rightarrow \R^d$ is an \textit{unbiased compressor} if
  there exists $\omega \geq 0$ such that
  \begin{align}
      \label{eq:compressor}
      \Exp{\cC(x)} = x, \qquad \Exp{\norm{\cC(x) - x}^2} \leq \omega \norm{x}^2, \qquad \forall x \in \R^d.
  \end{align}
\end{definition}

We will make use of the following assumption. 
\begin{assumption}
  \label{ass:unbiased_compressors}
The randomness in all compressors used in our method is drawn independently.
\end{assumption}

Let us denote the set of mappings satisfying Definition~\ref{def:biased_compression} and \ref{def:unbiased_compression} by $\mathbb{B}(\alpha)$ and $\mathbb{U}(\omega)$, respectively. The family of biased compressors $\mathbb{B}$ is wider. Indeed, it is well known if $\cC \in \mathbb{U}(\omega),$ then $\nicefrac{1}{(\omega + 1)} \cdot \cC \in \mathbb{B}\left(\nicefrac{1}{(\omega + 1)}\right)$. The canonical sparsification operators belonging to these classes are Top$K \in \mathbb{B}(\nicefrac{K}{d})$ and Rand$K \in \mathbb{U}(\nicefrac{d}{K} - 1)$. The former  outputs the $K$ largest values (in magnitude) of the input vector, while the latter outputs $K$ random values of the input vector, scaled by $\nicefrac{d}{K}$ \citep{beznosikov2020biased}.  Following \citep{MARINA,tyurin2022dasha}, we now define the {\em expected density} of a sparsifier as a way to  formalize its {\em compression} performance.
\begin{definition}
  \label{def:expected_density}
The expected density of a sparsifier $\cC:\R^d\to \R^d$ is the quantity $K_{\cC}\eqdef\sup_{x \in \R^d} \Exp{\norm{\cC(x)}_0},$ where $\norm{y}_0$ is the number of of non-zero components of $y \in \R^d.$
\end{definition}
Trivially, for the Rand$K$ and Top$K$ sparsifiers we have  $K_{\cC} = K.$

\subsection{Unidirectional (i.e., w2s) compression}
\label{sec:worker_to_server}

As mentioned in the introduction, virtually all theoretical works in the area of compressed communication ignore s2w communication cost and instead aim to minimize $\mathfrak{m}^{{\scriptscriptstyle \textnormal{w2s}}}_{{\scriptscriptstyle \cM}}.$  Algorithmic work related to methods that only perform w2s compression has a long history, and this area is  relatively well understood \citep{alistarh2017qsgd,DIANA,richtarik2021ef21}. 

We refer to the work of \citet{gruntkowska2022ef21} for a detailed discussion of the communication complexities of {\em non-accelerated} methods in the convex and non-convex settings. For instance, using  Rand$K$, the \algname{DIANA} method of \citet{DIANA} provably improves\footnote{Indeed, using Lemma~\ref{lemma:lipt_constants}, $K \leq d,$ and $L \geq \mu$, one can easily show that $d + \nicefrac{K L}{\mu} + \nicefrac{d L_{\max}}{n \mu} = \cO(\nicefrac{d L}{\mu}).$} the communication complexity of \algname{GD} to $\mathfrak{m}^{{\scriptscriptstyle \textnormal{w2s}}}_{{\scriptscriptstyle \textnormal{DIANA}}} = \widetilde{\Theta}\left(d + \nicefrac{K L}{\mu} + \nicefrac{d L_{\max}}{n \mu}\right).$ {\em Accelerated} methods focusing on w2s compression are also well investigated. For example, \citet{ADIANA} and \citet{li2021canita} developed accelerated methods, which are based on \citep{DIANA, kovalev2020don}, and provably improve the w2s complexity of \algname{DIANA}. Moreover, using Rand$K$  with $K \leq \nicefrac{d}{n}$, \algname{ADIANA} improves the communication complexity of \algname{AGD} to $\mathfrak{m}^{{\scriptscriptstyle \textnormal{w2s}}}_{{\scriptscriptstyle \textnormal{ADIANA}}} = \widetilde{\Theta}(d + d\sqrt{\nicefrac{L_{\max}}{n \mu}}).$

\subsection{Bidirectional  (i.e., w2s and s2w) compression}

The methods mentioned in Section~\ref{sec:worker_to_server} do {\em not} perform server-to-workers (s2w) compression, and one can show that the server-to-workers (s2w) communication complexities of these methods are worse than $\mathfrak{m}_{{\scriptscriptstyle \textnormal{AGD}}} = \widetilde{\Theta}(d \sqrt{\nicefrac{L}{\mu}}).$ For example, using the Rand$K$, the s2w communication complexity of \algname{ADIANA} is at least $\mathfrak{m}^{{\scriptscriptstyle \textnormal{s2w}}}_{{\scriptscriptstyle \textnormal{ADIANA}}} = \widetilde{\Omega}(d \times \omega) = \widetilde{\Omega}(\nicefrac{d^2}{K}),$ which can be $\nicefrac{d}{K}$ times larger than in \algname{GD} or \algname{AGD}.  Instead of $\mathfrak{m}^{{\scriptscriptstyle \textnormal{w2s}}}_{{\scriptscriptstyle \cM}},$ methods performing bidirectional compression attempt to minimize \emph{the total communication complexity}, which we define as a convex combination of the w2s and s2w communication complexities:
\begin{align}
  \label{eq:main_complexity}
  \mathfrak{m}^{r}_{{\scriptscriptstyle \cM}} \eqdef (1 - r)\mathfrak{m}^{{\scriptscriptstyle \textnormal{w2s}}}_{{\scriptscriptstyle \cM}} + r \mathfrak{m}^{{\scriptscriptstyle \textnormal{s2w}}}_{{\scriptscriptstyle \cM}}.
\end{align}

The parameter $r \in [0, 1]$ weights the importance of uplink (w2s) and downlink (s2w) connections\footnote{$\mathfrak{m}^{r}_{{\scriptscriptstyle \cM}} \propto s^{{\scriptscriptstyle \textnormal{w2s}}} \mathfrak{m}^{{\scriptscriptstyle \textnormal{w2s}}}_{{\scriptscriptstyle \cM}} + s^{{\scriptscriptstyle \textnormal{s2w}}} \mathfrak{m}^{{\scriptscriptstyle \textnormal{s2w}}}_{{\scriptscriptstyle \cM}},$ where $s^{{\scriptscriptstyle \textnormal{w2s}}}$ and $s^{{\scriptscriptstyle \textnormal{s2w}}}$ are connection speeds, and $r = \nicefrac{s^{{\scriptscriptstyle \textnormal{s2w}}}}{s^{{\scriptscriptstyle \textnormal{w2s}}} + s^{{\scriptscriptstyle \textnormal{s2w}}}}.$}. 
Methods from Section~\ref{sec:worker_to_server} assume that $r=0$, thus ignoring the s2w communication cost. On the other hand, when $r = \nicefrac{1}{2},$ the uplink and downlink communication speeds are equal. By considering any $r\in [0,1]$, our methods and findings are applicable to more situations arising in practice. 
Obviously, $\mathfrak{m}^{r}_{{\scriptscriptstyle \textnormal{GD}}} = \mathfrak{m}_{{\scriptscriptstyle \textnormal{GD}}}$ and $\mathfrak{m}^{r}_{{\scriptscriptstyle \textnormal{AGD}}} = \mathfrak{m}_{{\scriptscriptstyle \textnormal{AGD}}}$ for all $r \in [0, 1].$ Recently, \citet{gruntkowska2022ef21} proposed the \algname{EF21-P + DIANA} method. This is the first  method supporting bidirectional compression that provably improves both the w2s and s2w complexities of \algname{GD}: $\mathfrak{m}^{r}_{{\scriptscriptstyle \textnormal{EF21-P + DIANA}}} \leq \mathfrak{m}_{\textnormal{GD}}$ for all $r \in [0, 1].$ Bidirectional methods designed before \algname{EF21-P + DIANA}, including \citep{DoubleSqueeze, liu2020double, philippenko2021preserved}, do not guarantee the total communication complexities better than that of \algname{GD}. The \algname{EF21-P + DIANA} method is {\em not} an accelerated method and, in the worst case, can have communication complexities worse than \algname{AGD} when the condition number $\nicefrac{L}{\mu}$ is large.

\section{Contributions}

Motivated by the above discussion, in this work we aim to address the following 
\makeatletter
\newcommand{\setword}[2]{%
  \phantomsection
  #1\def\@currentlabel{\unexpanded{#1}}\label{#2}%
}
\makeatother

{\bf \setword{Main Problem}{word:problem}}: 
\begin{quote}\bf Is it possible to develop a method supporting bidirectional communication compression that  improves the current best theoretical total communication complexity of \algname{EF21-P + DIANA}, and guarantees the total communication complexity to be no worse than the communication complexity 
$\mathfrak{m}_{{\scriptscriptstyle \textnormal{AGD}}} = \widetilde{\Theta}(d \sqrt{\nicefrac{L}{\mu}})$ of \algname{AGD}, while improving on \algname{AGD} in at least some regimes?
\end{quote}

\phantom{X} {\bf A)} We develop a new fast method (\algname{2Direction}; see Algorithm~\ref{alg:bi_diana}) supporting bidirectional communication compression. Our analysis leads to new state-of-the-art complexity rates in the centralized distributed setting (see Table~\ref{table:strongly_convex_case}), and as a byproduct, we answer \ref{word:problem} in the affirmative.

\phantom{X} {\bf B)} \citet{gruntkowska2022ef21} proposed to use the \algname{EF21-P} error-feedback mechanism \eqref{eq:ef21_p} to improve the convergence rates of  \emph{non-accelerated} methods supporting bidirectional communication compression. \algname{EF21-P} is a reparameterization of the celebrated \algname{EF} mechanism \citep{Seide2014}. We tried to use \algname{EF21-P} in our method as well, but failed. Our failures indicated that a fundamentally new approach is needed, and this eventually led us to design a new error-feedback mechanism \eqref{eq:unitracktable_w} that is more appropriate for \emph{accelerated} methods. We believe that this  is a contribution of independent interest that might motivate future growth in the area.

\phantom{X} {\bf C)} Unlike previous theoretical works \citep{ADIANA, li2021canita} on accelerated methods, we present a unified analysis  in both the $\mu$--strongly-convex and general convex cases. Moreover, in the general convex setting and low accuracy regimes, our analysis improves the rate $\cO\left(\nicefrac{1}{\varepsilon^{1/3}}\right)$ of \citet{li2021canita} to $\cO\left(\log \nicefrac{1}{\varepsilon}\right)$ (see details in Section~\ref{sec:canita}).

\phantom{X} {\bf D)} Even though our central goal was to obtain new SOTA \emph{theoretical} communication complexities for centralized distributed optimization, we show that the newly developed algorithm enjoys faster communication complexities in practice
as well (see details in Section~\ref{sec:experiments}).

\begin{algorithm}[ht]
  \caption{\algname{2Direction}: A Fast Gradient Method Supporting Bidirectional Compression} 
  \label{alg:bi_diana}
  \footnotesize
  \begin{algorithmic}[1]
  \STATE \textbf{Parameters:} Lipschitz-like parameter $\bar{L} > 0$, strong-convexity parameter $\mu \geq 0,$ probability $p \in (0, 1],$ parameter $\Gamma_0 \geq 1$, momentum $\tau \in (0, 1],$ contraction parameter $\alpha \in (0, 1]$ from \eqref{eq:biased_compressor}, initial point $x^0 \in \R^d,$ initial gradient shifts $h^0_1, \dots , h^0_n\in \R^d,$ gradient shifts $k^0 \in \R^d$ and $v^0 \in \R^d$
  \STATE Initialize $\beta = \nicefrac{1}{(\omega + 1)},$ $w^0 = z^0 = u^0 = x^0,$ and $h^0 = \frac{1}{n}\sum_{i = 1}^n h_i^0$
  \FOR{$t = 0, 1, \dots, T - 1$} 
  \STATE $\Gamma_{t+1}, \gamma_{t+1}, \theta_{t+1}= \textnormal{CalculateLearningRates}(\Gamma_t, \bar{L}, \mu, p, \alpha, \tau, \beta)$ \hfill{\scriptsize \color{gray} Get learning rates using Algorithm~\ref{alg:learning_rates}}
  \FOR{$i = 1, \dots, n$ in parallel}
  \STATE $y^{t+1} = \theta_{t+1} {\green w^t} + (1 - \theta_{t+1}) z^t$ \alglinelabel{line:bi_diana:y}
  \STATE $m_i^{t,y} = \cC_i^{D, y}(\nabla f_i(y^{t+1}) - h_i^t)$  
  \hfill{\scriptsize \color{gray} Worker $i$ compresses the shifted gradient via the compressor $\cC_i^{D, y} \in \mathbb{U}(\omega)$}
  \STATE {Send compressed message $m_i^{t,y}$ to the server}
  \ENDFOR
  \STATE $g^{t+1} = h^t + \frac{1}{n}\sum_{i=1}^n m_i^{t,y}$ 
  \STATE $u^{t+1} = \argmin_{x \in \R^d} \inp{g^{t+1}}{x} + \frac{\bar{L} + \Gamma_{t} \mu}{2 \gamma_{t+1}} \norm{x - u^t}^2 + \frac{\mu}{2} \norm{x - y^{t+1}}^2$ \hfill{\scriptsize \color{gray} A gradient-like descent step}
  \STATE {\green $q^{t+1} = \argmin_{x \in \R^d} \inp{k^t}{x} + \frac{\bar{L} + \Gamma_{t} \mu}{2 \gamma_{t+1}} \norm{x - w^t}^2 + \frac{\mu}{2} \norm{x - y^{t+1}}^2$} \alglinelabel{line:bi_diana:q}
  \STATE {\green $p^{t+1} = \cC^{P}\left(u^{t+1} - q^{t+1}\right)$}  \hfill {\scriptsize \color{gray} Server compresses the shifted model via the compressor $\cC^{P} \in \mathbb{B}\left(\alpha\right)$} \alglinelabel{line:bi_diana:p}
  \STATE {\green $w^{t+1} = q^{t+1} + p^{t+1}$} \alglinelabel{line:bi_diana:w}
  \STATE $x^{t+1} = \theta_{t+1} u^{t+1} + (1 - \theta_{t+1}) z^t$
  \STATE {\green Send compressed message $p^{t+1}$ to all $n$ workers} \alglinelabel{line:bi_diana:broadcast_compr}
  \STATE Flip a coin $c^t \sim \textnormal{Bernoulli}(p)$
  \STATE {\green $k^{t+1} = \begin{cases}
    v^{t}, &c^t = 1 \\
    k^{t}, &c^t = 0
  \end{cases}$} \quad and \quad $z^{t+1} = \begin{cases}
    x^{t+1}, &c^t = 1 \\
    z^{t}, &c^t = 0
  \end{cases}$ \alglinelabel{line:bi_diana:k}
  \IF{$c^t = 1$}
  \STATE {\green Broadcast non-compressed messages $x^{t+1}$ and $k^{t+1}$ to all $n$ workers} \hfill{\scriptsize \color{gray} With small probability $p$!} \alglinelabel{line:bi_diana:broadcast}
  \ENDIF
  \FOR{$i = 1, \dots, n$ in parallel}
  \STATE {\green $q^{t+1} = \argmin_{x \in \R^d} \inp{k^t}{x} + \frac{\bar{L} + \Gamma_{t} \mu}{2 \gamma_{t+1}} \norm{x - w^t}^2 + \frac{\mu}{2} \norm{x - y^{t+1}}^2$} \hfill{\scriptsize \color{gray}} \alglinelabel{line:bi_diana:q_local}
  \STATE {\green $w^{t+1} = q^{t+1} + p^{t+1}$} \alglinelabel{line:bi_diana:w_local}
  \STATE $z^{t+1} = \begin{cases}
    x^{t+1}, &c^t = 1 \\
    z^{t}, &c^t = 0
  \end{cases}$
  \STATE $m_i^{t,z} = \cC_i^{D, z}(\nabla f_i(z^{t+1}) - h_i^t)$  \hfill{\scriptsize \color{gray} Worker $i$ compresses the shifted gradient via the compressor $\cC_i^{D, z} \in \mathbb{U}(\omega)$}
  \STATE $h_i^{t+1} = h_i^t + \beta m_i^{t,z}$ 
  \STATE {Send compressed message $m_i^{t,z}$ to the server}
  \ENDFOR
  {\green \STATE $v^{t+1} = (1 - \tau) v^{t} + \tau \left(h^t + \frac{1}{n} \sum_{i=1}^n m_i^{t,z}\right)$} \alglinelabel{line:bi_diana:v}
  \STATE $h^{t+1} = h^t + \beta \frac{1}{n} \sum_{i=1}^n m_i^{t,z}$ 
  \ENDFOR
  \end{algorithmic}
\end{algorithm}

\begin{algorithm}[ht]
  \caption{CalculateLearningRates}
  \label{alg:learning_rates}
  \footnotesize
  \begin{algorithmic}[1]
    \STATE \textbf{Parameters:} element $\Gamma_t$; parameter $\bar{L} > 0$; strong-convexity parameter $\mu \geq 0,$ probability $p,$ contraction parameter $\alpha$, momentum $\tau$, parameter $\beta$
    \STATE Find the largest root $\bar{\theta}_{t+1}$ of the quadratic equation $$p \bar{L} \Gamma_t \bar{\theta}^2_{t+1} + p (\bar{L} + \Gamma_t \mu) \bar{\theta}_{t+1} - (\bar{L} + \Gamma_t \mu) = 0$$ \alglinelabel{line:learning_rates:root}
    \STATE $\theta_{\min} = \frac{1}{4}\min\left\{1, \frac{\alpha}{p}, \frac{\tau}{p}, \frac{\beta}{p}\right\}$; \quad $\theta_{t+1} = \min\{\bar{\theta}_{t+1}, \theta_{\min}\}$; \quad$\gamma_{t+1} = \frac{p \theta_{t+1} \Gamma_t}{1 - p \theta_{t+1}}$; \quad$\Gamma_{t+1} = \Gamma_{t} + \gamma_{t+1}$
  \end{algorithmic}
\end{algorithm}

\section{New Method: 2Direction}

In order to provide an answer to \ref{word:problem}, at the beginning of our research journey we hoped that a rather straightforward approach might bear fruit. In particular, we considered the current state-the-art methods \algname{ADIANA} (Algorithm~\ref{alg:adiana})~\citep{ADIANA}, \algname{CANITA} \citep{li2021canita} and \algname{EF21-P + DIANA} (Algorithm~\ref{algorithm:diana_ef21_p}) \citep{gruntkowska2022ef21}, and tried to combine the \algname{EF21-P} compression mechanism on the server side with the \algname{ADIANA} (accelerated \algname{DIANA}) compression mechanism on the workers' side. In short, we were aiming to develop a ``\algname{EF21-P + ADIANA}'' method.
Note that while \algname{EF21-P + DIANA}  provides the current SOTA communication complexity among all methods supporting bidirectional compression, the method is not ``accelerated''. On the other hand, while \algname{ADIANA} (in the strongly convex regime) and \algname{CANITA}  (in the convex regime) are ``accelerated'', they support unidirectional (uplink) compression only.

In Sections~\ref{sec:original_adiana} and \ref{sec:original_ef21_p} we list the \algname{ADIANA} and \algname{EF21-P + DIANA} methods, respectively. One can see that in order to calculate $x^{t+1},$ $y^{t+1},$ and $z^{t+1}$ in \algname{ADIANA}  (Algorithm~\ref{alg:adiana}), it is sufficient for the server to broadcast the point $u^{t+1}$. At first sight, it seems  that we might be able to develop a ``\algname{EF21-P + ADIANA}'' method  by replacing Line~\ref{line:adiana:broadcast} in Algorithm~\ref{alg:adiana} with Lines~\ref{line:adiana:ef21_p_1}, \ref{line:adiana:ef21_p_2}, \ref{line:adiana:ef21_p_3}, and \ref{line:adiana:ef21_p_4} from Algorithm~\ref{algorithm:diana_ef21_p}. With these changes, we can try to calculate $x^{t+1}$ and $y^{t+1}$ using the  formulas
\begin{align}
  y^{t+1} &= \theta_{t+1} w^t + (1 - \theta_{t+1}) z^t, \label{eq:new_y}\\
 x^{t+1}  &= \theta_{t+1} w^{t+1} + (1 - \theta_{t+1}) z^t \label{eq:new_x}\\
 w^{t+1}  &= w^t + \cC^{P}\left(u^{t+1} - w^t\right) \label{eq:new_w},
\end{align}
instead of Lines~\ref{line:adiana:y} and \ref{line:adiana:x} in Algorithm~\ref{alg:adiana}. Unfortunately, all our attempts of making this work failed, and we now believe that this ``naive'' approach will not lead to a resolution of \ref{word:problem}. Let us briefly explain why we think so, and how we ultimately managed to resolve \ref{word:problem}.

$\blacklozenge$ The first issue arises from the fact that the point $x^{t+1},$ and, consequently, the point $z^{t+1}$, depend on $w^{t+1}$ instead of $u^{t+1},$ and thus the error from the {\em primal} (i.e., server) compressor $\cC^{P}$ affects them. In our proofs, we do not know how to prove a good convergence rate with \eqref{eq:new_x}. Therefore, we decided to use the original update (Line~\ref{line:adiana:x} from Algorithm~\ref{alg:adiana}) instead. We can do this almost for free because in Algorithm~\ref{alg:adiana} the point $x^{t+1}$ is only used in Line~\ref{line:adiana:z} (Algorithm~\ref{alg:adiana})  with small probability $p$. In the final version of our algorithm  \algname{2Direction} (see Algorithm~\ref{alg:bi_diana}), we broadcast a non-compressed messages $x^{t+1}$ with  probability $p.$ In Section~\ref{sec:sanity_check}, we show that  $p$ is so small that these non-compressed rare messages do not affect the total communication complexity of Algorithm~\ref{alg:adiana}. 

$\blacklozenge$ The second issue comes from the observation that we can not perform the same trick for point $y^{t+1}$ since it is required in each iteration of Algorithm~\ref{alg:adiana}. We tried to use \eqref{eq:new_y} and \eqref{eq:new_w}, but this still does not work. Deeper understanding of this can only be gained by a detailed examination our proof and the proofs of \citep{ADIANA,gruntkowska2022ef21}. One  way to explain the difficulty is to observe that in non-accelerated methods \citep{gorbunov2020unified,gruntkowska2022ef21}, the variance-reducing shifts $h^t$ converge to \emph{the fixed vector} $\nabla f(x^*),$ while in the accelerated methods \citep{ADIANA, li2021canita}, these shifts $h^t$ converge to \emph{the non-fixed vectors} $\nabla f(z^t)$ in the corresponding Lyapunov functions. Assume that $\mu = 0.$ Then, instead of the \algname{EF21-P} mechanism
\begin{eqnarray}
\begin{aligned}
  \label{eq:ef21_p}
  w^{t+1} &\squeeze = w^t + \cC^{P}\left(u^{t+1} - w^t\right) \overset{\textnormal{\tiny Line~\ref{line:adiana:u} in Alg.~\ref{alg:adiana}}}{=} w^t + \cC^{P}\left(u^{t} - \frac{\gamma_{t+1}}{\bar{L}} g^{t+1} - w^t\right) \\
  &\squeeze \hspace{-0.5cm}\overset{\nabla f(x^*) = 0}{=} w^t - \frac{\gamma_{t+1}}{\bar{L}} {\red \nabla f(x^*)} + \cC^{P}\left(u^{t} - \frac{\gamma_{t+1}}{\bar{L}} (g^{t+1} - {\red \nabla f(x^*)}) - w^t\right),
\end{aligned}
\end{eqnarray}
we propose to perform the step
\begin{align}
  w^{t+1} &\squeeze= w^t - \frac{\gamma_{t+1}}{\bar{L}} {\green \nabla f(z^t)} + \cC^{P}\left(u^{t} - \frac{\gamma_{t+1}}{\bar{L}} (g^{t+1} - {\green \nabla f(z^t)}) - w^t\right)
  \label{eq:unitracktable_w} \\
  &\squeeze = q^{t+1} + \cC^{P}\left(u^{t+1} - q^{t+1}\right) \nonumber, \textnormal{ where}\\
   u^{t+1} &\squeeze= \argmin \limits_{x \in \R^d} \inp{g^{t+1}}{x} + \frac{\bar{L}}{2 \gamma_{t+1}} \norm{x - u^t}^2, \; q^{t+1}   = \argmin \limits_{x \in \R^d} \inp{\nabla f(z^t)}{x} + \frac{\bar{L}}{2 \gamma_{t+1}} \norm{x - w^t}^2. \nonumber
\end{align}
Unlike \eqref{eq:ef21_p},  step \eqref{eq:unitracktable_w} resolves all our previous problems, and we were able to obtain new SOTA  rates.

$\blacklozenge$ However, step \eqref{eq:unitracktable_w} is not implementable since the server and the nodes need to know the vector $\nabla f(z^t).$ The last crucial observation is the same as with the points $x^t$ and $z^t$: the vector $\nabla f(z^t)$ changes with probability $p$ since the point $z^t$ changes with probability $p.$ Intuitively, this means that it easier to communicate $\nabla f(z^t)$ between the server and the workers. We do this using two auxiliary control vectors, $v^t$ and $k^t.$ The former ``learns'' the value of $\nabla f(z^t)$ in Line~\ref{line:bi_diana:v} (Algorithm~\ref{alg:bi_diana}), and the latter is used in Line~\ref{line:bi_diana:w} (Algorithm~\ref{alg:bi_diana}) instead of $\nabla f(z^t).$ Then, when the algorithm updates $z^{t+1},$ it also updates $k^t$ in Line~\ref{line:bi_diana:k} (Algorithm~\ref{alg:bi_diana}) and the updated non-compressed vector $k^t$ is broadcast to the workers.

The described changes are highlighted in Lines~\ref{line:bi_diana:y}, \ref{line:bi_diana:q}, \ref{line:bi_diana:p}, \ref{line:bi_diana:w}, \ref{line:bi_diana:k}, \ref{line:bi_diana:broadcast}, \ref{line:bi_diana:q_local}, \ref{line:bi_diana:w_local} and \ref{line:bi_diana:v} of our new Algorithm~\ref{alg:bi_diana} in {\green green} color. Other steps of the algorithm correspond to the original \algname{ADIANA} method (Algorithm~\ref{alg:adiana}). Remarkably, all these new steps are only required to substitute a single Line~\ref{line:adiana:broadcast} of Algorithm~\ref{alg:adiana}!

\section{Theoretical Communication Complexity of 2Direction}

Having outlined our thought process when developing  \algname{2Direction} (Algorithm~\ref{alg:bi_diana}), we are now ready to present our theoretical iteration and communication complexity results. Note that \algname{2Direction} depends on two hyper-parameters, probability $p$ (used in Lines~4 and 17) and momentum $\tau$ (used in Lines~4 and 30).
Further, while \citet{ADIANA, li2021canita} assume a strong relationship between  $L$ and $L_{\max}$ ($L=L_{\max}$),  \citet{gruntkowska2022ef21} differentiate between $L$ and $L_{\max}$, and thus perform a more general and analysis of their method. In order to perform a fair comparison to the above results, we have decided to minimize the {\em total communication complexity}   $\mathfrak{m}^{r}_{{\scriptscriptstyle \cM}}$ as a function of the hyper-parameters $p$ and $\tau$, depending on whether the ratio  $\nicefrac{L_{\max}}{L}$ is known or not.

Defining $R^2 \eqdef \norm{x^0 - x^*}^2$, Theorems~\ref{theorem:main_theorem_strongly} and \ref{theorem:main_theorem_general_convex}  state that \algname{2Direction} converges after
\begin{eqnarray}
\begin{aligned}
  \label{eq:strong_complexity}
 \squeeze T \eqdef \widetilde{\Theta}\Bigg(\max\Bigg\{&\squeeze\sqrt{\frac{L}{\alpha p \mu}},\sqrt{\frac{L}{\alpha \tau \mu}},\sqrt{\frac{\sqrt{L L_{\max}} (\omega + 1) \sqrt{\omega \tau}}{\alpha \sqrt{n} \mu}},\sqrt{\frac{\sqrt{L L_{\max}} \sqrt{\omega + 1} \sqrt{\omega \tau}}{\alpha \sqrt{p} \sqrt{n} \mu}}, \\
  &\squeeze\sqrt{\frac{L_{\max} \omega (\omega + 1)^2 p}{n \mu}}, \sqrt{\frac{L_{\max} \omega}{n p \mu}}, \frac{1}{\alpha}, \frac{1}{\tau}, (\omega + 1), \frac{1}{p}\Bigg\}\Bigg), \text{or}
\end{aligned}
\end{eqnarray}
\begin{eqnarray}
\begin{aligned}
  \label{eq:general_convex}
\squeeze  T_{\textnormal{gc}} \eqdef \Theta\Bigg(\max\Bigg\{&\squeeze\sqrt{\frac{L R^2}{\alpha p \varepsilon}},\sqrt{\frac{L R^2}{\alpha \tau \varepsilon}},\sqrt{\frac{\sqrt{L L_{\max}} (\omega + 1) \sqrt{\omega \tau} R^2}{\alpha \sqrt{n} \varepsilon}},\sqrt{\frac{\sqrt{L L_{\max}} \sqrt{\omega + 1} \sqrt{\omega \tau} R^2}{\alpha \sqrt{p} \sqrt{n} \varepsilon}}, \\
  &\squeeze\sqrt{\frac{L_{\max} \omega (\omega + 1)^2 p R^2}{n \varepsilon}}, \sqrt{\frac{L_{\max} \omega R^2}{n p \varepsilon}}\Bigg\}\Bigg) + \widetilde{\Theta}\Bigg(\max\Bigg\{\frac{1}{\alpha}, \frac{1}{\tau}, (\omega + 1), \frac{1}{p}\Bigg\}\Bigg)
\end{aligned}
\end{eqnarray}
iterations, in the $\mu$--strongly convex and general convex regimes, respectively. These complexities depend on two hyper-parameters: $p \in (0, 1]$ and $\tau \in (0, 1].$ For simplicity, in what follows we consider the $\mu$--strongly convex case only\footnote{One can always take $\mu = \nicefrac{\varepsilon}{R^2}$ to understand the dependencies in the general convex case.}.
While the iteration complexities \eqref{eq:strong_complexity} and \eqref{eq:general_convex} are clearly important, in the context of our paper, optimizing communication complexity \eqref{eq:main_complexity} is more important. In the following simple theorem, we give expressions for the communication complexities of \algname{2Direction}, taking into account both workers-to-server (w2s) and server-to-workers (s2w) communication.

\begin{theorem}
  \label{theorem:comm_complexity}
  Assume that $\cC_i^{D, \cdot}$ and $\cC^{P}$ have expected densities equal to $K_{\omega}$ and $K_{\alpha}$, respectively (see Definition~\ref{def:expected_density}). In view of Theorem~\ref{theorem:main_theorem_strongly}, in expectation, the w2s and s2w communication complexities are equal to \begin{align}
    \label{eq:abstract_comm_complexity}
    \mathfrak{m}^{{\scriptscriptstyle \textnormal{w2s}}}_{{\scriptscriptstyle \textnormal{new}}} = \widetilde{\Theta}\left(K_{\omega} \times T + d\right) \quad\textnormal{ and } \quad \mathfrak{m}^{{\scriptscriptstyle \textnormal{s2w}}}_{{\scriptscriptstyle \textnormal{new}}} = \widetilde{\Theta}\left(\left(K_{\alpha} + p d\right) \times T + d\right).
  \end{align}
\end{theorem}
\begin{proof}
  The first complexity in \eqref{eq:abstract_comm_complexity} follows because w2s communication involves $\cC_i^{D, y}(\cdot)$ and $\cC_i^{D, z}(\cdot)$ only. The second complexity in \eqref{eq:abstract_comm_complexity} follows because s2w communication involves $\cC_i^{P}(\cdot)$, plus two non-compressed vectors $x^{t+1}$ and $k^{t+1}$ with the probability $p.$ The term $d$ comes from the fact that non-compressed vectors are communicated in the initialization phase.
\end{proof}

\subsection{The ratio $\nicefrac{L_{\max}}{L}$ is not known}

In the following theorem, we consider the regime when the exact value of $\nicefrac{L_{\max}}{L}$ is not known. Hence, we seek to find $p$ and $\tau$ that minimize the worst case $\mathfrak{m}^{r}_{{\scriptscriptstyle \textnormal{new}}}$  (see \eqref{eq:main_complexity}) w.r.t.\ $L_{\max} \in [L, n L].$
\begin{restatable}{theorem}{COROLLARYMAINREALISTIC}
  \label{cor:realistic}
Choose  $r \in [0, 1]$ and let $\mu^r_{\omega, \alpha} \eqdef \frac{rd}{(1 - r) K_{\omega} + r K_{\alpha}}.$ In view of Theorem~\ref{theorem:comm_complexity},
the values $
      p = \min\left\{\frac{1}{\omega + 1}, \frac{1}{\mu^r_{\omega, \alpha}}\right\}$
      and
      $\tau = \frac{p^{1/3}}{(\omega + 1)^{2/3}}$
    minimize $\max\limits_{L_{\max} \in [L, n L]} \mathfrak{m}^{r}_{{\scriptscriptstyle \textnormal{new}}}.$ This choice leads to the following number of communication rounds:
    {\small \begin{align}
      \label{eq:realistic_compl}
 \squeeze     T^{\textnormal{realistic}} &\squeeze \eqdef \widetilde{\Theta}\Bigg(\max\Bigg\{\sqrt{\frac{L \max\{\omega + 1, \mu^r_{\omega, \alpha}\}}{\alpha \mu}},\sqrt{\frac{L_{\max} \omega \max\{\omega + 1, \mu^r_{\omega, \alpha}\}}{n \mu}}, \frac{1}{\alpha}, (\omega + 1), \mu^r_{\omega, \alpha}\Bigg\}\Bigg).
    \end{align}}The total communication complexity thus equals $\mathfrak{m}^{r}_{{\scriptscriptstyle \textnormal{realistic}}} = \widetilde{\Theta}\left(\left((1 - r)K_{\omega} + r K_{\alpha}\right) T_{{\scriptscriptstyle \textnormal{realistic}}} + d\right).$ 
\end{restatable}

\begin{remark}
  \label{rem:realistic}
  To simplify the rate \eqref{eq:realistic_compl} and understand the quantity $\mu^r_{\omega, \alpha}$, let  $\cC_i^{D, \cdot}$  be the Rand$K$ sparsifier\footnote{It is sufficient to assume that $\omega + 1 = \Theta\left(\nicefrac{d}{K_{\omega}}\right)$.} and consider the case when the s2w communication is not slower than the w2s communication, i.e., $r \leq \nicefrac{1}{2}.$ Then
    $
      T^{\textnormal{realistic}} = \widetilde{\Theta}\left(\max\left\{\sqrt{\frac{L (\omega + 1)}{\alpha \mu}},\sqrt{\frac{L_{\max} \omega (\omega + 1)}{n p \mu}}, \frac{1}{\alpha}, (\omega + 1)\right\}\right)$ and $\mu^r_{\omega, \alpha} \leq \omega + 1.
    $  Indeed, this follows from $r \leq \nicefrac{1}{2}$ and the fact that $\omega + 1 = \nicefrac{d}{K_{\omega}}$ for the Rand$K$ compressor:
  $\mu^r_{\omega, \alpha} \eqdef \frac{rd}{(1 - r) K_{\omega} + r K_{\alpha}} \leq \frac{r}{1 - r} \times \frac{d}{K_{\omega}} \leq \omega + 1.$
\end{remark}

\subsection{The ratio $\nicefrac{L_{\max}}{L}$ is known}
We now consider the case when we have information about the ratio $\nicefrac{L_{\max}}{L}.$

\begin{restatable}{theorem}{COROLLARYMAINSTRONG}
  \label{cor:optimistic}
Choose $r \in [0, 1],$ and let $\mu^r_{\omega, \alpha} \eqdef \frac{rd}{(1 - r) K_{\omega} + r K_{\alpha}}.$ In view of Theorem~\ref{theorem:comm_complexity}, the values $p$ and $\tau$ given by \eqref{eq:opt_p_know} and \eqref{eq:opt_tau_know}, respectively,
    minimize $\mathfrak{m}^{r}_{{\scriptscriptstyle \textnormal{new}}}$ from \eqref{eq:strong_complexity}. This choice leads to the following number of communication rounds:
    {\small 
    \begin{eqnarray}
    \begin{aligned}
      \label{eq:optimistic_compl}
      T^{\textnormal{optimistic}} = \widetilde{\Theta}\Bigg(\max\Bigg\{&\sqrt{\frac{L \max\{1, \mu^r_{\omega, \alpha}\}}{\alpha \mu}}, \sqrt{\frac{L^{2/3} L_{\max}^{1/3} (\omega + 1)}{\alpha n^{1/3} \mu}}, \sqrt{\frac{L^{1/2} L_{\max}^{1/2} (\omega + 1)^{3/2}}{\sqrt{\alpha n} \mu}}, \\
      &\sqrt{\frac{L_{\max} \omega \max\{\omega + 1, \mu^r_{\omega, \alpha}\}}{n \mu}}, \frac{1}{\alpha}, (\omega + 1), \mu^r_{\omega, \alpha}\Bigg\}\Bigg).
    \end{aligned}
    \end{eqnarray}}
    The total communication complexity thus equals $\squeeze\mathfrak{m}^{r}_{{\scriptscriptstyle \textnormal{optimistic}}} = \widetilde{\Theta}\left(\left((1 - r)K_{\omega} + r K_{\alpha}\right) T_{{\scriptscriptstyle \textnormal{optimistic}}} + d\right).$ 
\end{restatable}

Note that information about $\nicefrac{L_{\max}}{L}$ leads to a  better rate that in Theorem~\ref{cor:realistic}.

\begin{remark}
  \label{rem:optimistic}
  To simplify the rate \eqref{eq:optimistic_compl}, let $\cC_i^{D, \cdot}$ be the Rand$K$ sparsifier\footnote{It is sufficient to assume that $\omega + 1 = \Theta\left(\nicefrac{d}{K_{\omega}}\right)$.} and consider the case when the s2w communication is not slower than the w2s communication, i.e., $r \leq \nicefrac{1}{2}.$ Then
  {\small$T^{\textnormal{optimistic}} = \widetilde{\Theta}\left(\max\left\{\sqrt{\frac{L \max\{1, r \left(\omega + 1\right)\}}{\alpha \mu}}, \sqrt{\frac{L^{2/3} L_{\max}^{1/3} (\omega + 1)}{\alpha n^{1/3} \mu}}, \sqrt{\frac{L^{1/2} L_{\max}^{1/2} (\omega + 1)^{3/2}}{\sqrt{\alpha n} \mu}}, \sqrt{\frac{L_{\max} \omega (\omega + 1)}{n \mu}}, \frac{1}{\alpha}, (\omega + 1)\right\}\right).$}
Indeed, this follows from $r \leq \nicefrac{1}{2}$ and the fact that $\omega + 1 = \nicefrac{d}{K_{\omega}}$ for the Rand$K$ compressor:
    $\mu^r_{\omega, \alpha} \eqdef \frac{rd}{(1 - r) K_{\omega} + r K_{\alpha}} \leq \frac{r}{1 - r} \times \frac{d}{K_{\omega}} \leq 2 r \left(\omega + 1\right)$.
    
\end{remark}

\section{Theoretical Comparison with Previous  State of the Art}
\label{sec:sanity_check}
We now show that the communication complexity of \algname{2Direction} is always {\em no worse}  than that of \algname{EF21 + DIANA} and \algname{AGD}. Crucially, in some regimes, it can be substantially better. Furthermore, we show that if the s2w communication cost is zero (i.e., if $r=0$), the \algname{2Direction} obtains the same communication complexity as  \algname{ADIANA} \citep{ADIANA} (see Section~\ref{sec:comp_adiana}).

{\bf Comparison with \algname{EF21 + DIANA}.}
The \algname{EF21-P + DIANA} method has the communication complexities that equal $$\mathfrak{m}^{{\scriptscriptstyle \textnormal{w2s}}}_{{\scriptscriptstyle \textnormal{EF21-P + DIANA}}} = \widetilde{\Theta}\left(K_{\omega} \times T^{\scriptscriptstyle \textnormal{EF21-P + DIANA}} + d\right) \;\; \textnormal{ and } \;\; \mathfrak{m}^{{\scriptscriptstyle \textnormal{s2w}}}_{{\scriptscriptstyle \textnormal{EF21-P + DIANA}}} = \widetilde{\Theta}\left(K_{\alpha} \times T^{\scriptscriptstyle \textnormal{EF21-P + DIANA}} + d\right),$$
where $K_{\omega}$ and $K_{\alpha}$ are the expected densities of $\cC_i^{D}$ and $\cC^{P}$ in Algorithm~\ref{algorithm:diana_ef21_p}. The last term $d$ comes from the fact that \algname{EF21-P + DIANA} sends non-compressed vectors in the initialization phase. Let us define $K_{\omega, \alpha}^r~\eqdef~(1 - r)K_{\omega} + r K_{\alpha}.$ Therefore, the total communication complexity equals
 \begin{align}
  \label{eq:ef21_p_diana}
  \squeeze
  \mathfrak{m}^{r}_{{\scriptscriptstyle \textnormal{EF21-P + DIANA}}} = \widetilde{\Theta}\left(K_{\omega, \alpha}^r \left(\frac{L}{\alpha \mu} + \frac{\omega L_{\max}}{n \mu} + \omega\right) + d\right).
\end{align}
Theorem~\ref{cor:realistic} ensures that the total communication complexity of \algname{2Direction} is
{\small \begin{align}
  \label{eq:acc_diana_total}
  \mathfrak{m}^{r}_{{\scriptscriptstyle \textnormal{realistic}}} = \widetilde{\Theta}\Bigg(K_{\omega, \alpha}^r \Bigg(\sqrt{\frac{L \max\{\omega + 1, \mu^r_{\omega, \alpha}\}}{\alpha \mu}} + \sqrt{\frac{L_{\max} \omega \max\{\omega + 1, \mu^r_{\omega, \alpha}\}}{n \mu}} + \frac{1}{\alpha} + \omega + \mu^r_{\omega, \alpha}\Bigg) + d\Bigg).
\end{align}}

 One can see that \eqref{eq:acc_diana_total} is an accelerated rate; it has much better dependence on the condition numbers $\nicefrac{L}{\mu}$ and $\nicefrac{L_{\max}}{\mu}.$ In Section~\ref{sec:compare_ef21}, we prove the following simple theorem, which means  that \algname{2Direction} is not worse than \algname{EF21-P + DIANA}.
\begin{restatable}{theorem}{THEOREMBETTEREF}
  \label{theorem:better_ef21}
  For all $r \in [0, 1],$ $\mathfrak{m}^{r}_{{\scriptscriptstyle \textnormal{realistic}}} = \widetilde{\cO}\left(\mathfrak{m}^{r}_{{\scriptscriptstyle \textnormal{EF21-P + DIANA}}}\right).$
\end{restatable}

{\bf Comparison with \algname{AGD}.}
To compare the abstract complexity \eqref{eq:acc_diana_total} with the non-abstract complexity $\widetilde{\Theta}(d \sqrt{\nicefrac{L}{\mu}}),$ we take the Rand$K$ and Top$K$ compressors in Algorithm~\ref{alg:bi_diana}. 

\begin{restatable}{theorem}{THEOREMBETTERAGD}
  \label{theorem:better_agd}
  For all $r \in [0, 1]$ and for all $K \in [d],$ let us take the Rand$K$ and Top$K$ compressors with the parameters (expected densities) i) $K_{\omega} = K$ and $K_{\alpha} = \min\{\lceil\nicefrac{1 - r}{r} K\rceil, d\}$ for $r \in [0, \nicefrac{1}{2}],$ ii) $K_{\omega} = \min\{\lceil\nicefrac{r}{1 - r} K \rceil, d\}$ and $K_{\alpha} = K$ for $r \in (\nicefrac{1}{2}, 1].$  Then we have $\mathfrak{m}^{r}_{{\scriptscriptstyle \textnormal{realistic}}} = \widetilde{\cO}\left(\mathfrak{m}_{{\scriptscriptstyle \textnormal{AGD}}}\right).$
\end{restatable}

Theorem~\ref{theorem:better_agd} states that the total communication complexity of \algname{2Direction} is not worse than that of \algname{AGD}. It can be strictly better in the regimes when $\alpha > \nicefrac{K}{d}$ (for Top$K$), $L_{\max} < n L,$ and $n > 1.$

\section{Proof Sketch}
After we settle on the final version of Algorithm~\ref{alg:bi_diana}, the proof technique is pretty standard at the beginning \citep{ADIANA,tyurin2022dasha,gruntkowska2022ef21}. We proved a descent lemma (Lemma~\ref{lemma:aux_first}) and lemmas that control the convergences of auxiliary sequences (Lemmas~\ref{lemma:aux_control}, \ref{lemma:bi_fast_diana_iterates}, \ref{lemma:bi_fast_diana_k}, \ref{lemma:bi_fast_diana_v}). Using these lemmas, we construct the Lyapunov function \eqref{eq:main_theorem} with the coefficients $\kappa,$ $\rho,$ $\lambda$ and $\nu_{t} \geq 0$ for all $t \geq 0.$

One of the main problems was to find appropriate $\bar{L},$ $\kappa,$ $\rho,$ $\lambda$ and $\nu_{t} \geq 0$ such that we get a convergence. In more details, to get a converge it is sufficient to find $\bar{L},$ $\kappa,$ $\rho,$ $\lambda$ and $\nu_{t} \geq 0$ such that \eqref{eq:parameters_task}, \eqref{eq:lemma:negative_parameters:bregman} and \eqref{eq:lemma:negative_parameters:norm} hold. Using the symbolic computation SymPy library \citep{sympy}, we found appropriate $\kappa,$ $\rho,$ $\lambda$ and $\nu_{t} \geq 0$ (\eqref{eq:rho_sol}, \eqref{eq:kappa_sol}, \eqref{eq:lambda_sol}, \eqref{eq:nu_sol}) such that the inequalities hold. But that is not all. To get a convergence, we also found the bounds on the parameter $\bar{L}$, which essentially describe the speed of the convergence. In raw form, using symbolic computations, we obtained a huge number of bounds on $\bar{L}$ (see Sections~\ref{sec:sym_comp_kapap_const}, \ref{sec:sym_comp_rho_const}, \ref{sec:sym_comp_bregman_const}, \ref{sec:sym_comp_dist_const}). It was clear that most of them are redundant and it was required to find the essential ones. After a close look at the bounds on $\bar{L}$, we found out that it is sufficient to require \eqref{eq:constr_lipts} to ensure that all inequalities from Sections~\ref{sec:sym_comp_kapap_const}, \ref{sec:sym_comp_rho_const}, \ref{sec:sym_comp_bregman_const} and \ref{sec:sym_comp_dist_const} hold (see Section~\ref{sec:sym_comp_check} for details).

\begin{figure}[H]
\centering
\begin{tikzpicture}[scale=0.8, transform shape] 
  \node [draw,
    minimum width=2cm,
    minimum height=1.2cm,
  ]  (maintheorem) {\makecell{Theorem~\ref{theorem:main_theorem} \\ {\footnotesize (Lyapunov Function} \\ {\footnotesize Convergence)}}};

  \node [draw,
    minimum width=3cm, 
    minimum height=1.2cm,
    above right= -0.6cm and 0.5cm of maintheorem
  ] (maintheoremstrongly) {\makecell{Theorem~\ref{theorem:main_theorem_strongly} \\ {\footnotesize ($\mu$--strongly convex)}}};

  \node [draw,
    minimum width=3cm, 
    minimum height=0.6cm,
    below right= -0.6cm and 0.5cm of maintheorem
  ] (maintheoremconvex) {\makecell{Theorem~\ref{theorem:main_theorem_general_convex} \\ {\footnotesize (general convex)}}};

  \node [draw,
    minimum width=2cm, 
    minimum height=0.6cm,
    below right= -0.6cm and 0.5cm of maintheoremstrongly
  ] (commcomplexity) {\makecell{Theorems~\ref{theorem:comm_complexity}, \ref{cor:realistic} \\ {\footnotesize(Find comm. compl.} \\ {\footnotesize and optimal $p$ and $\tau$)}}};

  \node [draw,
    minimum width=3cm, 
    minimum height=1.2cm,
    above right= -0.6cm and 0.5cm of commcomplexity
  ] (betterthanef) {\makecell{Theorem~\ref{theorem:better_ef21} \\ {\footnotesize (No worse than} \\ {\footnotesize \algnamesmall{EF21-P + DIANA})}}};

  \node [draw,
    minimum width=3cm, 
    minimum height=1.2cm,
    below right= -0.6cm and 0.5cm of commcomplexity
  ] (betterthanagd) {\makecell{Theorem~\ref{theorem:better_agd} \\ {\footnotesize (No worse than} {\footnotesize \algnamesmall{AGD})}}};
  
  \draw[-stealth] (maintheorem.east) -- (maintheoremstrongly.west) 
    node[midway,above]{};

  \draw[-stealth] (maintheorem.east) -- (maintheoremconvex.west) 
    node[midway,above]{};
  
  \draw[-stealth] (maintheoremconvex.east) -- (commcomplexity.west)
    node[midway,above]{};

  \draw[-stealth] (maintheoremstrongly.east) -- (commcomplexity.west)
  node[midway,above]{};

  \draw[-stealth] (commcomplexity.east) -- (betterthanef.west)
  node[midway,above]{};

  \draw[-stealth] (commcomplexity.east) -- (betterthanagd.west)
  node[midway,above]{};
\end{tikzpicture}
\caption{Roadmap to our resolution of \ref{word:problem}.}
\label{fig:roadmap}
\end{figure}
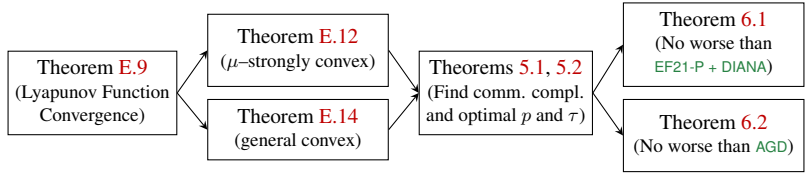

\section{Limitations and Future Work}
In contrast to Algorithm~\ref{algorithm:diana_ef21_p} (\algname{EF21+DIANA}), in which the server always broadcasts compressed vectors, in our Algorithm~\ref{alg:bi_diana} (\algname{2Direction})   the server needs to broadcast non-compressed vectors  with probability $p$ (see Line~\ref{line:bi_diana:broadcast}). While in Section~\ref{sec:sanity_check} we explain that this does not have an adverse effect on the theoretical communication complexity since $p$ is small, one may wonder whether it might be  possible to  achieve the same (or better) bounds as ours without having to resort to intermittent non-compressed broadcasts. This remains an open problem; possibly a challenging one. Another limitation comes from the fact that \algname{2Direction} requires more iterations than \algname{AGD} in general (this is the case of all methods that reduce communication complexity). While, indeed, \eqref{eq:strong_complexity} can be higher than $\widetilde{\Theta}(\sqrt{\nicefrac{L}{\mu}})$, the total communication complexity of \algname{2Direction} is not worse than that of \algname{AGD}.

\subsection*{Acknowledgements}
This work of P. Richt\'{a}rik and A. Tyurin was supported by the KAUST Baseline Research Scheme (KAUST BRF) and the KAUST Extreme Computing Research Center (KAUST ECRC), and the work of P. Richt\'{a}rik was supported by the SDAIA-KAUST Center of Excellence in Data Science and Artificial Intelligence (SDAIA-KAUST AI).

\bibliographystyle{apalike}
\bibliography{paper}

\clearpage
\appendix

\part*{Appendix}

\tableofcontents
\newpage

\section{Table of Notations}

\newcommand{\ditto}[1][.4pt]{\xrfill{#1}~''~\xrfill{#1}}
\begin{table}[h]
\centering
\begin{tabular}{cc}
\hline
Notation & Meaning\\
\hline
$g = \cO(f)$ & Exist $C > 0$ such that $g(z) \leq C \times f(z)$ for all $z \in \mathcal{Z}$\\
$g = \Omega(f)$ & Exist $C > 0$ such that $g(z) \geq C \times f(z)$ for all $z \in \mathcal{Z}$\\
$g = \Theta(f)$ & $g = \cO(f)$ and $g = \Omega(f)$ \\
$g = \widetilde{\cO}(f)$ & Exist $C > 0$ such that $g(z) \leq C \times f(z) \times \log (\textnormal{poly}(z))$ for all $z \in \mathcal{Z}$ \\
$g = \widetilde{\Omega}(f)$ & Exist $C > 0$ such that $g(z) \geq C \times f(z) \times \log (\textnormal{poly}(z))$ for all $z \in \mathcal{Z}$ \\
$g = \widetilde{\Theta}(f)$ & $g = \widetilde{\cO}(f)$ and $g = \widetilde{\Omega}(f)$ \\
$\{a, \dots, b\}$ & Set $\{i \in \mathbb{Z}\,|\, a \leq i \leq b\}$ \\
$[n]$ & $\{1, \dots, n\}$ \\
\hline
\end{tabular}
\end{table}

\section{The Original \algname{ADIANA} Algorithm}

\label{sec:original_adiana}

In this section, we present the \algname{ADIANA} algorithm from \citep{ADIANA}. In the following method, the notations, parameterization, and order of steps can be slightly different, but the general idea is the same.

\begin{algorithm}[ht]
  \caption{Accelerated \algname{DIANA} (\algname{ADIANA}) by \citet{ADIANA}}
  \label{alg:adiana}
  \footnotesize
  \begin{algorithmic}[1]
  \STATE \textbf{Parameters:} Lipschitz-like parameter $\bar{L} > 0$, strong-convexity parameter $\mu \geq 0,$ probability $p,$ parameter $\Gamma_0$, initial point $x^0 \in \R^d,$ initial gradient shifts $h^0_1, \dots , h^0_n\in \R^d$
  \STATE Initialize $\beta = \nicefrac{1}{\omega + 1}$ and $z^0 = u^0 = x^0$
  \STATE Initialize $h^0 = \frac{1}{n}\sum_{i = 1}^n h_i^0$
  \FOR{$t = 0, 1, \dots, T - 1$} 
  \STATE $\Gamma_{t+1}, \gamma_{t+1}, \theta_{t+1}= \textnormal{CalculateOriginalLearningRates}(\dots)$
  \FOR{$i = 1, \dots, n$ in parallel}
  \STATE $y^{t+1} = \theta_{t+1} {u^t} + (1 - \theta_{t+1}) z^t$ \alglinelabeladiana{line:adiana:y}
  \STATE $m_i^{t,y} = \cC_i^{D, y}(\nabla f_i(y^{t+1}) - h_i^t)$  \hfill{\scriptsize \color{gray} Worker $i$ compresses the shifted gradient via the compressor $\cC_i^{D, y} \in \mathbb{U}(\omega)$}
  \STATE {Send compressed message $m_i^{t,y}$ to the server}
  \ENDFOR
  \STATE $m^{t,y} = \frac{1}{n}\sum_{i = 1}^n m_i^{t,y}$ \hfill{\scriptsize \color{gray}Server averages the messages}
  \STATE $g^{t+1} = h^t + m^{t,y} \equiv \frac{1}{n}\sum_{i=1}^n \left(h_i^t + m_i^{t,y}\right)$ 
  \STATE $u^{t+1} = \argmin_{x \in \R^d} \inp{g^{t+1}}{x} + \frac{\bar{L} + \Gamma_{t} \mu}{2 \gamma_{t+1}} \norm{x - u^t}^2 + \frac{\mu}{2} \norm{x - y^{t+1}}^2$ \hfill{\scriptsize \color{gray} Server does a gradient-like descent step} \alglinelabeladiana{line:adiana:u}
  \STATE Flip a coin $c^t \sim \textnormal{Bernoulli}(p)$
  \STATE {\red Broadcast non-compressed messages $u^{t+1}$ to all $n$ workers} \alglinelabeladiana{line:adiana:broadcast}
  \FOR{$i = 1, \dots, n$ in parallel}
  \STATE $x^{t+1} = \theta_{t+1} u^{t+1} + (1 - \theta_{t+1}) z^t$ \alglinelabeladiana{line:adiana:x}
  \STATE $z^{t+1} = \begin{cases}
    x^{t+1}, &c^t = 1 \\
    z^{t}, &c^t = 0
  \end{cases}$ \alglinelabeladiana{line:adiana:z}
  \STATE $m_i^{t,z} = \cC_i^{D, z}(\nabla f_i(z^{t+1}) - h_i^t)$  \hfill{\scriptsize \color{gray} Worker $i$ compresses the shifted gradient via the compressor $\cC_i^{D, z} \in \mathbb{U}(\omega)$}
  \STATE $h_i^{t+1} = h_i^t + \beta m_i^{t,z}$ \hfill{\scriptsize \color{gray} Update the control variables}
  \STATE {Send compressed message $m_i^{t,z}$ to the server}
  \ENDFOR
  \STATE $h^{t+1} = h^t + \beta \frac{1}{n} \sum_{i=1}^n m_i^{t,z}$ \hfill{\scriptsize \color{gray}Server averages the messages}
  \ENDFOR
  \end{algorithmic}
\end{algorithm}

\newpage
\section{The Original \algname{EF21-P + DIANA} Algorithm}

\label{sec:original_ef21_p}

In this section, we present the \algname{EF21-P + DIANA} algorithm from \citep{gruntkowska2022ef21}. In the following method, the notations, parameterization, and order of steps can be slightly different, but the general idea is the same.

\begin{algorithm*}[h]
  \caption{\algname{EF21-P + DIANA} by \citet{gruntkowska2022ef21}}
  \footnotesize
  \begin{algorithmic}[1]
  \label{algorithm:diana_ef21_p}
  \STATE \textbf{Parameters:} learning rates $\gamma > 0$ and $\beta > 0,$ initial model $u^0 \in \R^d,$ initial gradient shifts $h^0_1, \dots , h^0_n\in \R^d,$ average of the initial gradient shifts $h^0 = \frac{1}{n}\sum_{i = 1}^n h_i^0,$ initial model shift $w^0 = u^0\in \R^d$
  \FOR{$t = 0, 1, \dots, T - 1$} 
  \FOR{$i = 1, \dots, n$ in parallel} 
  \STATE $m_i^t = \cC_i^{D}(\nabla f_i(w^t) - h_i^t)$  \hfill{\scriptsize \color{gray} Worker $i$ compresses the shifted gradient via the dual compressor $\cC_i^{D} \in \mathbb{U}(\omega)$}
  \STATE {Send compressed message $m_i^t$ to the server}
  \STATE $h^{t+1}_i = h^t_i + \beta m^t_i$ \hfill{\scriptsize \color{gray} Worker $i$ updates its local gradient shift with stepsize $\beta$}
  \ENDFOR
  \STATE $m^t = \frac{1}{n}\sum_{i = 1}^n m_i^t$ \hfill{\scriptsize \color{gray}Server averages the $n$ messages received from the workers}
  \STATE $h^{t+1} = h^t + \beta m^t$ \hfill{\scriptsize \color{gray} Server updates the average gradient shift so that $h^t = \frac{1}{n}\sum_{i = 1}^n h_i^t$}
  \STATE $g^{t} = h^t + m^t$ \hfill{\scriptsize \color{gray} Server computes the gradient estimator}
  \STATE $u^{t+1} = u^t - \gamma g^t$ \hfill{\scriptsize \color{gray} Server takes a gradient-type step with stepsize $\gamma$}
  \STATE {\green $p^{t+1} = \cC^{P}\left(u^{t+1} - w^t\right)$  \hfill {\scriptsize \color{gray} Server compresses the shifted model   via the primal compressor $\cC^{P} \in \mathbb{B}\left(\alpha\right)$}} \alglinelabelef{line:adiana:ef21_p_1}
  \STATE {\green $w^{t+1} = w^t + p^{t+1}$\hfill{\scriptsize \color{gray}Server updates the model shift}} \alglinelabelef{line:adiana:ef21_p_2}
  \STATE {\green Broadcast compressed message $p^{t+1}$ to all $n$ workers} \alglinelabelef{line:adiana:ef21_p_3}
  \FOR{$i = 1, \dots, n$ in parallel}
  \STATE {\green $w^{t+1} = w^{t} + p^{t+1}$ \hfill{\scriptsize \color{gray} Worker $i$ updates its local copy of the model shift}} \alglinelabelef{line:adiana:ef21_p_4}
  \ENDFOR
  \ENDFOR
  \end{algorithmic}
\end{algorithm*}

\section{Useful Identities and Inequalities}

For all $x,y,x_1,\ldots,x_n \in \R^d$, $s>0$ and $\alpha\in(0,1]$, we have:
\begin{align}
    \label{eq:young}
    \norm{x+y}^2 &\leq (1+s) \norm{x}^2 + (1+s^{-1}) \norm{y}^2, \\
    \label{eq:young_2}
    \norm{x+y}^2 &\leq 2 \norm{x}^2 + 2 \norm{y}^2, \\
    \label{eq:fenchel}
    \inp{x}{y} &\leq \frac{\norm{x}^2}{2s} + \frac{s\norm{y}^2}{2}, \\
    \label{eq:ineq1}
    \left( 1 - \alpha \right)\left( 1 + \frac{\alpha}{2} \right) &\leq 1 - \frac{\alpha}{2}, \\
    \label{eq:ineq2}
    \left( 1 - \alpha \right)\left( 1 + \frac{2}{\alpha} \right) &\leq \frac{2}{\alpha}, \\
    \label{eq:inp}
    \inp{a}{b} &= \frac{1}{2} \left( \norm{a}^2 + \norm{b}^2 - \norm{a-b}^2 \right).
\end{align}
\textbf{Variance decomposition:} For any random vector $X\in\R^d$ and any non-random $c\in\R^d$, we have
\begin{align}
\label{eq:vardecomp}
   \Exp{\norm{X-c}^2} = \Exp{\norm{X - \Exp{X}}^2} + \norm{\Exp{X}-c}^2.
\end{align}

\begin{lemma}[\cite{nesterov2018lectures}]
   \label{lemma:lipt_func}
   Let $f:\R^d\to \R$ be a function for which Assumptions \ref{ass:lipschitz_constant} and \ref{ass:convex} are satisfied. Then for all $x, y\in\R^d$ we have:
   \begin{align}
       \norm{\nabla f(x) - \nabla f(y)}^2 \leq 2L (f(x) - f(y) - \inp{\nabla f(y)}{x - y}).
   \end{align}
\end{lemma}

\section{Proofs of Theorems}

\subsection{Analysis of learning rates}
In this section, we establish inequalities for the sequences from Algorithm~\ref{alg:learning_rates}.
\begin{lemma}
  \label{lemma:learning_rates}
  Suppose that parameter $\bar{L} > 0,$ strong-convexity parameter $\mu \geq 0,$ probability $p \in (0, 1],$ $\bar{L} \geq \mu,$ and $\Gamma_0 \geq 1.$ Then the sequences generated by Algorithm~\ref{alg:learning_rates} have the following properties:
  \begin{enumerate}
    \item The quantities $\theta_{t+1}, \gamma_{t+1},$ and $\Gamma_{t+1}$ are well-defined and $\theta_{t+1}, \gamma_{t+1} \geq 0$ for all $t \geq 0.$
    \item $\gamma_{t+1} = p \theta_{t+1} \Gamma_{t+1}$ for all $t \geq 0.$
    \item $\bar{L} \theta_{t+1} \gamma_{t+1} \leq (\bar{L} + \Gamma_t \mu)$ for all $t \geq 0.$
    \item \begin{align*}
      \Gamma_{t} \geq \frac{\Gamma_0}{2}\exp\left(t \min\left\{\sqrt{\frac{p \mu}{4 \bar{L}}}, p\theta_{\min}\right\}\right)
    \end{align*}
    for all $t \geq 0.$
    \item 
    \begin{align*}
      \Gamma_{t} &\geq \begin{cases}
        \frac{\Gamma_0}{2} \exp \left(t p \theta_{\min}\right), &t < \bar{t} \\
        \frac{1}{4 p \theta_{\min}^2} + \frac{p (t - \bar{t})^2}{16}, &t \geq \bar{t}.
      \end{cases}
    \end{align*}
    where $\bar{t} \eqdef \max\left\{\left\lceil\frac{1}{p \theta_{\min}} \log \frac{1}{2 \Gamma_0 p \theta_{\min}^2}\right\rceil, 0\right\}.$
    \item $\{\theta_{t+1}\}_{t=0}^{\infty}$ is a non-increasing sequence.
  \end{enumerate}
\end{lemma}

\begin{proof}
  \hfill
  \begin{enumerate}
  \item 
  Note that $\bar{\theta}_{t+1}$ is the largest root of
  \begin{align}
    \label{eq:quad}
    p \bar{L} \Gamma_t \bar{\theta}^2_{t+1} + p (\bar{L} + \Gamma_t \mu) \bar{\theta}_{t+1} - (\bar{L} + \Gamma_t \mu) = 0.
  \end{align} 
  We fix $t \geq 0.$ Assume that $\Gamma_t > 0.$
  
  Then
  \begin{align*}
    p \bar{L} \Gamma_t \times 0 + p (\bar{L} + \Gamma_t \mu) \times 0 - (\bar{L} + \Gamma_t \mu) < 0.
  \end{align*}
  Therefore, the largest root $\bar{\theta}_{t+1}$ is well-defined and $\bar{\theta}_{t+1} \geq 0,$ and $\theta_{t+1} = \min\{\bar{\theta}_{t+1}, \theta_{\min}\} \geq 0.$ Next, $\gamma_{t+1}$ is well-defined and
  $$\gamma_{t+1} = p \theta_{t+1} \Gamma_t / (1 - p \theta_{t+1}) \geq 0$$ since $p \theta_{t+1} \in [0, \nicefrac{1}{4}].$ Finally, $\Gamma_{t+1} = \Gamma_{t} + \gamma_{t+1} > 0.$ 
  We showed that, for all $t \geq 0,$ if $\Gamma_t > 0,$ then $\theta_{t+1}, \gamma_{t+1} \geq 0$ and $\Gamma_{t+1} > 0.$ Note that $\Gamma_0 > 0,$ thus $\theta_{t+1}, \gamma_{t+1} \geq 0$ and $\Gamma_{t+1} > 0$ for all $t \geq 0.$
  \item
  From the definition of $\gamma_{t+1}$ and $\Gamma_{t+1},$ we have
  $$(1 - p \theta_{t+1}) \gamma_{t+1} = p \theta_{t+1} \Gamma_t,$$
  which is equivalent to
  $$\gamma_{t+1} = p \theta_{t+1} \left(\Gamma_t + \gamma_{t+1}\right) = p \theta_{t+1} \Gamma_{t+1}.$$
  \item
  Recall again that 
  $\bar{\theta}_{t+1} \geq 0$ is the largest root of
  $$p \bar{L} \Gamma_t \bar{\theta}^2_{t+1} + p (\bar{L} + \Gamma_t \mu) \bar{\theta}_{t+1} - (\bar{L} + \Gamma_t \mu) = 0.$$ 
  If $\bar{\theta}_{t+1} \leq \theta_{\min},$ then
  $$p \bar{L} \Gamma_t \theta^2_{t+1} + p (\bar{L} + \Gamma_t \mu) \theta_{t+1} - (\bar{L} + \Gamma_t \mu) = 0.$$ 
  Otherwise, if $\bar{\theta}_{t+1} > \theta_{\min},$ since
  $$p \bar{L} \Gamma_t \times 0 + p (\bar{L} + \Gamma_t \mu) \times 0 - (\bar{L} + \Gamma_t \mu) < 0,$$
  and $\theta_{t+1} = \theta_{\min} < \bar{\theta}_{t+1},$ then 
  \begin{align}p \bar{L} \Gamma_t \theta^2_{t+1} + p (\bar{L} + \Gamma_t \mu) \theta_{t+1} - (\bar{L} + \Gamma_t \mu) \leq 0.\label{eq:cond_theta}\end{align} In all cases, the inequality \eqref{eq:cond_theta} holds. From this inequality, we can get
  \begin{align*}
    p \bar{L} \Gamma_t \theta^2_{t+1} \leq (\bar{L} + \Gamma_t \mu) \left(1 - p \theta_{t+1}\right)
  \end{align*}
  and
  \begin{align*}
    \bar{L} \theta_{t+1} \frac{p \theta_{t+1} \Gamma_t}{\left(1 - p \theta_{t+1}\right)} \leq (\bar{L} + \Gamma_t \mu).
  \end{align*}
  Using the definition of $\gamma_{t+1},$ we obtain
  \begin{align*}
    \bar{L} \theta_{t+1} \gamma_{t+1} \leq (\bar{L} + \Gamma_t \mu)
  \end{align*}
  for all $t \geq 0.$
  \item Let us find the largest root of the quadratic equation \eqref{eq:quad}:
  \begin{align}
    \label{eq:theta_sol}
    \bar{\theta}_{t+1} \eqdef \frac{-p (\bar{L} + \Gamma_t \mu) + \sqrt{p^2 (\bar{L} + \Gamma_t \mu)^2 + 4 p \bar{L} \Gamma_t (\bar{L} + \Gamma_t \mu)}}{2 p \bar{L} \Gamma_t}.
  \end{align}
  Let us define $a \eqdef p^2\left(\bar{L} + \Gamma_{t} \mu\right)^2$ and $b \eqdef 4 \bar{L} p \left(\bar{L} + \Gamma_{t} \mu\right) \Gamma_{t},$
then
\begin{align*}
  \bar{\theta}_{t+1} = \frac{-\sqrt{a} + \sqrt{a + b}}{2 \bar{L} p \Gamma_t}
\end{align*}
Since $\Gamma_{t} \geq 1$ for all $t \geq 0,$ and $\bar{L} \geq \mu,$ we have $$a = p^2\left(\bar{L} + \Gamma_{t} \mu\right)^2 \leq p\left(\bar{L} + \Gamma_{t} \mu\right)^2 \leq p\left(\bar{L} + \Gamma_{t} \mu\right)\left(\Gamma_{t} \bar{L} + \Gamma_{t} \mu\right) \leq 2 \bar{L} p \left(\bar{L} + \Gamma_{t} \mu\right)\Gamma_{t} = \frac{b}{2}.$$
Using $\sqrt{x + y} \geq \left(\sqrt{x} + \sqrt{y}\right) / \sqrt{2}$ for all $x, y \geq 0, $ and $\sqrt{b} \geq \sqrt{2} \sqrt{a}$ we have
\begin{align*}
  \bar{\theta}_{t+1} &= \frac{-\sqrt{a} + \sqrt{a + b}}{2 \bar{L} p \Gamma_t} \geq \frac{-\sqrt{a} + \frac{1}{\sqrt{2}}\sqrt{a} + \frac{1}{\sqrt{2}}\sqrt{b}}{2 \bar{L} p \Gamma_t} \\
  &= \frac{\left(\frac{1}{\sqrt{2}} - 1\right)\sqrt{a} + \frac{1}{\sqrt{2}} \left(\frac{1}{\sqrt{2}} + 1 - \frac{1}{\sqrt{2}}\right)\sqrt{b}}{2 \bar{L} p \Gamma_t} \geq \frac{\frac{1}{\sqrt{2}} \left(\frac{1}{\sqrt{2}}\right)\sqrt{b}}{2 \bar{L} p \Gamma_t} = \frac{\sqrt{b}}{4 \bar{L} p \Gamma_t}.
\end{align*}
Therefore
\begin{align*}
  \bar{\theta}_{t+1} \geq \frac{\sqrt{4 \bar{L} p \left(\bar{L} + \Gamma_{t} \mu\right) \Gamma_{t}}}{4 \bar{L} p \Gamma_t} = \sqrt{\frac{\left(\bar{L} + \Gamma_{t} \mu\right)}{4 \bar{L} p \Gamma_t}} \geq \max\left\{\sqrt{\frac{1}{4 p \Gamma_t}}, \sqrt{\frac{\mu}{4 p \bar{L}}}\right\}
\end{align*}
and
\begin{align}
  \label{eq:theta_ineq}
  \theta_{t+1} \geq \min\left\{\max\left\{\sqrt{\frac{1}{4 p \Gamma_t}}, \sqrt{\frac{\mu}{4 p \bar{L}}}\right\}, \theta_{\min}\right\}.
\end{align}
Next, since $p \theta_{t+1} \in [0, \nicefrac{1}{4}],$ we have
\begin{align}
  \label{eq:gamma_ineq}
  \gamma_{t+1} \eqdef p \theta_{t+1} \Gamma_t / (1 - p \theta_{t+1}) \geq p \theta_{t+1} \Gamma_t \left(1 + p \theta_{t+1}\right)
\end{align}
and
\begin{align*}
  \Gamma_{t+1} \eqdef \Gamma_t + \gamma_{t+1} \geq \left(1 + p \theta_{t+1} + p^2 \theta_{t+1}^2\right) \Gamma_t.
\end{align*}
Using \eqref{eq:theta_ineq} and \eqref{eq:gamma_ineq}, we obtain
\begin{equation}
\begin{aligned}
  \label{eq:gamma_ineq_1}
  \Gamma_{t+1} &\geq \left(1 + p \min\left\{\sqrt{\frac{\mu}{4 p \bar{L}}}, \theta_{\min}\right\}\right) \Gamma_t = \left(1 + \min\left\{\sqrt{\frac{p \mu}{4 \bar{L}}}, p\theta_{\min}\right\}\right) \Gamma_t \\
  &\geq \Gamma_0 \left(1 + \min\left\{\sqrt{\frac{p \mu}{4 \bar{L}}}, p\theta_{\min}\right\}\right)^{t + 1} \geq \frac{\Gamma_0}{2}\exp\left((t + 1) \min\left\{\sqrt{\frac{p \mu}{4 \bar{L}}}, p\theta_{\min}\right\}\right),
\end{aligned}
\end{equation}
where we use that $1 + x \geq e^x / 2$ for all $x \in [0, 1].$ 
\item
Using \eqref{eq:theta_ineq} and \eqref{eq:gamma_ineq}, we have
\begin{align*}
  \Gamma_{t+1} &\geq \left(1 + p \theta_{t+1} + p^2 \theta_{t+1}^2\right) \Gamma_t \\
  &\geq \Gamma_t + p\min\left\{\sqrt{\frac{1}{4 p \Gamma_t}}, \theta_{\min}\right\} \Gamma_t + p^2\min\left\{\sqrt{\frac{1}{4 p \Gamma_t}}, \theta_{\min}\right\}^2 \Gamma_t.
\end{align*}
The sequence $\Gamma_{t+1}$ is strongly increasing. Thus, exists the minimal $\widehat{t} \geq 0$ such that $\Gamma_{\widehat{t}} \geq (4 p \theta_{\min}^2)^{-1}.$ For all $0 \leq t < \widehat{t},$ it holds that $\Gamma_{t} < (4 p \theta_{\min}^2)^{-1},$ $\sqrt{\frac{1}{4 p \Gamma_t}} > \theta_{\min},$ and 
\begin{align}
  \label{eq:gamma_first}
  \Gamma_{t+1} \geq \Gamma_t + p \theta_{\min} \Gamma_t + p^2 \theta_{\min}^2 \Gamma_t \geq \Gamma_t + p \theta_{\min} \Gamma_t \geq \Gamma_0 \left(1 + p \theta_{\min}\right)^{t+1} \geq \frac{\Gamma_0}{2} \exp \left((t+1) p \theta_{\min}\right).
\end{align}
Therefore, if $\widehat{t} > 0,$ then
\begin{align*}
  \frac{1}{4 p \theta_{\min}^2} > \Gamma_{\widehat{t} - 1} \geq \frac{\Gamma_0}{2} \exp \left((\widehat{t} - 1) p \theta_{\min}\right).
\end{align*}
Thus, we have the following bound for $\widehat{t}$:
\begin{align*}
  \widehat{t} \leq \bar{t} \eqdef \max\left\{\left\lceil\frac{1}{p \theta_{\min}} \log \frac{1}{2 \Gamma_0 p \theta_{\min}^2}\right\rceil, 0\right\}.
\end{align*}
For all $t \geq \widehat{t},$ we have $\sqrt{\frac{1}{4 p \Gamma_t}} \leq \theta_{\min}$ and
\begin{align*}
  \Gamma_{t+1} \geq \Gamma_t + p\sqrt{\frac{1}{4 p \Gamma_t}} \Gamma_t + p^2 \frac{1}{4 p \Gamma_t} \Gamma_t = \Gamma_t + \sqrt{\frac{p}{4}} \sqrt{\Gamma_t} + \frac{p}{4}.
\end{align*}
Using mathematical induction, let us show that 
\begin{align}
  \label{eq:gamma_second}
  \Gamma_{t} \geq \frac{1}{4 p \theta_{\min}^2} + \frac{p (t - \widehat{t})^2}{16}
\end{align} for all $t \geq \widehat{t}.$ For $t = \widehat{t}$ it is true since $\Gamma_{\widehat{t}} \geq \left(4 p \theta_{\min}^2\right)^{-1}$ by the definition of $\widehat{t}.$ Next, for some $t \geq \widehat{t},$ assume that \eqref{eq:gamma_second} holds, then 
\begin{align*}
  \Gamma_{t+1} &\geq \Gamma_t + \frac{\sqrt{p}}{2} \sqrt{\Gamma_t} + \frac{p}{4}\\
  &\geq \frac{1}{4 p \theta_{\min}^2} + \frac{p (t - \widehat{t})^2}{16} + \frac{\sqrt{p}}{2} \sqrt{\frac{1}{4 p \theta_{\min}^2} + \frac{p (t - \widehat{t})^2}{16}} + \frac{p}{4}\\
  &\geq \frac{1}{4 p \theta_{\min}^2} + \frac{p (t - \widehat{t})^2}{16} + \frac{p (t - \widehat{t})}{8} + \frac{p}{4}\\
  &\geq \frac{1}{4 p \theta_{\min}^2} + \frac{p (t - \widehat{t} + 1)^2}{16}.
\end{align*}
We proved the inequality using mathematical induction.
Combining \eqref{eq:gamma_first} and \eqref{eq:gamma_second}, we obtain the unified inequality for $\Gamma_t$:
\begin{align*}
  \Gamma_{t} &\geq \min\left\{\frac{\Gamma_0}{2} \exp \left(t p \theta_{\min}\right), \frac{1}{4 p \theta_{\min}^2} + \frac{p (t - \widehat{t})^2}{16}\right\}
\end{align*}
for all $t \geq 0.$ Also, if $t < \bar{t},$ then the first term in the minimum is less or equal than the second one. Therefore, 
\begin{align}
  \label{eq:aux_gamma_min}
  \Gamma_{t} &\geq \begin{cases}
    \frac{\Gamma_0}{2} \exp \left(t p \theta_{\min}\right), &t < \bar{t} \\
    \min\left\{\frac{\Gamma_0}{2} \exp \left(t p \theta_{\min}\right), \frac{1}{4 p \theta_{\min}^2} + \frac{p (t - \widehat{t})^2}{16}\right\}, &t \geq \bar{t}.
  \end{cases}
\end{align}
Let us bound the second term.
Recall that $\bar{t} \geq \widehat{t},$ thus, if $t \geq \bar{t},$ then $(t - \widehat{t})^2 \geq (t - \bar{t})^2.$ In \eqref{eq:aux_gamma_min}, we can change $\widehat{t}$ to $\bar{t}$ and get
\begin{align*}
  \Gamma_{t} &\geq \begin{cases}
    \frac{\Gamma_0}{2} \exp \left(t p \theta_{\min}\right), &t < \bar{t} \\
    \min\left\{\frac{\Gamma_0}{2} \exp \left(t p \theta_{\min}\right), \frac{1}{4 p \theta_{\min}^2} + \frac{p (t - \bar{t})^2}{16}\right\}, &t \geq \bar{t}.
  \end{cases}
\end{align*}
Using the Taylor expansion of the exponent at the point $\bar{t}$, we get
\begin{align*}
  \frac{\Gamma_0}{2} \exp \left(t p \theta_{\min}\right) &\geq \frac{\Gamma_0}{2} \exp \left(\bar{t} p \theta_{\min}\right) + \frac{\Gamma_0}{2} p^2 \theta_{\min}^2 \exp \left(\bar{t} p \theta_{\min}\right) \frac{(t - \bar{t})^2}{2}\\
  &\geq \frac{1}{4 p \theta_{\min}^2} + \frac{p}{8} (t - \bar{t})^2 \geq \frac{1}{4 p \theta_{\min}^2} + \frac{p}{16} (t - \bar{t})^2.
\end{align*}
for all $t \geq \bar{t}.$ Finally, we can conclude that
\begin{align*}
  \Gamma_{t} &\geq \begin{cases}
    \frac{\Gamma_0}{2} \exp \left(t p \theta_{\min}\right), &t < \bar{t} \\
    \frac{1}{4 p \theta_{\min}^2} + \frac{p (t - \bar{t})^2}{16}, &t \geq \bar{t}.
  \end{cases}
\end{align*}
\item
Let us rewrite \eqref{eq:theta_sol}:
\begin{align*}
  \bar{\theta}_{t+1} &= \frac{-p (\bar{L} + \Gamma_t \mu) + \sqrt{p^2 (\bar{L} + \Gamma_t \mu)^2 + 4 p \bar{L} \Gamma_t (\bar{L} + \Gamma_t \mu)}}{2 p \bar{L} \Gamma_t} \\
  &= -\frac{1}{2} \left(\frac{1}{\Gamma_t} + \frac{\mu}{\bar{L}}\right) + \sqrt{\frac{1}{4} \left(\frac{1}{\Gamma_t} + \frac{\mu}{\bar{L}}\right)^2 + \frac{1}{p} \left(\frac{1}{\Gamma_t} + \frac{\mu}{\bar{L}}\right)}. \\
\end{align*}
Let us temporarily denote $a_t \eqdef \left(\frac{1}{\Gamma_t} + \frac{\mu}{\bar{L}}\right)$ for all $t \geq 0.$ Note that $a_t$ is a non-increasing sequence, since $\Gamma_{t+1} \geq \Gamma_{t}$ for all $t \geq 0.$ Therefore,
\begin{align*}
  \bar{\theta}_{t+1} &= -\frac{1}{2} a_t + \sqrt{\frac{1}{4} a_t^2 + \frac{1}{p} a_t}. \\
\end{align*}
Let us take the derivative of the last term w.r.t. $a_t$ and compare it to zero:
\begin{align*}
  -\frac{1}{2} + \frac{\frac{1}{2} a_t + \frac{1}{p}}{2 \sqrt{\frac{1}{4} a_t^2 + \frac{1}{p} a_t}} \geq 0 &\Leftrightarrow \frac{1}{2} a_t + \frac{1}{p} \geq \sqrt{\frac{1}{4} a_t^2 + \frac{1}{p} a_t} \\
  &\Leftrightarrow \frac{1}{4} a_t^2 + \frac{1}{p} a_t + \frac{1}{p^2} \geq \frac{1}{4} a_t^2 + \frac{1}{p} a_t \Leftrightarrow \frac{1}{p^2} \geq 0.
\end{align*}
Thus, $\bar{\theta}_{t+1}$ is a non-decreasing sequence w.r.t. $a_t.$ But the sequence $a_t$ is non-increasing, therefore $\bar{\theta}_{t+1}$ is a non-increasing sequence w.r.t. $t.$ It is left to use that $\theta_{t+1}$ is the minimum of $\bar{\theta}_{t+1}$ and the constant quantity.
\end{enumerate}
\end{proof}

\subsection{Generic lemmas}

First, we prove a well-known lemma from the theory of accelerated methods \citep{lan2020first, stonyakin2021inexact}.
\begin{lemma}
  \label{lemma:aux_opt}
  Let us take vectors $a, b, g \in \R^d,$ numerical quantities $\alpha, \beta \geq 0,$ and
  \begin{align*}
    u = \argmin_{x \in \R^d} \inp{g}{x} + \frac{\alpha}{2} \norm{x - a}^2 + \frac{\beta}{2} \norm{x - b}^2.
  \end{align*}
  Then
  \begin{align}
    \label{eq:aux_opt}
    &\inp{g}{x} + \frac{\alpha}{2} \norm{x - a}^2 + \frac{\beta}{2} \norm{x - b}^2 \geq \inp{g}{u} + \frac{\alpha}{2} \norm{u - a}^2 + \frac{\beta}{2} \norm{u - b}^2 + \frac{\alpha + \beta}{2} \norm{x - u}^2
  \end{align}
  for all $x \in \R^d.$
\end{lemma}

\begin{proof}
  A function 
  \begin{align*}
    \widehat{f}(x) \eqdef \inp{g}{x} + \frac{\alpha}{2} \norm{x - a}^2 + \frac{\beta}{2} \norm{x - b}^2
  \end{align*}
  is strongly-convex with the parameter $\alpha + \beta.$ From the strong-convexity and the optimality of $u$, we obtain
  \begin{align*}
    \widehat{f}(x) \geq \widehat{f}(u) + \inp{\widehat{f}(u)}{x - u} + \frac{\alpha + \beta}{2} \norm{x - u}^2 = \widehat{f}(u) + \frac{\alpha + \beta}{2} \norm{x - u}^2
  \end{align*}
  for all $x \in \R^d.$
  This inequality is equivalent to \eqref{eq:aux_opt}.
\end{proof}

We assume that the conditional expectation $\ExpSub{t}{\cdot}$ is conditioned on the randomness from the first $t$ iterations. Also, let us define $D_f(x, y) \eqdef f(x) - f(y) - \inp{\nabla f(y)}{x - y}.$

\begin{lemma}
  \label{lemma:aux_stoch_grad}
  Suppose that Assumptions~\ref{ass:lipschitz_constant}, \ref{ass:workers_lipschitz_constant}, \ref{ass:convex} and \ref{ass:unbiased_compressors} hold. For Algorithm~\ref{alg:bi_diana}, the following inequality holds:
  \begin{equation}
    \begin{aligned}
      \label{eq:nabla_g_f}
    &\ExpSub{t}{\norm{g^{t+1} - \nabla f(y^{t+1})}^2} \\ 
    &\leq\frac{2 \omega}{n^2} \sum_{i=1}^n \norm{\nabla f_i(z^{t}) - h_i^t}^2 + \frac{4 \omega L_{\max}}{n} (f(z^t) - f(y^{t+1}) - \inp{\nabla f(y^{t+1})}{z^t - y^{t+1}}).
  \end{aligned}
  \end{equation}
\end{lemma}

\begin{proof}
  Using the definition of $g^{t+1},$ we have
  \begin{align*}
    &\ExpSub{t}{\norm{g^{t+1} - \nabla f(y^{t+1})}^2} \\ 
    &=\ExpSub{t}{\norm{h^t + \frac{1}{n}\sum_{i=1}^n \cC_i^{D, y}(\nabla f_i(y^{t+1}) - h_i^t) - \nabla f(y^{t+1})}^2} \\
    &=\frac{1}{n^2} \sum_{i=1}^n \ExpSub{t}{\norm{\cC_i^{D, y}(\nabla f_i(y^{t+1}) - h_i^t) - \left(\nabla f(y^{t+1}) - h_i^t\right)}^2},
  \end{align*}
  where we use the independence and the unbiasedness of the compressors (see Assumption~\ref{ass:unbiased_compressors}). Using Definition~\ref{def:unbiased_compression} and Assumption~\ref{ass:workers_lipschitz_constant}, we have
  \begin{equation*}
  \begin{aligned}
    &\ExpSub{t}{\norm{g^{t+1} - \nabla f(y^{t+1})}^2} \\ 
    &\leq\frac{\omega}{n^2} \sum_{i=1}^n \norm{\nabla f_i(y^{t+1}) - h_i^t}^2 \\
    &\overset{\eqref{eq:young_2}}{\leq}\frac{2 \omega}{n^2} \sum_{i=1}^n \norm{\nabla f_i(z^{t}) - h_i^t}^2 + \frac{2 \omega}{n^2} \sum_{i=1}^n \norm{\nabla f_i(y^{t+1}) - \nabla f_i(z^{t})}^2 \\
    &\overset{L.\ref{lemma:lipt_func}}{\leq}\frac{2 \omega}{n^2} \sum_{i=1}^n \norm{\nabla f_i(z^{t}) - h_i^t}^2 + \frac{2 \omega}{n^2} \sum_{i=1}^n 2L_i (f_i(z^t) - f_i(y^{t+1}) - \inp{\nabla f_i(y^{t+1})}{z^t - y^{t+1}}).
  \end{aligned}
  \end{equation*}
  Using that $L_{\max} = \max_{i \in [n]} L_i,$ we obtain \eqref{eq:nabla_g_f}.
\end{proof}

\begin{lemma}
  \label{lemma:aux_first}
  Suppose that Assumptions~\ref{ass:lipschitz_constant}, \ref{ass:workers_lipschitz_constant}, \ref{ass:convex} and \ref{ass:unbiased_compressors} hold. For Algorithm~\ref{alg:bi_diana}, the following inequality holds:
  \begin{align*}
    &\ExpSub{t}{f(z^{t+1}) - f(x^*)} \\
    &\leq (1 - p\theta_{t+1}) \left(f(z^{t}) - f(x^*)\right) \\
    &\quad + \frac{2 p \omega}{n \bar{L}} \left(\frac{1}{n}\sum_{i=1}^n \norm{h_i^t - \nabla f_i(z^{t})}^2\right) \\
    &\quad+ p\theta_{t+1} \left(\frac{\bar{L} + \Gamma_{t} \mu}{2 \gamma_{t+1}} \norm{u^t - x^*}^2 - \frac{\bar{L} + \Gamma_{t+1} \mu}{2 \gamma_{t+1}} \ExpSub{t}{\norm{u^{t+1} - x^*}^2}\right) \\
    &\quad+ p\left(\frac{4 \omega L_{\max}}{n \bar{L}} + \theta_{t+1} - 1 \right) D_f(z^t, y^{t+1}) \\
    &\quad+ \frac{pL}{2} \ExpSub{t}{\norm{x^{t+1} - y^{t+1}}^2} - \frac{p \theta_{t+1}^2 \bar{L}}{2} \ExpSub{t}{\norm{u^{t+1} - u^t}^2}.
  \end{align*}
\end{lemma}

\begin{proof}
  Using Assumption~\ref{ass:lipschitz_constant}, we have
\begin{align*}
  &f(x^{t+1}) - f(x^*) \leq f(y^{t+1}) - f(x^*) + \inp{\nabla f(y^{t+1})}{x^{t+1} - y^{t+1}} + \frac{L}{2} \norm{x^{t+1} - y^{t+1}}^2.
\end{align*}
Using the definition of $x^{t+1},$ we obtain
\begin{equation}
\begin{aligned}
  &f(x^{t+1}) - f(x^*)\leq (1 - \theta_{t+1}) \left(f(y^{t+1}) - f(x^*) + \inp{\nabla f(y^{t+1})}{z^t - y^{t+1}}\right) \\
  &\quad+ \theta_{t+1} \left(f(y^{t+1}) - f(x^*) + \inp{\nabla f(y^{t+1})}{u^{t+1} - y^{t+1}}\right) \\
  &\quad+ \frac{L}{2} \norm{x^{t+1} - y^{t+1}}^2 \\
  &= (1 - \theta_{t+1}) \left(f(y^{t+1}) - f(x^*) + \inp{\nabla f(y^{t+1})}{z^t - y^{t+1}}\right) \\
  &\quad+ \theta_{t+1} \left(f(y^{t+1}) - f(x^*) + \inp{g^{t+1}}{u^{t+1} - y^{t+1}}\right) \\
  &\quad+ \theta_{t+1} \left(\inp{\nabla f(y^{t+1}) - g^{t+1}}{u^{t+1} - y^{t+1}}\right) \\
  &\quad+ \frac{L}{2} \norm{x^{t+1} - y^{t+1}}^2,
  \label{eq:tmp_1}
\end{aligned}
\end{equation}
in the last inequality we add and subtract $g^{t+1}.$
Using the definition of $u^{t+1}$ and Lemma~\ref{lemma:aux_opt} with $x = x^*$, we have
\begin{align*}
  \inp{g^{t+1}}{u^{t+1} - y^{t+1}} &\leq \inp{g^{t+1}}{x^* - y^{t+1}} + \frac{\bar{L} + \Gamma_{t} \mu}{2 \gamma_{t+1}} \norm{x^* - u^t}^2 + \frac{\mu}{2} \norm{x^* - y^{t+1}}^2 \\
  &\quad - \frac{\bar{L} + \Gamma_{t} \mu}{2 \gamma_{t+1}} \norm{u^{t+1} - u^t}^2 - \frac{\mu}{2} \norm{u^t - y^{t+1}}^2 - \frac{\bar{L} + \Gamma_{t+1} \mu}{2 \gamma_{t+1}} \norm{x^* - u^{t+1}}^2.
\end{align*}
We use the fact that $\Gamma_{t+1} = \Gamma_{t} + \gamma_{t+1}.$ Since $\norm{u^t - y^{t+1}}^2 \geq 0,$ we have
\begin{align*}
  \inp{g^{t+1}}{u^{t+1} - y^{t+1}} &\leq \inp{g^{t+1}}{x^* - y^{t+1}} + \frac{\bar{L} + \Gamma_{t} \mu}{2 \gamma_{t+1}} \norm{x^* - u^t}^2 + \frac{\mu}{2} \norm{x^* - y^{t+1}}^2 \\
  &\quad - \frac{\bar{L} + \Gamma_{t} \mu}{2 \gamma_{t+1}} \norm{u^{t+1} - u^t}^2 - \frac{\bar{L} + \Gamma_{t+1} \mu}{2 \gamma_{t+1}} \norm{x^* - u^{t+1}}^2.
\end{align*}
By substituting this inequality to \eqref{eq:tmp_1}, we get
\begin{align*}
  &f(x^{t+1}) - f(x^*) \\
  &\leq (1 - \theta_{t+1}) \left(f(y^{t+1}) - f(x^*) + \inp{\nabla f(y^{t+1})}{z^t - y^{t+1}}\right) \\
  &\quad+ \theta_{t+1} \left(f(y^{t+1}) - f(x^*) + \inp{g^{t+1}}{x^* - y^{t+1}}\right) \\
  &\quad+ \theta_{t+1} \left(\frac{\bar{L} + \Gamma_{t} \mu}{2 \gamma_{t+1}} \norm{x^* - u^t}^2 + \frac{\mu}{2} \norm{x^* - y^{t+1}}^2 - \frac{\bar{L} + \Gamma_{t} \mu}{2 \gamma_{t+1}} \norm{u^{t+1} - u^t}^2 - \frac{\bar{L} + \Gamma_{t+1} \mu}{2 \gamma_{t+1}} \norm{x^* - u^{t+1}}^2\right) \\
  &\quad+ \theta_{t+1} \left(\inp{\nabla f(y^{t+1}) - g^{t+1}}{u^{t+1} - y^{t+1}}\right) \\
  &\quad+ \frac{L}{2} \norm{x^{t+1} - y^{t+1}}^2.
\end{align*}
Using $\mu$--strong convexity, we have
\begin{align*}
  f(x^*) \geq f(y^{t+1}) + \inp{\nabla f(y^{t+1})}{x^* - y^{t+1}} + \frac{\mu}{2} \norm{x^* - y^{t+1}}^2
\end{align*}
and
\begin{align*}
  &f(x^{t+1}) - f(x^*) \\
  &\leq (1 - \theta_{t+1}) \left(f(y^{t+1}) - f(x^*) + \inp{\nabla f(y^{t+1})}{z^t - y^{t+1}}\right) \\
  &\quad+ \theta_{t+1} \left(\inp{g^{t+1} - \nabla f(y^{t+1})}{x^* - y^{t+1}}\right) \\
  &\quad+ \theta_{t+1} \left(\frac{\bar{L} + \Gamma_{t} \mu}{2 \gamma_{t+1}} \norm{x^* - u^t}^2 - \frac{\bar{L} + \Gamma_{t} \mu}{2 \gamma_{t+1}} \norm{u^{t+1} - u^t}^2 - \frac{\bar{L} + \Gamma_{t+1} \mu}{2 \gamma_{t+1}} \norm{x^* - u^{t+1}}^2\right) \\
  &\quad+ \theta_{t+1} \left(\inp{\nabla f(y^{t+1}) - g^{t+1}}{u^{t+1} - y^{t+1}}\right) \\
  &\quad+ \frac{L}{2} \norm{x^{t+1} - y^{t+1}}^2.
\end{align*}
Let us take the conditional expectation $\ExpSub{t}{\cdot}$ conditioned on the randomness from the first $t$ iterations:
\begin{equation}
\begin{aligned}
  &\ExpSub{t}{f(x^{t+1}) - f(x^*)} \\
  &\leq (1 - \theta_{t+1}) \left(f(y^{t+1}) - f(x^*) + \inp{\nabla f(y^{t+1})}{z^t - y^{t+1}}\right) \\
  &\quad+ \theta_{t+1} \left(\inp{\ExpSub{t}{g^{t+1} - \nabla f(y^{t+1})}}{x^* - y^{t+1}}\right) \\
  &\quad+ \theta_{t+1} \left(\frac{\bar{L} + \Gamma_{t} \mu}{2 \gamma_{t+1}} \norm{x^* - u^t}^2 - \frac{\bar{L} + \Gamma_{t} \mu}{2 \gamma_{t+1}} \ExpSub{t}{\norm{u^{t+1} - u^t}^2} - \frac{\bar{L} + \Gamma_{t+1} \mu}{2 \gamma_{t+1}} \ExpSub{t}{\norm{x^* - u^{t+1}}^2}\right) \\
  &\quad+ \theta_{t+1} \ExpSub{t}{\inp{\nabla f(y^{t+1}) - g^{t+1}}{u^{t+1} - y^{t+1}}} \\
  &\quad+ \frac{L}{2} \ExpSub{t}{\norm{x^{t+1} - y^{t+1}}^2} \\
  &= (1 - \theta_{t+1}) \left(f(y^{t+1}) - f(x^*) + \inp{\nabla f(y^{t+1})}{z^t - y^{t+1}}\right) \\
  &\quad+ \theta_{t+1} \left(\frac{\bar{L} + \Gamma_{t} \mu}{2 \gamma_{t+1}} \norm{x^* - u^t}^2 - \frac{\bar{L} + \Gamma_{t} \mu}{2 \gamma_{t+1}} \ExpSub{t}{\norm{u^{t+1} - u^t}^2} - \frac{\bar{L} + \Gamma_{t+1} \mu}{2 \gamma_{t+1}} \ExpSub{t}{\norm{x^* - u^{t+1}}^2}\right) \\
  &\quad+ \theta_{t+1} \ExpSub{t}{\inp{\nabla f(y^{t+1}) - g^{t+1}}{u^{t+1} - y^{t+1}}} \\
  &\quad+ \frac{L}{2} \ExpSub{t}{\norm{x^{t+1} - y^{t+1}}^2},
  \label{eq:tmp_2}
\end{aligned}
\end{equation}
where use that $\ExpSub{t}{g^{t+1}} = \nabla f(y^{t+1}).$ We can find $u^{t+1}$ analytically and obtain that
\begin{align*}
  u^{t+1} = \frac{\bar{L} + \Gamma_{t} \mu}{\bar{L} + \Gamma_{t+1} \mu} u^t + \frac{\mu \gamma_{t+1}}{\bar{L} + \Gamma_{t+1} \mu} y^{t+1} - \frac{\gamma_{t+1}}{\bar{L} + \Gamma_{t+1} \mu} g^{t+1}.
\end{align*}
Therefore, using that $\ExpSub{t}{g^{t+1}} = \nabla f(y^{t+1})$ and $u^t$ and $y^{t+1}$ are conditionally nonrandom, we obtain
\begin{align}
  &\ExpSub{t}{\inp{\nabla f(y^{t+1}) - g^{t+1}}{u^{t+1} - y^{t+1}}} = \frac{\gamma_{t+1}}{\bar{L} + \Gamma_{t+1} \mu} \ExpSub{t}{\norm{g^{t+1} - \nabla f(y^{t+1})}^2}.
  \label{eq:tmp_3}
\end{align}
Combining \eqref{eq:nabla_g_f} from Lemma~\ref{lemma:aux_stoch_grad} with \eqref{eq:tmp_2} and \eqref{eq:tmp_3}, one can get
\begin{equation*}
  \begin{aligned}
    &\ExpSub{t}{f(x^{t+1}) - f(x^*)} \\
    &\leq (1 - \theta_{t+1}) \left(f(y^{t+1}) - f(x^*) + \inp{\nabla f(y^{t+1})}{z^t - y^{t+1}}\right) \\
    &\quad+ \theta_{t+1} \left(\frac{\bar{L} + \Gamma_{t} \mu}{2 \gamma_{t+1}} \norm{x^* - u^t}^2 - \frac{\bar{L} + \Gamma_{t} \mu}{2 \gamma_{t+1}} \ExpSub{t}{\norm{u^{t+1} - u^t}^2} - \frac{\bar{L} + \Gamma_{t+1} \mu}{2 \gamma_{t+1}} \ExpSub{t}{\norm{x^* - u^{t+1}}^2}\right) \\
    &\quad+ \frac{\theta_{t+1} \gamma_{t+1}}{\bar{L} + \Gamma_{t+1} \mu} \left(\frac{2 \omega}{n^2} \sum_{i=1}^n \norm{\nabla f_i(z^{t}) - h_i^t}^2 + \frac{4 \omega L_{\max}}{n} (f(z^t) - f(y^{t+1}) - \inp{\nabla f(y^{t+1})}{z^t - y^{t+1}})\right) \\
    &\quad+ \frac{L}{2} \ExpSub{t}{\norm{x^{t+1} - y^{t+1}}^2}.
  \end{aligned}
\end{equation*}
Using the notation $D_f(x, y) \eqdef f(x) - f(y) - \inp{\nabla f(y)}{x - y},$ we get
\begin{equation*}
  \begin{aligned}
    &\ExpSub{t}{f(x^{t+1}) - f(x^*)} \\
    &\leq (1 - \theta_{t+1}) \left(f(z^{t}) - f(x^*) - D_f(z^t, y^{t+1})\right) \\
    &\quad+ \theta_{t+1} \left(\frac{\bar{L} + \Gamma_{t} \mu}{2 \gamma_{t+1}} \norm{x^* - u^t}^2 - \frac{\bar{L} + \Gamma_{t} \mu}{2 \gamma_{t+1}} \ExpSub{t}{\norm{u^{t+1} - u^t}^2} - \frac{\bar{L} + \Gamma_{t+1} \mu}{2 \gamma_{t+1}} \ExpSub{t}{\norm{x^* - u^{t+1}}^2}\right) \\
    &\quad+ \frac{\theta_{t+1} \gamma_{t+1}}{\bar{L} + \Gamma_{t+1} \mu} \left(\frac{2 \omega}{n^2} \sum_{i=1}^n \norm{\nabla f_i(z^{t}) - h_i^t}^2 + \frac{4 \omega L_{\max}}{n} D_f(z^t, y^{t+1})\right) \\
    &\quad+ \frac{L}{2} \ExpSub{t}{\norm{x^{t+1} - y^{t+1}}^2} \\
    &= (1 - \theta_{t+1}) \left(f(z^{t}) - f(x^*)\right) \\
    &\quad + \frac{\theta_{t+1} \gamma_{t+1}}{\bar{L} + \Gamma_{t+1} \mu} \frac{2 \omega}{n} \left(\frac{1}{n}\sum_{i=1}^n \norm{\nabla f_i(z^{t}) - h_i^t}^2\right) \\
    &\quad+ \theta_{t+1} \left(\frac{\bar{L} + \Gamma_{t} \mu}{2 \gamma_{t+1}} \norm{x^* - u^t}^2 - \frac{\bar{L} + \Gamma_{t+1} \mu}{2 \gamma_{t+1}} \ExpSub{t}{\norm{x^* - u^{t+1}}^2}\right) \\
    &\quad+ \left(\frac{\theta_{t+1} \gamma_{t+1}}{\bar{L} + \Gamma_{t+1} \mu} \frac{4 \omega L_{\max}}{n} + \theta_{t+1} - 1 \right) D_f(z^t, y^{t+1}) \\
    &\quad+ \frac{L}{2} \ExpSub{t}{\norm{x^{t+1} - y^{t+1}}^2} - \frac{\theta_{t+1} \left(\bar{L} + \Gamma_{t} \mu\right)}{2 \gamma_{t+1}} \ExpSub{t}{\norm{u^{t+1} - u^t}^2}.
  \end{aligned}
\end{equation*}
In the last equality, we simply regrouped the terms.
Using the definition of $z^{t+1},$ we get
\begin{align*}
  &\ExpSub{t}{f(z^{t+1}) - f(x^*)} = p\ExpSub{t}{f(x^{t+1}) - f(x^*)} + (1 - p) \left(f(z^t) - f(x^*)\right) \\
  &\leq p(1 - \theta_{t+1}) \left(f(z^{t}) - f(x^*)\right) \\
  &\quad + p\frac{\theta_{t+1} \gamma_{t+1}}{\bar{L} + \Gamma_{t+1} \mu} \frac{2 \omega}{n} \left(\frac{1}{n}\sum_{i=1}^n \norm{\nabla f_i(z^{t}) - h_i^t}^2\right) \\
  &\quad+ p\theta_{t+1} \left(\frac{\bar{L} + \Gamma_{t} \mu}{2 \gamma_{t+1}} \norm{x^* - u^t}^2 - \frac{\bar{L} + \Gamma_{t+1} \mu}{2 \gamma_{t+1}} \ExpSub{t}{\norm{x^* - u^{t+1}}^2}\right) \\
  &\quad+ p\left(\frac{\theta_{t+1} \gamma_{t+1}}{\bar{L} + \Gamma_{t+1} \mu} \frac{4 \omega L_{\max}}{n} + \theta_{t+1} - 1 \right) D_f(z^t, y^{t+1}) \\
  &\quad+ \frac{pL}{2} \ExpSub{t}{\norm{x^{t+1} - y^{t+1}}^2} - \frac{p \theta_{t+1} \left(\bar{L} + \Gamma_{t} \mu\right)}{2 \gamma_{t+1}} \ExpSub{t}{\norm{u^{t+1} - u^t}^2} + (1 - p) \left(f(z^t) - f(x^*)\right) \\
  &= (1 - p\theta_{t+1}) \left(f(z^{t}) - f(x^*)\right) \\
  &\quad + p\frac{\theta_{t+1} \gamma_{t+1}}{\bar{L} + \Gamma_{t+1} \mu} \frac{2 \omega}{n} \left(\frac{1}{n}\sum_{i=1}^n \norm{\nabla f_i(z^{t}) - h_i^t}^2\right) \\
  &\quad+ p\theta_{t+1} \left(\frac{\bar{L} + \Gamma_{t} \mu}{2 \gamma_{t+1}} \norm{x^* - u^t}^2 - \frac{\bar{L} + \Gamma_{t+1} \mu}{2 \gamma_{t+1}} \ExpSub{t}{\norm{x^* - u^{t+1}}^2}\right) \\
  &\quad+ p\left(\frac{\theta_{t+1} \gamma_{t+1}}{\bar{L} + \Gamma_{t+1} \mu} \frac{4 \omega L_{\max}}{n} + \theta_{t+1} - 1 \right) D_f(z^t, y^{t+1}) \\
  &\quad+ \frac{pL}{2} \ExpSub{t}{\norm{x^{t+1} - y^{t+1}}^2} - \frac{p \theta_{t+1} \left(\bar{L} + \Gamma_{t} \mu\right)}{2 \gamma_{t+1}} \ExpSub{t}{\norm{u^{t+1} - u^t}^2}.
\end{align*}
In the last equality, we grouped the terms with $f(z^{t}) - f(x^*).$ In Algorithm~\ref{alg:learning_rates}, we choose the learning rates so that (see Lemma~\ref{lemma:learning_rates})
\begin{align*}
  \bar{L} \theta_{t+1} \gamma_{t+1} \leq \bar{L} + \Gamma_{t} \mu.
\end{align*}
Since $\Gamma_{t+1} \geq \Gamma_{t}$ for all $t \in \N_0,$ thus
\begin{align*}
  \frac{\theta_{t+1} \gamma_{t+1}}{\bar{L} + \Gamma_{t+1} \mu} \leq \frac{\theta_{t+1} \gamma_{t+1}}{\bar{L} + \Gamma_{t} \mu} \leq \frac{1}{\bar{L}}
\end{align*}
and
\begin{align*}
  &\ExpSub{t}{f(z^{t+1}) - f(x^*)} \\
  &\leq (1 - p\theta_{t+1}) \left(f(z^{t}) - f(x^*)\right) \\
  &\quad + p\frac{2 \omega}{n \bar{L}} \left(\frac{1}{n}\sum_{i=1}^n \norm{\nabla f_i(z^{t}) - h_i^t}^2\right) \\
  &\quad+ p\theta_{t+1} \left(\frac{\bar{L} + \Gamma_{t} \mu}{2 \gamma_{t+1}} \norm{x^* - u^t}^2 - \frac{\bar{L} + \Gamma_{t+1} \mu}{2 \gamma_{t+1}} \ExpSub{t}{\norm{x^* - u^{t+1}}^2}\right) \\
  &\quad+ p\left(\frac{4 \omega L_{\max}}{n \bar{L}} + \theta_{t+1} - 1 \right) D_f(z^t, y^{t+1}) \\
  &\quad+ \frac{pL}{2} \ExpSub{t}{\norm{x^{t+1} - y^{t+1}}^2} - \frac{p \theta_{t+1}^2 \bar{L}}{2} \ExpSub{t}{\norm{u^{t+1} - u^t}^2}.
\end{align*}
\end{proof}

\subsection{Construction of the Lyapunov function}

In this section, we provide lemmas that will help us to construct a Lyapunov function.

\begin{lemma}
  \label{lemma:aux_control}
  Suppose that Assumptions~\ref{ass:lipschitz_constant}, \ref{ass:workers_lipschitz_constant}, \ref{ass:convex} and \ref{ass:unbiased_compressors} hold. The parameter $\beta \leq \nicefrac{1}{\omega + 1}.$ Then, for Algorithm~\ref{alg:bi_diana}, the following inequality holds:
  \begin{align}
    \label{eq:aux_control}
    &\ExpSub{t}{\frac{1}{n}\sum_{i=1}^n \norm{h_i^{t+1} - \nabla f_i(z^{t+1})}^2} \\
    &\leq 8 p \left(1 + \frac{p}{\beta}\right) L_{\max} D_f(z^t, y^{t+1}) + 4 p \left(1 + \frac{p}{\beta}\right) \widehat{L}^2 \ExpSub{t}{\norm{x^{t+1} - y^{t+1}}^2} + \left(1 - \frac{\beta}{2}\right) \frac{1}{n}\sum_{i=1}^n \norm{h^{t}_i - \nabla f_i(z^t)}^2.
  \end{align}
\end{lemma}

\begin{proof}
  Using the definition of $h^{t+1}_i$, we have
  \begin{align*}
    &\ExpSub{t}{\frac{1}{n}\sum_{i=1}^n \norm{h_i^{t+1} - \nabla f_i(z^{t+1})}^2} \\
    &=\ExpSub{t}{\frac{1}{n}\sum_{i=1}^n \norm{h_i^{t} + \beta \cC_i^{D, z}(\nabla f_i(z^{t+1}) - h_i^t) - \nabla f_i(z^{t+1})}^2} \\
    &=\ExpSub{t}{\frac{1}{n}\sum_{i=1}^n \norm{h^{t}_i - \nabla f_i(z^{t+1})}^2} \\
    &\quad + \ExpSub{t}{\frac{2 \beta}{n}\sum_{i=1}^n \inp{h^{t}_i - \nabla f_i(z^{t+1})}{\cC_i^{D, z}(\nabla f_i(z^{t+1}) - h_i^t)} + \frac{\beta^2}{n}\sum_{i=1}^n \norm{\cC_i^{D, z}(\nabla f_i(z^{t+1}) - h_i^t)}^2}.
  \end{align*}
  Note that $\ExpSub{\cC}{\cC_i^{D, z}(\nabla f_i(z^{t+1}) - h_i^t)} = \nabla f_i(z^{t+1}) - h_i^t$ and $$\ExpSub{\cC}{\norm{\cC_i^{D, z}(\nabla f_i(z^{t+1}) - h_i^t)}^2} \leq \left(\omega + 1\right)\norm{\nabla f_i(z^{t+1}) - h_i^t}^2,$$
  where $\ExpSub{\cC}{\cdot}$ is a conditional expectation that is conditioned on $z^{t+1}$ and $h^t_i.$ Therefore, 
  \begin{align*}
    &\ExpSub{t}{\frac{1}{n}\sum_{i=1}^n \norm{h_i^{t+1} - \nabla f_i(z^{t+1})}^2} \\
    &\leq\ExpSub{t}{\frac{1}{n}\sum_{i=1}^n \norm{h^{t}_i - \nabla f_i(z^{t+1})}^2} \\
    &\quad + \ExpSub{t}{\frac{2 \beta}{n}\sum_{i=1}^n \inp{h^{t}_i - \nabla f_i(z^{t+1})}{\nabla f_i(z^{t+1}) - h_i^t} + \frac{\beta^2 (\omega + 1)}{n}\sum_{i=1}^n \norm{\nabla f_i(z^{t+1}) - h_i^t}^2} \\
    &=\left(1 - 2 \beta + \beta^2 (\omega + 1)\right)\ExpSub{t}{\frac{1}{n}\sum_{i=1}^n \norm{h^{t}_i - \nabla f_i(z^{t+1})}^2}. \\
  \end{align*}
  Since $\beta \leq \nicefrac{1}{\omega + 1},$ we have
  \begin{align*}
    &\ExpSub{t}{\frac{1}{n}\sum_{i=1}^n \norm{h_i^{t+1} - \nabla f_i(z^{t+1})}^2} \leq\left(1 - \beta\right)\ExpSub{t}{\frac{1}{n}\sum_{i=1}^n \norm{h^{t}_i - \nabla f_i(z^{t+1})}^2}. \\
  \end{align*}
  Next, we use the definition of $z^{t+1}$ and obtain
  \begin{align*}
    &\ExpSub{t}{\frac{1}{n}\sum_{i=1}^n \norm{h_i^{t+1} - \nabla f_i(z^{t+1})}^2} \\
    &\leq\left(1 - \beta\right) p \ExpSub{t}{\frac{1}{n}\sum_{i=1}^n \norm{h^{t}_i - \nabla f_i(x^{t+1})}^2} + \left(1 - \beta\right) (1 - p) \frac{1}{n}\sum_{i=1}^n \norm{h^{t}_i - \nabla f_i(z^{t})}^2 \\
    &\overset{\eqref{eq:young}}{\leq}\left(1 + \frac{2 p}{\beta}\right)\left(1 - \beta\right) p \ExpSub{t}{\frac{1}{n}\sum_{i=1}^n \norm{f_i(z^t) - \nabla f_i(x^{t+1})}^2} + \left(1 + \frac{\beta}{2 p}\right)\left(1 - \beta\right) p \frac{1}{n}\sum_{i=1}^n \norm{h^{t}_i - \nabla f_i(z^t)}^2 \\
    &\quad + \left(1 - \beta\right) (1 - p) \ExpSub{t}{\frac{1}{n}\sum_{i=1}^n \norm{h^{t}_i - \nabla f_i(z^{t})}^2} \\
    &=\left(1 + \frac{2 p}{\beta}\right) \left(1 - \beta\right) p \ExpSub{t}{\frac{1}{n}\sum_{i=1}^n \norm{f_i(z^t) - \nabla f_i(x^{t+1})}^2} + \left(1 - \beta\right) \left(1 + \frac{\beta}{2}\right) \frac{1}{n}\sum_{i=1}^n \norm{h^{t}_i - \nabla f_i(z^t)}^2.
  \end{align*}
  Using $1 - \beta \leq 1,$ \eqref{eq:young_2} and \eqref{eq:ineq1}, we get
  \begin{align*}
    &\ExpSub{t}{\frac{1}{n}\sum_{i=1}^n \norm{h_i^{t+1} - \nabla f_i(z^{t+1})}^2} \\
    &\leq 4 p \left(1 + \frac{p}{\beta}\right) \frac{1}{n}\sum_{i=1}^n \norm{f_i(z^t) - \nabla f_i(y^{t+1})}^2 + 4 p \left(1 + \frac{p}{\beta}\right) \ExpSub{t}{\frac{1}{n}\sum_{i=1}^n \norm{f_i(x^{t+1}) - \nabla f_i(y^{t+1})}^2} \\
    &\quad + \left(1 - \frac{\beta}{2}\right) \frac{1}{n}\sum_{i=1}^n \norm{h^{t}_i - \nabla f_i(z^t)}^2.
  \end{align*}
  From Assumptions~\ref{ass:workers_lipschitz_constant} and \ref{ass:convex} and Lemma~\ref{lemma:lipt_func}, we obtain
  \begin{align*}
    &\ExpSub{t}{\frac{1}{n}\sum_{i=1}^n \norm{h_i^{t+1} - \nabla f_i(z^{t+1})}^2} \\
    &\leq 8 p \left(1 + \frac{p}{\beta}\right) L_{\max} \left(f(z^t) - f(y^{t+1}) - \inp{\nabla f(y^{t+1})}{z^t - y^{t+1}}\right) + 4 p \left(1 + \frac{p}{\beta}\right) \widehat{L}^2 \ExpSub{t}{\norm{x^{t+1} - y^{t+1}}^2} \\
    &\quad + \left(1 - \frac{\beta}{2}\right) \frac{1}{n}\sum_{i=1}^n \norm{h^{t}_i - \nabla f_i(z^t)}^2 \\
    &= 8 p \left(1 + \frac{p}{\beta}\right) L_{\max} D_f(z^t, y^{t+1}) + 4 p \left(1 + \frac{p}{\beta}\right) \widehat{L}^2 \ExpSub{t}{\norm{x^{t+1} - y^{t+1}}^2} + \left(1 - \frac{\beta}{2}\right) \frac{1}{n}\sum_{i=1}^n \norm{h^{t}_i - \nabla f_i(z^t)}^2.
  \end{align*}
\end{proof}

\begin{lemma}
  \label{lemma:bi_fast_diana_iterates}
  Suppose that Assumptions~\ref{ass:lipschitz_constant}, \ref{ass:workers_lipschitz_constant}, \ref{ass:convex} and \ref{ass:unbiased_compressors} hold. Then, for Algorithm~\ref{alg:bi_diana}, the following inequality holds:
  \begin{equation}
  \begin{aligned}
    \label{eq:bi_fast_diana_iterates}
    &\ExpSub{t}{\norm{w^{t+1} - u^{t+1}}^2} \leq \left(1 - \frac{\alpha}{2}\right) \norm{w^t - u^t}^2  + \frac{4}{\alpha} \left(\frac{\gamma_{t+1}}{\bar{L} + \Gamma_{t+1} \mu}\right)^2 \norm{k^t - \nabla f(z^{t})}^2 \\
  &\quad + \frac{2 \omega}{n} \left(\frac{\gamma_{t+1}}{\bar{L} + \Gamma_{t+1} \mu}\right)^2\frac{1}{n}\sum_{i=1}^n \norm{\nabla f_i(z^{t}) - h_i^t}^2 + \left(\frac{\gamma_{t+1}}{\bar{L} + \Gamma_{t+1} \mu}\right)^2\left(\frac{4 \omega L_{\max}}{n}  + \frac{8L}{\alpha}\right)D_f(z^t, y^{t+1}).
  \end{aligned}
  \end{equation}
\end{lemma}
\begin{proof}
  Using the definition of $w^{t+1}$ and Definition~\ref{def:biased_compression}, we get the following inequality
\begin{align*}
  &\ExpSub{t}{\norm{w^{t+1} - u^{t+1}}^2} =\ExpSub{t}{\norm{\cC^{P}\left(u^{t+1} - q^{t+1}\right) - (u^{t+1} - q^{t+1})}^2} \leq (1 - \alpha)\ExpSub{t}{\norm{u^{t+1} - q^{t+1}}^2}.
\end{align*}
We can find the analytical formulas for $u^{t+1}$ and $q^{t+1},$ and obtain that
\begin{align*}
  u^{t+1} = \frac{\bar{L} + \Gamma_{t} \mu}{\bar{L} + \Gamma_{t+1} \mu} u^t + \frac{\mu \gamma_{t+1}}{\bar{L} + \Gamma_{t+1} \mu} y^{t+1} - \frac{\gamma_{t+1}}{\bar{L} + \Gamma_{t+1} \mu} g^{t+1}
\end{align*}
and 
\begin{align*}
  q^{t+1} = \frac{\bar{L} + \Gamma_{t} \mu}{\bar{L} + \Gamma_{t+1} \mu} w^t + \frac{\mu \gamma_{t+1}}{\bar{L} + \Gamma_{t+1} \mu} y^{t+1} - \frac{\gamma_{t+1}}{\bar{L} + \Gamma_{t+1} \mu} k^t.
\end{align*}
Therefore,
\begin{align*}
  &\ExpSub{t}{\norm{w^{t+1} - u^{t+1}}^2} \\
  &\leq (1 - \alpha)\ExpSub{t}{\norm{\frac{\bar{L} + \Gamma_{t} \mu}{\bar{L} + \Gamma_{t+1} \mu} (w^t - u^t) - \frac{\gamma_{t+1}}{\bar{L} + \Gamma_{t+1} \mu} (k^t - g^{t+1})}^2} \\
  &\overset{\eqref{eq:vardecomp}}{=} (1 - \alpha)\norm{\frac{\bar{L} + \Gamma_{t} \mu}{\bar{L} + \Gamma_{t+1} \mu} (w^t - u^t) - \frac{\gamma_{t+1}}{\bar{L} + \Gamma_{t+1} \mu} (k^t - \nabla f(y^{t+1}))}^2 \\
  &\quad + (1 - \alpha)\left(\frac{\gamma_{t+1}}{\bar{L} + \Gamma_{t+1} \mu}\right)^2\ExpSub{t}{\norm{g^{t+1} - \nabla f(y^{t+1})}^2} \\
  &\overset{\eqref{eq:young}}{\leq} \left(1 - \frac{\alpha}{2}\right) \left(\frac{\bar{L} + \Gamma_{t} \mu}{\bar{L} + \Gamma_{t+1} \mu}\right)^2\norm{w^t - u^t}^2 + \frac{2}{\alpha} \left(\frac{\gamma_{t+1}}{\bar{L} + \Gamma_{t+1} \mu}\right)^2 \norm{k^t - \nabla f(y^{t+1})}^2 \\
  &\quad + (1 - \alpha)\left(\frac{\gamma_{t+1}}{\bar{L} + \Gamma_{t+1} \mu}\right)^2\ExpSub{t}{\norm{g^{t+1} - \nabla f(y^{t+1})}^2} \\
  &\overset{\eqref{eq:young_2}}{\leq} \left(1 - \frac{\alpha}{2}\right) \left(\frac{\bar{L} + \Gamma_{t} \mu}{\bar{L} + \Gamma_{t+1} \mu}\right)^2\norm{w^t - u^t}^2 \\
  &\quad + \frac{4}{\alpha} \left(\frac{\gamma_{t+1}}{\bar{L} + \Gamma_{t+1} \mu}\right)^2 \norm{k^t - \nabla f(z^{t})}^2 \\
  &\quad + \frac{4}{\alpha} \left(\frac{\gamma_{t+1}}{\bar{L} + \Gamma_{t+1} \mu}\right)^2 \norm{\nabla f(z^t) - \nabla f(y^{t+1})}^2 \\
  &\quad + (1 - \alpha)\left(\frac{\gamma_{t+1}}{\bar{L} + \Gamma_{t+1} \mu}\right)^2\ExpSub{t}{\norm{g^{t+1} - \nabla f(y^{t+1})}^2}.
\end{align*}
One can substitute \eqref{eq:nabla_g_f} from Lemma~\ref{lemma:aux_stoch_grad} to the last inequality and get
\begin{align*}
  &\ExpSub{t}{\norm{w^{t+1} - u^{t+1}}^2} \\
  &\leq \left(1 - \frac{\alpha}{2}\right) \left(\frac{\bar{L} + \Gamma_{t} \mu}{\bar{L} + \Gamma_{t+1} \mu}\right)^2\norm{w^t - u^t}^2 \\
  &\quad + \frac{4}{\alpha} \left(\frac{\gamma_{t+1}}{\bar{L} + \Gamma_{t+1} \mu}\right)^2 \norm{k^t - \nabla f(z^{t})}^2 \\
  &\quad + \frac{4}{\alpha} \left(\frac{\gamma_{t+1}}{\bar{L} + \Gamma_{t+1} \mu}\right)^2 \norm{\nabla f(z^t) - \nabla f(y^{t+1})}^2 \\
  &\quad + (1 - \alpha)\left(\frac{\gamma_{t+1}}{\bar{L} + \Gamma_{t+1} \mu}\right)^2\left(\frac{2 \omega}{n^2} \sum_{i=1}^n \norm{\nabla f_i(z^{t}) - h_i^t}^2 + \frac{4 \omega L_{\max}}{n} D_f(z^t, y^{t+1})\right).
\end{align*}
Using Lemma~\ref{lemma:lipt_func} and $1 - \alpha \leq 1$, we obtain
\begin{align*}
  &\ExpSub{t}{\norm{w^{t+1} - u^{t+1}}^2} \\
  &\leq \left(1 - \frac{\alpha}{2}\right) \left(\frac{\bar{L} + \Gamma_{t} \mu}{\bar{L} + \Gamma_{t+1} \mu}\right)^2\norm{w^t - u^t}^2 \\
  &\quad + \frac{4}{\alpha} \left(\frac{\gamma_{t+1}}{\bar{L} + \Gamma_{t+1} \mu}\right)^2 \norm{k^t - \nabla f(z^{t})}^2 \\
  &\quad + \frac{2 \omega}{n} \left(\frac{\gamma_{t+1}}{\bar{L} + \Gamma_{t+1} \mu}\right)^2\frac{1}{n}\sum_{i=1}^n \norm{\nabla f_i(z^{t}) - h_i^t}^2 + \left(\frac{\gamma_{t+1}}{\bar{L} + \Gamma_{t+1} \mu}\right)^2\left(\frac{4 \omega L_{\max}}{n}  + \frac{8L}{\alpha}\right)D_f(z^t, y^{t+1}) \\
  &\leq \left(1 - \frac{\alpha}{2}\right) \norm{w^t - u^t}^2 \\
  &\quad + \frac{4}{\alpha} \left(\frac{\gamma_{t+1}}{\bar{L} + \Gamma_{t+1} \mu}\right)^2 \norm{k^t - \nabla f(z^{t})}^2 \\
  &\quad + \frac{2 \omega}{n} \left(\frac{\gamma_{t+1}}{\bar{L} + \Gamma_{t+1} \mu}\right)^2\frac{1}{n}\sum_{i=1}^n \norm{\nabla f_i(z^{t}) - h_i^t}^2 + \left(\frac{\gamma_{t+1}}{\bar{L} + \Gamma_{t+1} \mu}\right)^2\left(\frac{4 \omega L_{\max}}{n}  + \frac{8L}{\alpha}\right)D_f(z^t, y^{t+1}),
\end{align*}
where we use that $\Gamma_{t+1} \geq  \Gamma_{t}$ for all $t \geq 0.$
\end{proof}

\begin{lemma}
  \label{lemma:bi_fast_diana_k}
  Suppose that Assumptions~\ref{ass:lipschitz_constant} and \ref{ass:convex} hold.
   Then, for Algorithm~\ref{alg:bi_diana}, the following inequality holds:
   \begin{equation}
   \begin{aligned}
    &\ExpSub{t}{\norm{k^{t+1} - \nabla f(z^{t+1})}^2} \\
    &\leq 2 p\ExpSub{t}{\norm{v^{t} - \nabla f(z^{t})}^2} + 8 p L D_f(z^{t}, y^{t+1}) + 4 p L^2 \ExpSub{t}{\norm{x^{t+1} - y^{t+1}}^2} + (1 - p)\norm{k^{t} - \nabla f(z^{t})}^2.
    \label{eq:bi_fast_diana_k}
  \end{aligned}
  \end{equation}
\end{lemma}
\begin{proof}
  Note that $k^{t+1}$ and $z^{t+1}$ are coupled by the same random variable $c^t.$ Therefore,
  \begin{align*}
    &\ExpSub{t}{\norm{k^{t+1} - \nabla f(z^{t+1})}^2} \\
    &=p\ExpSub{t}{\norm{v^{t} - \nabla f(x^{t+1})}^2} + (1 - p)\norm{k^{t} - \nabla f(z^{t})}^2 \\
    &\overset{\eqref{eq:young_2}}{\leq}2 p\ExpSub{t}{\norm{v^{t} - \nabla f(z^{t})}^2} + 4 p \norm{\nabla f(z^{t}) - \nabla f(y^{t+1})}^2 + 4 p\ExpSub{t}{\norm{\nabla f(y^{t+1}) - \nabla f(x^{t+1})}^2} \\
    &\quad + (1 - p)\norm{k^{t} - \nabla f(z^{t})}^2 \\
    &\leq 2 p\ExpSub{t}{\norm{v^{t} - \nabla f(z^{t})}^2} + 8 p L D_f(z^{t}, y^{t+1}) + 4 p L^2 \ExpSub{t}{\norm{x^{t+1} - y^{t+1}}^2} + (1 - p)\norm{k^{t} - \nabla f(z^{t})}^2,
  \end{align*}
  where we use Assumptions~\ref{ass:lipschitz_constant} and Lemma~\ref{lemma:lipt_func}.
\end{proof}

\begin{lemma}
  \label{lemma:bi_fast_diana_v}
  Suppose that Assumptions~\ref{ass:lipschitz_constant}, \ref{ass:workers_lipschitz_constant} and \ref{ass:convex} hold. The momentum $\tau \in (0, 1]$ and the probability $p \in (0, 1].$
   Then, for Algorithm~\ref{alg:bi_diana}, the following inequality holds:
   \begin{equation}
   \begin{aligned}
    \label{eq:bi_fast_diana_v}
    &\ExpSub{t}{\norm{v^{t+1} - \nabla f(z^{t+1})}^2} \\
    &\leq \left(1 - \frac{\tau}{2}\right) \norm{v^{t} - \nabla f(z^{t})}^2 + \frac{2 \tau^2 \omega}{n^2} \sum_{i=1}^n \norm{h^t_i - \nabla f_i(z^{t})}^2 \\
    &\quad + \left(4 p \left(1 + \frac{2 p}{\tau}\right) L + \frac{8 p \tau^2 \omega L_{\max}}{n}\right)D_f(z^t, y^{t+1}) + \left(2 p \left(1 + \frac{2 p}{\tau}\right) L^2 + \frac{4 p \tau^2 \omega \widehat{L}^2}{n}\right)\ExpSub{t}{\norm{x^{t+1} - y^{t+1}}^2}.
  \end{aligned}
  \end{equation}
\end{lemma}

\begin{proof}
Using the definition of $v^{t+1},$ we get
\begin{align*}
  \ExpSub{t}{\norm{v^{t+1} - \nabla f(z^{t+1})}^2} = \ExpSub{t}{\norm{(1 - \tau) v^{t} + \tau\left(h^t + \frac{1}{n} \sum_{i=1}^n \cC_i^{D, z}\left(\nabla f_i(z^{t+1}) - h^t_i\right)\right) - \nabla f(z^{t+1})}^2}.
\end{align*}
Assumption~\ref{ass:unbiased_compressors}, including the independence and the unbiasedness of the compressors, insures that
\begin{align*}
  &\ExpSub{t}{\norm{v^{t+1} - \nabla f(z^{t+1})}^2} \\
  &\overset{\eqref{eq:vardecomp}}{=}(1 - \tau)^2 \ExpSub{t}{\norm{v^{t} - \nabla f(z^{t+1})}^2} + \tau^2\ExpSub{t}{\norm{\frac{1}{n} \sum_{i=1}^n \cC_i^{D, z}\left(\nabla f_i(z^{t+1}) - h^t_i\right) - (\nabla f(z^{t+1}) - h^t)}^2} \\
  &= (1 - \tau)^2 \ExpSub{t}{\norm{v^{t} - \nabla f(z^{t+1})}^2} + \frac{\tau^2}{n^2} \sum_{i=1}^n \ExpSub{t}{\norm{\cC_i^{D, z}\left(\nabla f_i(z^{t+1}) - h^t_i\right) - (\nabla f(z^{t+1}) - h^t)}^2} \\
  &\leq (1 - \tau)^2 \ExpSub{t}{\norm{v^{t} - \nabla f(z^{t+1})}^2} + \frac{\tau^2 \omega}{n^2} \sum_{i=1}^n \ExpSub{t}{\norm{\nabla f_i(z^{t+1}) - h^t_i}^2}.
\end{align*}
Using the definition of $z^{t+1}$, we have
\begin{align*}
  &\ExpSub{t}{\norm{v^{t+1} - \nabla f(z^{t+1})}^2} \\
  &\leq (1 - \tau)^2 (1 - p) \norm{v^{t} - \nabla f(z^{t})}^2 + (1 - \tau)^2 p \ExpSub{t}{\norm{v^{t} - \nabla f(x^{t+1})}^2} \\
  &\quad + \frac{(1 - p) \tau^2 \omega}{n^2} \sum_{i=1}^n \norm{h^t_i - \nabla f_i(z^{t})}^2 + \frac{p \tau^2 \omega}{n^2} \sum_{i=1}^n \ExpSub{t}{\norm{h^t_i - \nabla f_i(x^{t+1})}^2} \\
  &\overset{\eqref{eq:young}, \eqref{eq:young_2}}{\leq} (1 - \tau)^2 (1 - p) \norm{v^{t} - \nabla f(z^{t})}^2 \\
  &\quad + (1 - \tau)^2 \left(1 + \frac{\tau}{2 p}\right) p \norm{v^{t} - \nabla f(z^{t})}^2 + (1 - \tau)^2 \left(1 + \frac{2 p}{\tau}\right) p \ExpSub{t}{\norm{\nabla f(z^{t}) - \nabla f(x^{t+1})}^2} \\
  &\quad + \frac{(1 - p) \tau^2 \omega}{n^2} \sum_{i=1}^n \norm{h^t_i - \nabla f_i(z^{t})}^2 \\
  &\quad + \frac{2 p \tau^2 \omega}{n^2} \sum_{i=1}^n \ExpSub{t}{\norm{h^t_i - \nabla f_i(z^t)}^2} + \frac{2 p \tau^2 \omega}{n^2} \sum_{i=1}^n \ExpSub{t}{\norm{\nabla f_i(z^t) - \nabla f_i(x^{t+1})}^2} \\
  &= (1 - \tau)^2 \left(1 + \frac{\tau}{2}\right) \norm{v^{t} - \nabla f(z^{t})}^2 + (1 - \tau)^2 \left(1 + \frac{2 p}{\tau}\right) p \ExpSub{t}{\norm{\nabla f(z^{t}) - \nabla f(x^{t+1})}^2} \\
  &\quad + \frac{(1 + p) \tau^2 \omega}{n^2} \sum_{i=1}^n \norm{h^t_i - \nabla f_i(z^{t})}^2 + \frac{2 p \tau^2 \omega}{n^2} \sum_{i=1}^n \ExpSub{t}{\norm{\nabla f_i(z^t) - \nabla f_i(x^{t+1})}^2}.
\end{align*}
Using $0 \leq 1 - \tau \leq 1,$ $p \in (0, 1]$ and \eqref{eq:ineq1}, we get
\begin{align*}
  &\ExpSub{t}{\norm{v^{t+1} - \nabla f(z^{t+1})}^2} \\
  &\leq \left(1 - \frac{\tau}{2}\right) \norm{v^{t} - \nabla f(z^{t})}^2 + \left(1 + \frac{2 p}{\tau}\right) p \ExpSub{t}{\norm{\nabla f(z^{t}) - \nabla f(x^{t+1})}^2} \\
  &\quad + \frac{2 \tau^2 \omega}{n^2} \sum_{i=1}^n \norm{h^t_i - \nabla f_i(z^{t})}^2 + \frac{2 p \tau^2 \omega}{n^2} \sum_{i=1}^n \ExpSub{t}{\norm{\nabla f_i(z^t) - \nabla f_i(x^{t+1})}^2} \\
  &\overset{\eqref{eq:young_2}}{\leq} \left(1 - \frac{\tau}{2}\right) \norm{v^{t} - \nabla f(z^{t})}^2 \\
  &\quad + 2 p \left(1 + \frac{2 p}{\tau}\right) \norm{\nabla f(z^{t}) - \nabla f(y^{t+1})}^2 + 2 p \left(1 + \frac{2 p}{\tau}\right) \ExpSub{t}{\norm{\nabla f(y^{t+1}) - \nabla f(x^{t+1})}^2} \\
  &\quad + \frac{2 \tau^2 \omega}{n^2} \sum_{i=1}^n \norm{h^t_i - \nabla f_i(z^{t})}^2 \\
  &\quad + \frac{4 p \tau^2 \omega}{n^2} \sum_{i=1}^n \norm{\nabla f_i(z^t) - \nabla f_i(y^{t+1})}^2 + \frac{4 p \tau^2 \omega}{n^2} \sum_{i=1}^n \ExpSub{t}{\norm{\nabla f_i(x^{t+1}) - \nabla f_i(y^{t+1})}^2}.
\end{align*}
It is left to use Assumptions~\ref{ass:lipschitz_constant} and \ref{ass:workers_lipschitz_constant} with Lemma~\ref{lemma:lipt_func} to obtain
\begin{align*}
  &\ExpSub{t}{\norm{v^{t+1} - \nabla f(z^{t+1})}^2} \\
  &\leq \left(1 - \frac{\tau}{2}\right) \norm{v^{t} - \nabla f(z^{t})}^2 \\
  &\quad + 4 p \left(1 + \frac{2 p}{\tau}\right) L D_f(z^t, y^{t+1}) + 2 p \left(1 + \frac{2 p}{\tau}\right) L^2 \ExpSub{t}{\norm{y^{t+1} - x^{t+1}}^2} \\
  &\quad + \frac{2 \tau^2 \omega}{n^2} \sum_{i=1}^n \norm{h^t_i - \nabla f_i(z^{t})}^2 \\
  &\quad + \frac{8 p \tau^2 \omega L_{\max}}{n} D_f(z^t, y^{t+1}) + \frac{4 p \tau^2 \omega \widehat{L}^2}{n} \ExpSub{t}{\norm{x^{t+1} - y^{t+1}}^2} \\
  &= \left(1 - \frac{\tau}{2}\right) \norm{v^{t} - \nabla f(z^{t})}^2 + \frac{2 \tau^2 \omega}{n^2} \sum_{i=1}^n \norm{h^t_i - \nabla f_i(z^{t})}^2 \\
  &\quad + \left(4 p \left(1 + \frac{2 p}{\tau}\right) L + \frac{8 p \tau^2 \omega L_{\max}}{n}\right)D_f(z^t, y^{t+1}) + \left(2 p \left(1 + \frac{2 p}{\tau}\right) L^2 + \frac{4 p \tau^2 \omega \widehat{L}^2}{n}\right)\ExpSub{t}{\norm{x^{t+1} - y^{t+1}}^2}. \\
\end{align*}
\end{proof}

\subsection{Main theorem}
\newcommand{\liptconstraintconst}{660508}
\newcommand{\liptconstraint}{\bar{L} = \liptconstraintconst \times \max\left\{\frac{L}{\alpha}, \frac{L p}{\alpha \tau}, \frac{\sqrt{L L_{\max}} p \sqrt{\omega \tau}}{\alpha \beta \sqrt{n}}, \frac{\sqrt{L L_{\max}} \sqrt{p} \sqrt{\omega \tau}}{\alpha \sqrt{\beta} \sqrt{n}}, \frac{L_{\max} \omega p^2}{\beta^2 n}, \frac{L_{\max} \omega}{n}\right\}}
\begin{theorem}
  \label{theorem:main_theorem}
  Suppose that Assumptions~\ref{ass:lipschitz_constant}, \ref{ass:workers_lipschitz_constant}, \ref{ass:convex}, \ref{ass:unbiased_compressors} hold. Let
\begin{align}
  \label{eq:constr_lipts}
  \liptconstraint,
\end{align}
$\beta \leq \frac{1}{\omega + 1},$ and $\theta_{\min} = \frac{1}{4}\min\left\{1, \frac{\alpha}{p}, \frac{\tau}{p}, \frac{\beta}{p}\right\}.$
For all $t \geq 0,$ Algorithm~\ref{alg:bi_diana} guarantees that
\begin{equation}
\begin{aligned}
  \label{eq:main_theorem}
  &\Gamma_{t+1}\left(\Exp{f(z^{t+1}) - f(x^*)} + \kappa \Exp{\frac{1}{n}\sum_{i=1}^n \norm{h_i^{t+1} - \nabla f_i(z^{t+1})}^2} + \nu_{t+1} \Exp{\norm{w^{t+1} - u^{t+1}}^2}\right.\\
    &\quad  \left.+ \rho \Exp{\norm{k^{t+1} - \nabla f(z^{t+1})}^2} + \lambda \Exp{\norm{v^{t+1} - \nabla f(z^{t+1})}^2}\right) + \frac{\bar{L} + \Gamma_{t+1} \mu}{2} \Exp{\norm{u^{t+1} - x^*}^2}\\
    &\leq \Gamma_{t} \Bigg( \Exp{f(z^{t}) - f(x^*)} + \kappa \Exp{\frac{1}{n}\sum_{i=1}^n \norm{h_i^t - \nabla f_i(z^{t})}^2}  + \nu_{t} \Exp{\norm{w^t - u^t}^2} \\
    &\qquad\qquad + \rho \Exp{\norm{k^{t} - \nabla f(z^t)}^2} + \lambda \Exp{\norm{v^{t} - \nabla f(z^t)}^2} \Bigg) + \frac{\bar{L} + \Gamma_{t} \mu}{2} \Exp{\norm{u^t - x^*}^2}
\end{aligned}
\end{equation}
for some $\kappa,$ $\rho,$ $\lambda,$ $\nu_{t} \geq 0.$
\end{theorem}

\begin{proof}
  We fix some constants $\kappa,$ $\rho,$ $\lambda,$ $\nu_{t} \geq 0$ for all $t \geq 0$ that we define later. By combining Lemma~\ref{lemma:aux_first} with $\kappa \times \eqref{eq:aux_control}$ from Lemma~\ref{lemma:aux_control}, $\nu_{t} \times \eqref{eq:bi_fast_diana_iterates}$ from Lemma~\ref{lemma:bi_fast_diana_iterates}, $\rho \times \eqref{eq:bi_fast_diana_k}$ from Lemma~\ref{lemma:bi_fast_diana_k}, and $\lambda \times \eqref{eq:bi_fast_diana_v}$ from Lemma~\ref{lemma:bi_fast_diana_v}, we get the following inequality:
  \begin{align*}
    &\ExpSub{t}{f(z^{t+1}) - f(x^*)} + \kappa \ExpSub{t}{\frac{1}{n}\sum_{i=1}^n \norm{h_i^{t+1} - \nabla f_i(z^{t+1})}^2} + \nu_{t} \ExpSub{t}{\norm{w^{t+1} - u^{t+1}}^2}\\
    &\quad  + \rho \ExpSub{t}{\norm{k^{t+1} - \nabla f(z^{t+1})}^2} + \lambda \ExpSub{t}{\norm{v^{t+1} - \nabla f(z^{t+1})}^2}\\
    &\leq (1 - p\theta_{t+1}) \left(f(z^{t}) - f(x^*)\right) \\
    &\quad + \frac{2 p \omega}{n \bar{L}} \left(\frac{1}{n}\sum_{i=1}^n \norm{h_i^t - \nabla f_i(z^{t})}^2\right) \\
    &\quad+ p\theta_{t+1} \left(\frac{\bar{L} + \Gamma_{t} \mu}{2 \gamma_{t+1}} \norm{u^t - x^*}^2 - \frac{\bar{L} + \Gamma_{t+1} \mu}{2 \gamma_{t+1}} \ExpSub{t}{\norm{u^{t+1} - x^*}^2}\right) \\
    &\quad+ p\left(\frac{4 \omega L_{\max}}{n \bar{L}} + \theta_{t+1} - 1 \right) D_f(z^t, y^{t+1}) \\
    &\quad+ \frac{pL}{2} \ExpSub{t}{\norm{x^{t+1} - y^{t+1}}^2} - \frac{p \theta_{t+1}^2 \bar{L}}{2} \ExpSub{t}{\norm{u^{t+1} - u^t}^2} \\
    &\quad + \kappa \left(8 p \left(1 + \frac{p}{\beta}\right) L_{\max} D_f(z^t, y^{t+1}) + 4 p \left(1 + \frac{p}{\beta}\right) \widehat{L}^2 \ExpSub{t}{\norm{x^{t+1} - y^{t+1}}^2} + \left(1 - \frac{\beta}{2}\right) \frac{1}{n}\sum_{i=1}^n \norm{h^{t}_i - \nabla f_i(z^t)}^2\right) \\
    &\quad + \nu_{t} \left(\left(1 - \frac{\alpha}{2}\right) \norm{w^t - u^t}^2  + \frac{4}{\alpha} \left(\frac{\gamma_{t+1}}{\bar{L} + \Gamma_{t+1} \mu}\right)^2 \norm{k^t - \nabla f(z^{t})}^2\right. \\
    &\qquad\qquad \left.+ \frac{2 \omega}{n} \left(\frac{\gamma_{t+1}}{\bar{L} + \Gamma_{t+1} \mu}\right)^2\frac{1}{n}\sum_{i=1}^n \norm{\nabla f_i(z^{t}) - h_i^t}^2 + \left(\frac{\gamma_{t+1}}{\bar{L} + \Gamma_{t+1} \mu}\right)^2\left(\frac{4 \omega L_{\max}}{n}  + \frac{8L}{\alpha}\right)D_f(z^t, y^{t+1})\right) \\
    &\quad + \rho \left(2 p\ExpSub{t}{\norm{v^{t} - \nabla f(z^{t})}^2} + 8 p L D_f(z^{t}, y^{t+1}) + 4 p L^2 \ExpSub{t}{\norm{x^{t+1} - y^{t+1}}^2} + (1 - p)\norm{k^{t} - \nabla f(z^{t})}^2\right) \\
    &\quad + \lambda \left(\left(1 - \frac{\tau}{2}\right) \norm{v^{t} - \nabla f(z^{t})}^2 + \frac{2 \tau^2 \omega}{n^2} \sum_{i=1}^n \norm{h^t_i - \nabla f_i(z^{t})}^2 \right.\\
    &\qquad\qquad \left.+ \left(4 p \left(1 + \frac{2 p}{\tau}\right) L + \frac{8 p \tau^2 \omega L_{\max}}{n}\right)D_f(z^t, y^{t+1}) + \left(2 p \left(1 + \frac{2 p}{\tau}\right) L^2 + \frac{4 p \tau^2 \omega \widehat{L}^2}{n}\right)\ExpSub{t}{\norm{x^{t+1} - y^{t+1}}^2}\right).
  \end{align*}
  We regroup the terms and obtain
  \begin{align*}
    &\ExpSub{t}{f(z^{t+1}) - f(x^*)} + \kappa \ExpSub{t}{\frac{1}{n}\sum_{i=1}^n \norm{h_i^{t+1} - \nabla f_i(z^{t+1})}^2} + \nu_{t} \ExpSub{t}{\norm{w^{t+1} - u^{t+1}}^2}\\
    &\quad  + \rho \ExpSub{t}{\norm{k^{t+1} - \nabla f(z^{t+1})}^2} + \lambda \ExpSub{t}{\norm{v^{t+1} - \nabla f(z^{t+1})}^2}\\
    &\leq (1 - p\theta_{t+1}) \left(f(z^{t}) - f(x^*)\right) \\
    &\quad+ p\theta_{t+1} \left(\frac{\bar{L} + \Gamma_{t} \mu}{2 \gamma_{t+1}} \norm{u^t - x^*}^2 - \frac{\bar{L} + \Gamma_{t+1} \mu}{2 \gamma_{t+1}} \ExpSub{t}{\norm{u^{t+1} - x^*}^2}\right) \\
    &\quad+ p\left(\frac{4 \omega L_{\max}}{n \bar{L}} + \kappa 8 \left(1 + \frac{p}{\beta}\right) L_{\max} + \nu_{t} \left(\frac{\gamma_{t+1}}{\bar{L} + \Gamma_{t+1} \mu}\right)^2\left(\frac{4 \omega L_{\max}}{p n} + \frac{8L}{p \alpha}\right) + \rho 8 L +  \right. \\
    &\qquad\qquad+ \left. \lambda \left(4 \left(1 + \frac{2 p}{\tau}\right) L + \frac{8 \tau^2 \omega L_{\max}}{n}\right) + \theta_{t+1} - 1\right) D_f(z^t, y^{t+1}) \\
    &\quad+ \left(\frac{pL}{2} + \kappa 4 p \left(1 + \frac{p}{\beta}\right) \widehat{L}^2 + \rho 4 p L^2 + \lambda \left(2 p \left(1 + \frac{2 p}{\tau}\right) L^2 + \frac{4 p \tau^2 \omega \widehat{L}^2}{n}\right) \right) \ExpSub{t}{\norm{x^{t+1} - y^{t+1}}^2} \\
    &\quad - \frac{p \theta_{t+1}^2 \bar{L}}{2} \ExpSub{t}{\norm{u^{t+1} - u^t}^2} \\
    &\quad + \nu_{t} \left(1 - \frac{\alpha}{2}\right) \norm{w^t - u^t}^2 \\
    &\quad + \left(\nu_{t} \frac{4}{\alpha} \left(\frac{\gamma_{t+1}}{\bar{L} + \Gamma_{t+1} \mu}\right)^2 + \rho (1 - p)\right) \norm{k^{t} - \nabla f(z^t)}^2 \\
    &\quad + \left(\rho 2 p + \lambda \left(1 - \frac{\tau}{2}\right)\right) \norm{v^{t} - \nabla f(z^t)}^2 \\
    &\quad + \left(\frac{2 p \omega}{n \bar{L}} + \nu_{t} \frac{2 \omega}{n} \left(\frac{\gamma_{t+1}}{\bar{L} + \Gamma_{t+1} \mu}\right)^2 + \lambda \frac{2 \tau^2 \omega}{n} + \kappa \left(1 - \frac{\beta}{2}\right)\right) \left(\frac{1}{n}\sum_{i=1}^n \norm{h_i^t - \nabla f_i(z^{t})}^2\right).
  \end{align*}
  Using $x^{t+1} - y^{t+1} = \theta_{t+1} \left(u^{t+1} - w^t\right),$ we get (we mark the changes with color)
  \begin{align*}
    &\ExpSub{t}{f(z^{t+1}) - f(x^*)} + \kappa \ExpSub{t}{\frac{1}{n}\sum_{i=1}^n \norm{h_i^{t+1} - \nabla f_i(z^{t+1})}^2} + \nu_{t} \ExpSub{t}{\norm{w^{t+1} - u^{t+1}}^2}\\
    &\quad  + \rho \ExpSub{t}{\norm{k^{t+1} - \nabla f(z^{t+1})}^2} + \lambda \ExpSub{t}{\norm{v^{t+1} - \nabla f(z^{t+1})}^2}\\
    &\leq (1 - p\theta_{t+1}) \left(f(z^{t}) - f(x^*)\right) \\
    &\quad+ p\theta_{t+1} \left(\frac{\bar{L} + \Gamma_{t} \mu}{2 \gamma_{t+1}} \norm{u^t - x^*}^2 - \frac{\bar{L} + \Gamma_{t+1} \mu}{2 \gamma_{t+1}} \ExpSub{t}{\norm{u^{t+1} - x^*}^2}\right) \\
    &\quad+ p\left(\frac{4 \omega L_{\max}}{n \bar{L}} + \kappa 8 \left(1 + \frac{p}{\beta}\right) L_{\max} + \nu_{t} \left(\frac{\gamma_{t+1}}{\bar{L} + \Gamma_{t+1} \mu}\right)^2\left(\frac{4 \omega L_{\max}}{p n} + \frac{8L}{p \alpha}\right) + \rho 8 L +  \right. \\
    &\qquad\qquad+ \left. \lambda \left(4 \left(1 + \frac{2 p}{\tau}\right) L + \frac{8 \tau^2 \omega L_{\max}}{n}\right) + \theta_{t+1} - 1\right) D_f(z^t, y^{t+1}) \\
    &\quad+ \red{\theta_{t+1}^2 \left(\frac{pL}{2} + \kappa 4 p \left(1 + \frac{p}{\beta}\right) \widehat{L}^2 + \rho 4 p L^2 + \lambda \left(2 p \left(1 + \frac{2 p}{\tau}\right) L^2 + \frac{4 p \tau^2 \omega \widehat{L}^2}{n}\right) \right) \ExpSub{t}{\norm{u^{t+1} - w^t}^2}} \\
    &\quad - \frac{p \theta_{t+1}^2 \bar{L}}{2} \ExpSub{t}{\norm{u^{t+1} - u^t}^2} \\
    &\quad + \nu_{t} \left(1 - \frac{\alpha}{2}\right) \norm{w^t - u^t}^2 \\
    &\quad + \left(\nu_{t} \frac{4}{\alpha} \left(\frac{\gamma_{t+1}}{\bar{L} + \Gamma_{t+1} \mu}\right)^2 + \rho (1 - p)\right) \norm{k^{t} - \nabla f(z^t)}^2 \\
    &\quad + \left(\rho 2 p + \lambda \left(1 - \frac{\tau}{2}\right)\right) \norm{v^{t} - \nabla f(z^t)}^2 \\
    &\quad + \left(\frac{2 p \omega}{n \bar{L}} + \nu_{t} \frac{2 \omega}{n} \left(\frac{\gamma_{t+1}}{\bar{L} + \Gamma_{t+1} \mu}\right)^2 + \lambda \frac{2 \tau^2 \omega}{n} + \kappa \left(1 - \frac{\beta}{2}\right)\right) \left(\frac{1}{n}\sum_{i=1}^n \norm{h_i^t - \nabla f_i(z^{t})}^2\right).
  \end{align*}
  The inequality \eqref{eq:young} implies $\norm{u^{t+1} - w^t}^2 \leq 2 \norm{u^{t+1} - u^{t}}^2 + 2 \norm{u^{t} - w^{t}}^2$ and
  \begin{align*}
    &\ExpSub{t}{f(z^{t+1}) - f(x^*)} + \kappa \ExpSub{t}{\frac{1}{n}\sum_{i=1}^n \norm{h_i^{t+1} - \nabla f_i(z^{t+1})}^2} + \nu_{t} \ExpSub{t}{\norm{w^{t+1} - u^{t+1}}^2}\\
    &\quad  + \rho \ExpSub{t}{\norm{k^{t+1} - \nabla f(z^{t+1})}^2} + \lambda \ExpSub{t}{\norm{v^{t+1} - \nabla f(z^{t+1})}^2}\\
    &\leq (1 - p\theta_{t+1}) \left(f(z^{t}) - f(x^*)\right) \\
    &\quad+ p\theta_{t+1} \left(\frac{\bar{L} + \Gamma_{t} \mu}{2 \gamma_{t+1}} \norm{u^t - x^*}^2 - \frac{\bar{L} + \Gamma_{t+1} \mu}{2 \gamma_{t+1}} \ExpSub{t}{\norm{u^{t+1} - x^*}^2}\right) \\
    &\quad+ p\left(\frac{4 \omega L_{\max}}{n \bar{L}} + \kappa 8 \left(1 + \frac{p}{\beta}\right) L_{\max} + \nu_{t} \left(\frac{\gamma_{t+1}}{\bar{L} + \Gamma_{t+1} \mu}\right)^2\left(\frac{4 \omega L_{\max}}{p n} + \frac{8L}{p \alpha}\right) + \rho 8 L +  \right. \\
    &\qquad\qquad+ \left. \lambda \left(4 \left(1 + \frac{2 p}{\tau}\right) L + \frac{8 \tau^2 \omega L_{\max}}{n}\right) + \theta_{t+1} - 1\right) D_f(z^t, y^{t+1}) \\
    &\quad+ \red{2 \theta_{t+1}^2 \left(\frac{pL}{2} + \kappa 4 p \left(1 + \frac{p}{\beta}\right) \widehat{L}^2 + \rho 4 p L^2 + \lambda \left(2 p \left(1 + \frac{2 p}{\tau}\right) L^2 + \frac{4 p \tau^2 \omega \widehat{L}^2}{n}\right) \right) \ExpSub{t}{\norm{u^{t+1} - u^t}^2}} \\
    &\quad - \frac{p \theta_{t+1}^2 \bar{L}}{2} \ExpSub{t}{\norm{u^{t+1} - u^t}^2} \\
    &\quad + \red{\left(2 \theta_{t+1}^2 \left(\frac{pL}{2} + \kappa 4 p \left(1 + \frac{p}{\beta}\right) \widehat{L}^2 + \rho 4 p L^2 + \lambda \left(2 p \left(1 + \frac{2 p}{\tau}\right) L^2 + \frac{4 p \tau^2 \omega \widehat{L}^2}{n}\right) \right) + \nu_{t} \left(1 - \frac{\alpha}{2}\right)\right) \norm{w^t - u^t}^2} \\
    &\quad + \left(\nu_{t} \frac{4}{\alpha} \left(\frac{\gamma_{t+1}}{\bar{L} + \Gamma_{t+1} \mu}\right)^2 + \rho (1 - p)\right) \norm{k^{t} - \nabla f(z^t)}^2 \\
    &\quad + \left(\rho 2 p + \lambda \left(1 - \frac{\tau}{2}\right)\right) \norm{v^{t} - \nabla f(z^t)}^2 \\
    &\quad + \left(\frac{2 p \omega}{n \bar{L}} + \nu_{t} \frac{2 \omega}{n} \left(\frac{\gamma_{t+1}}{\bar{L} + \Gamma_{t+1} \mu}\right)^2 + \lambda \frac{2 \tau^2 \omega}{n} + \kappa \left(1 - \frac{\beta}{2}\right)\right) \left(\frac{1}{n}\sum_{i=1}^n \norm{h_i^t - \nabla f_i(z^{t})}^2\right).
  \end{align*}
  Now, we want to find appropriate $\kappa,$ $\rho,$ $\lambda,$ and $\nu_{t}$ such that
  \begin{equation}
  \begin{gathered}
    2 \theta_{t+1}^2 \left(\frac{pL}{2} + \kappa 4 p \left(1 + \frac{p}{\beta}\right) \widehat{L}^2 + \rho 4 p L^2 + \lambda \left(2 p \left(1 + \frac{2 p}{\tau}\right) L^2 + \frac{4 p \tau^2 \omega \widehat{L}^2}{n}\right) \right) + \nu_{t} \left(1 - \frac{\alpha}{2}\right) \leq \nu_{t} \left(1 - \frac{\alpha}{4}\right), \\
    \nu_{t} \frac{4}{\alpha} \left(\frac{\gamma_{t+1}}{\bar{L} + \Gamma_{t+1} \mu}\right)^2 + \rho (1 - p) \leq \rho \left(1 - \frac{p}{2}\right), \\
    \rho 2 p + \lambda \left(1 - \frac{\tau}{2}\right) \leq \lambda \left(1 - \frac{\tau}{4}\right), \\
    \frac{2 p \omega}{n \bar{L}} + \nu_{t} \frac{2 \omega}{n} \left(\frac{\gamma_{t+1}}{\bar{L} + \Gamma_{t+1} \mu}\right)^2 + \lambda \frac{2 \tau^2 \omega}{n} + \kappa \left(1 - \frac{\beta}{2}\right)\leq \kappa \left(1 - \frac{\beta}{4}\right).
  \end{gathered}
  \label{eq:parameters_task}
  \end{equation}
  We analyze inequalities \eqref{eq:parameters_task} in the following lemma:
  \begin{restatable}[First Symbolically Computed]{lemma}{LEMMAPARAMETERS}
    \label{lemma:parameters}
    Assume that for the parameter $\bar{L},$ the inequalities from Sections~\ref{sec:sym_comp_kapap_const} and \ref{sec:sym_comp_rho_const} hold. Then, for all $t \geq 0,$ exists $\rho$ in \eqref{eq:rho_sol}, $\kappa$ in \eqref{eq:kappa_sol}, $\lambda$ in \eqref{eq:lambda_sol}, and $\nu_{t}$ in \eqref{eq:nu_sol} such that \eqref{eq:parameters_task} holds.
  \end{restatable}
  We proof lemma separately in Section~\ref{sec:lemma_parameters}. Using the lemma, we have
  \begin{align*}
    &\ExpSub{t}{f(z^{t+1}) - f(x^*)} + \kappa \ExpSub{t}{\frac{1}{n}\sum_{i=1}^n \norm{h_i^{t+1} - \nabla f_i(z^{t+1})}^2} + \nu_{t} \ExpSub{t}{\norm{w^{t+1} - u^{t+1}}^2}\\
    &\quad  + \rho \ExpSub{t}{\norm{k^{t+1} - \nabla f(z^{t+1})}^2} + \lambda \ExpSub{t}{\norm{v^{t+1} - \nabla f(z^{t+1})}^2}\\
    &\leq (1 - p\theta_{t+1}) \left(f(z^{t}) - f(x^*)\right) \\
    &\quad+ p\theta_{t+1} \left(\frac{\bar{L} + \Gamma_{t} \mu}{2 \gamma_{t+1}} \norm{u^t - x^*}^2 - \frac{\bar{L} + \Gamma_{t+1} \mu}{2 \gamma_{t+1}} \ExpSub{t}{\norm{u^{t+1} - x^*}^2}\right) \\
    &\quad+ p\left(\frac{4 \omega L_{\max}}{n \bar{L}} + \kappa 8 \left(1 + \frac{p}{\beta}\right) L_{\max} + \nu_{t} \left(\frac{\gamma_{t+1}}{\bar{L} + \Gamma_{t+1} \mu}\right)^2\left(\frac{4 \omega L_{\max}}{p n} + \frac{8L}{p \alpha}\right) + \rho 8 L +  \right. \\
    &\qquad\qquad+ \left. \lambda \left(4 \left(1 + \frac{2 p}{\tau}\right) L + \frac{8 \tau^2 \omega L_{\max}}{n}\right) + \theta_{t+1} - 1\right) D_f(z^t, y^{t+1}) \\
    &\quad+ 2 \theta_{t+1}^2 \left(\frac{pL}{2} + \kappa 4 p \left(1 + \frac{p}{\beta}\right) \widehat{L}^2 + \rho 4 p L^2 + \lambda \left(2 p \left(1 + \frac{2 p}{\tau}\right) L^2 + \frac{4 p \tau^2 \omega \widehat{L}^2}{n}\right) \right) \ExpSub{t}{\norm{u^{t+1} - u^t}^2} \\
    &\quad - \frac{p \theta_{t+1}^2 \bar{L}}{2} \ExpSub{t}{\norm{u^{t+1} - u^t}^2} \\
    &\quad + \nu_{t} \left(1 - \frac{\alpha}{4}\right) \norm{w^t - u^t}^2 \\
    &\quad + \rho \left(1 - \frac{p}{2}\right) \norm{k^{t} - \nabla f(z^t)}^2 \\
    &\quad + \lambda \left(1 - \frac{\tau}{4}\right) \norm{v^{t} - \nabla f(z^t)}^2 \\
    &\quad + \kappa \left(1 - \frac{\beta}{4}\right) \left(\frac{1}{n}\sum_{i=1}^n \norm{h_i^t - \nabla f_i(z^{t})}^2\right).
  \end{align*}
  Let us separately analyze the terms w.r.t. $D_f(z^t, y^{t+1})$ and $\ExpSub{t}{\norm{u^{t+1} - u^t}^2}$:
  \begin{restatable}[Second Symbolically Computed]{lemma}{LEMMANEGATIVETERMS}
    \label{lemma:negative_parameters}
    Consider the parameters $\rho,$ $\kappa,$ $\lambda,$ and $\nu_{t}$ from Lemma~\ref{lemma:parameters}. Assume that for the parameter $\bar{L},$ the inequalities from Sections~\ref{sec:sym_comp_bregman_const} and \ref{sec:sym_comp_dist_const} hold, and the step size $\theta_{t+1} \leq \nicefrac{1}{4}$ for all $t \geq 0.$ Then, for all $t \geq 0,$ the following inequalities are satisfied:
    \begin{equation}
    \begin{aligned}
      \label{eq:lemma:negative_parameters:bregman}
      &p\left(\frac{4 \omega L_{\max}}{n \bar{L}} + \kappa 8 \left(1 + \frac{p}{\beta}\right) L_{\max} + \nu_{t} \left(\frac{\gamma_{t+1}}{\bar{L} + \Gamma_{t+1} \mu}\right)^2\left(\frac{4 \omega L_{\max}}{p n} + \frac{8L}{p \alpha}\right) + \rho 8 L +  \right. \\
    &\qquad\qquad+ \left. \lambda \left(4 \left(1 + \frac{2 p}{\tau}\right) L + \frac{8 \tau^2 \omega L_{\max}}{n}\right) + \theta_{t+1} - 1\right) D_f(z^t, y^{t+1}) \leq 0
    \end{aligned}
    \end{equation}
    and
    \begin{equation}
      \begin{aligned}
      \label{eq:lemma:negative_parameters:norm}
      &2 \theta_{t+1}^2 \left(\frac{pL}{2} + \kappa 4 p \left(1 + \frac{p}{\beta}\right) \widehat{L}^2 + \rho 4 p L^2 + \lambda \left(2 p \left(1 + \frac{2 p}{\tau}\right) L^2 + \frac{4 p \tau^2 \omega \widehat{L}^2}{n}\right) \right) \ExpSub{t}{\norm{u^{t+1} - u^t}^2} \\
    &\quad - \frac{p \theta_{t+1}^2 \bar{L}}{2} \ExpSub{t}{\norm{u^{t+1} - u^t}^2} \leq 0. \\
  \end{aligned}
\end{equation}
  \end{restatable}
  We prove Lemma~\ref{lemma:negative_parameters} in Section~\ref{sec:negative_parameters}.
  Using the lemma, we get
  \begin{align*}
    &\ExpSub{t}{f(z^{t+1}) - f(x^*)} + \kappa \ExpSub{t}{\frac{1}{n}\sum_{i=1}^n \norm{h_i^{t+1} - \nabla f_i(z^{t+1})}^2} + \nu_{t} \ExpSub{t}{\norm{w^{t+1} - u^{t+1}}^2}\\
    &\quad  + \rho \ExpSub{t}{\norm{k^{t+1} - \nabla f(z^{t+1})}^2} + \lambda \ExpSub{t}{\norm{v^{t+1} - \nabla f(z^{t+1})}^2}\\
    &\leq (1 - p\theta_{t+1}) \left(f(z^{t}) - f(x^*)\right) \\
    &\quad+ p\theta_{t+1} \left(\frac{\bar{L} + \Gamma_{t} \mu}{2 \gamma_{t+1}} \norm{u^t - x^*}^2 - \frac{\bar{L} + \Gamma_{t+1} \mu}{2 \gamma_{t+1}} \ExpSub{t}{\norm{u^{t+1} - x^*}^2}\right) \\
    &\quad + \nu_{t} \left(1 - \frac{\alpha}{4}\right) \norm{w^t - u^t}^2 \\
    &\quad + \rho \left(1 - \frac{p}{2}\right) \norm{k^{t} - \nabla f(z^t)}^2 \\
    &\quad + \lambda \left(1 - \frac{\tau}{4}\right) \norm{v^{t} - \nabla f(z^t)}^2 \\
    &\quad + \kappa \left(1 - \frac{\beta}{4}\right) \left(\frac{1}{n}\sum_{i=1}^n \norm{h_i^t - \nabla f_i(z^{t})}^2\right).
  \end{align*}
  Note that $0 \leq \nu_{t+1} \leq \nu_{t}$ for all $t \geq 0,$ since $\theta_{t+1}$ is a non-increasing sequence (see Lemma~\ref{lemma:learning_rates} and the definition of $\nu_{t}$ in \eqref{eq:nu_sol}). Using 
  $$\theta_{t+1} \leq \frac{1}{4}\min\left\{1, \frac{\alpha}{p}, \frac{\tau}{p}, \frac{\beta}{p}\right\}$$ for all $t \geq 0,$ we obtain
  \begin{align*}
    &\ExpSub{t}{f(z^{t+1}) - f(x^*)} + \kappa \ExpSub{t}{\frac{1}{n}\sum_{i=1}^n \norm{h_i^{t+1} - \nabla f_i(z^{t+1})}^2} + \nu_{t+1} \ExpSub{t}{\norm{w^{t+1} - u^{t+1}}^2}\\
    &\quad  + \rho \ExpSub{t}{\norm{k^{t+1} - \nabla f(z^{t+1})}^2} + \lambda \ExpSub{t}{\norm{v^{t+1} - \nabla f(z^{t+1})}^2}\\
    &\leq (1 - p\theta_{t+1}) \Bigg( f(z^{t}) - f(x^*) + \kappa \left(\frac{1}{n}\sum_{i=1}^n \norm{h_i^t - \nabla f_i(z^{t})}^2\right)  + \nu_{t} \norm{w^t - u^t}^2 \\
    &\qquad\qquad\qquad\qquad + \rho \norm{k^{t} - \nabla f(z^t)}^2 + \lambda \norm{v^{t} - \nabla f(z^t)}^2 \Bigg)\\
    &\quad+ p\theta_{t+1} \left(\frac{\bar{L} + \Gamma_{t} \mu}{2 \gamma_{t+1}} \norm{u^t - x^*}^2 - \frac{\bar{L} + \Gamma_{t+1} \mu}{2 \gamma_{t+1}} \ExpSub{t}{\norm{u^{t+1} - x^*}^2}\right).
  \end{align*}
  Let us multiply the inequality by $\frac{\gamma_{t+1}}{p\theta_{t+1}}:$
  \begin{align*}
    &\frac{\gamma_{t+1}}{p\theta_{t+1}}\left(\ExpSub{t}{f(z^{t+1}) - f(x^*)} + \kappa \ExpSub{t}{\frac{1}{n}\sum_{i=1}^n \norm{h_i^{t+1} - \nabla f_i(z^{t+1})}^2} + \nu_{t+1} \ExpSub{t}{\norm{w^{t+1} - u^{t+1}}^2}\right.\\
    &\quad  \left.+ \rho \ExpSub{t}{\norm{k^{t+1} - \nabla f(z^{t+1})}^2} + \lambda \ExpSub{t}{\norm{v^{t+1} - \nabla f(z^{t+1})}^2}\right)\\
    &\leq \left(\frac{\gamma_{t+1}}{p\theta_{t+1}} - \gamma_{t+1}\right) \Bigg( f(z^{t}) - f(x^*) + \kappa \left(\frac{1}{n}\sum_{i=1}^n \norm{h_i^t - \nabla f_i(z^{t})}^2\right)  + \nu_{t} \norm{w^t - u^t}^2 \\
    &\qquad\qquad\qquad\qquad + \rho \norm{k^{t} - \nabla f(z^t)}^2 + \lambda \norm{v^{t} - \nabla f(z^t)}^2 \Bigg)\\
    &\quad+ \left(\frac{\bar{L} + \Gamma_{t} \mu}{2} \norm{u^t - x^*}^2 - \frac{\bar{L} + \Gamma_{t+1} \mu}{2} \ExpSub{t}{\norm{u^{t+1} - x^*}^2}\right).
  \end{align*}
  It is left to use that $\Gamma_{t+1} \eqdef \Gamma_{t} + \gamma_{t+1}$ and $\gamma_{t+1} = p \theta_{t+1} \Gamma_{t+1}$ (see Lemma~\ref{lemma:learning_rates}) and take the full expectation to obtain \eqref{eq:main_theorem}.

  In the proof we require that, for the parameter $\bar{L},$ the inequalities from Sections~\ref{sec:sym_comp_kapap_const}, \ref{sec:sym_comp_rho_const}, \ref{sec:sym_comp_bregman_const} and \ref{sec:sym_comp_dist_const} hold. In Section~\ref{sec:sym_comp_check}, we show that these inequalities follow from \eqref{eq:constr_lipts}.
\end{proof}

\subsection{Strongly-convex case}
\begin{theorem}
  \label{theorem:main_theorem_strongly}
  Suppose that Assumptions~\ref{ass:lipschitz_constant}, \ref{ass:workers_lipschitz_constant}, \ref{ass:convex}, \ref{ass:unbiased_compressors} hold. Let
\begin{align}
  \liptconstraint,
\end{align}
$\beta = \frac{1}{\omega + 1},$ $\theta_{\min} = \frac{1}{4}\min\left\{1, \frac{\alpha}{p}, \frac{\tau}{p}, \frac{\beta}{p}\right\},$ $h_i^0 = \nabla f_i(z^{0})$ for all $i \in [n],$ $w^0 = u^0,$ $k^{0} = \nabla f(z^0),$ $v^{0} = \nabla f(z^0),$ and $\Gamma_0 \geq 1.$
Then Algorithm~\ref{alg:bi_diana} guarantees that
\begin{equation}
\begin{aligned}
  \label{eq:main_theorem_strongly}
  \Exp{f(z^{T}) - f(x^*)} + \frac{\mu}{2} \Exp{\norm{u^{T} - x^*}^2} \leq 2\exp\left(-\frac{T}{Q}\right)\left(\left(f(z^{0}) - f(x^*)\right) + \left(\frac{\bar{L}}{\Gamma_{0}} + \mu\right) \norm{u^0 - x^*}^2\right),
\end{aligned}
\end{equation}
where
\begin{align*}
  &Q \eqdef 2 \times \sqrt{\liptconstraintconst} \times \\
  &\max\Bigg\{\sqrt{\frac{L}{\alpha p \mu}},\sqrt{\frac{L}{\alpha \tau \mu}},\sqrt{\frac{\sqrt{L L_{\max}} (\omega + 1) \sqrt{\omega \tau}}{\alpha \sqrt{n} \mu}},\sqrt{\frac{\sqrt{L L_{\max}} \sqrt{\omega + 1} \sqrt{\omega \tau}}{\alpha \sqrt{p} \sqrt{n} \mu}},\sqrt{\frac{L_{\max} \omega (\omega + 1)^2 p}{n \mu}}, \sqrt{\frac{L_{\max} \omega}{n p \mu}}, \\
  & \qquad \quad \frac{1}{\alpha}, \frac{1}{\tau}, (\omega + 1), \frac{1}{p}\Bigg\}.
\end{align*}
\end{theorem}

\begin{remark}
  Up to a constant factor of $2$, one can see that the optimal $\Gamma_0$ in \eqref{eq:main_theorem_strongly} from Theorem~\ref{theorem:main_theorem_strongly} equals $\Gamma_0 = \nicefrac{\bar{L}}{\mu}.$ But the dependence on $\Gamma_0$ is under the logarithm, so if the dependence on the logarithm is not critical, one can take any $\Gamma_0 \geq 1.$
\end{remark}

\begin{proof}
All conditions from Theorem~\ref{theorem:main_theorem} are satisfied. Let us sum the inequality \eqref{eq:main_theorem} for $t = 0$ to $T - 1:$
\begin{align*}
  &\Gamma_{T}\left(\Exp{f(z^{T}) - f(x^*)} + \kappa \Exp{\frac{1}{n}\sum_{i=1}^n \norm{h_i^{T} - \nabla f_i(z^{T})}^2} + \nu_{T} \Exp{\norm{w^{T} - u^{T}}^2}\right.\\
    &\quad  \left.+ \rho \Exp{\norm{k^{T} - \nabla f(z^{T})}^2} + \lambda \Exp{\norm{v^{T} - \nabla f(z^{T})}^2}\right)\\
    &\leq \Gamma_{0} \Bigg( \Exp{f(z^{0}) - f(x^*)} + \kappa \Exp{\frac{1}{n}\sum_{i=1}^n \norm{h_i^0 - \nabla f_i(z^{0})}^2}  + \nu_{0} \Exp{\norm{w^0 - u^0}^2} \\
    &\qquad\qquad + \rho \Exp{\norm{k^{0} - \nabla f(z^0)}^2} + \lambda \Exp{\norm{v^{0} - \nabla f(z^0)}^2} \Bigg)\\
    &\quad+ \left(\frac{\bar{L} + \Gamma_{0} \mu}{2} \Exp{\norm{u^0 - x^*}^2} - \frac{\bar{L} + \Gamma_{T} \mu}{2} \Exp{\norm{u^{T} - x^*}^2}\right).
\end{align*}
Using the initial conditions and the non-negativity of the terms, we get
\begin{align*}
  &\Gamma_{T}\Exp{f(z^{T}) - f(x^*)} + \frac{\Gamma_{T} \mu}{2} \Exp{\norm{u^{T} - x^*}^2} \leq \Gamma_{0} \left(f(z^{0}) - f(x^*)\right) + \frac{\bar{L} + \Gamma_{0} \mu}{2} \norm{u^0 - x^*}^2.
\end{align*}
Using Lemma~\ref{lemma:learning_rates}, we have
\begin{align*}
  &\Exp{f(z^{T}) - f(x^*)} + \frac{\mu}{2} \Exp{\norm{u^{T} - x^*}^2} \\
  &\leq \exp\left(-T \min\left\{\sqrt{\frac{p \mu}{4 \bar{L}}}, p\theta_{\min}\right\}\right)\left(2 \left(f(z^{0}) - f(x^*)\right) + \left(\frac{\bar{L}}{\Gamma_{0}} + \mu\right) \norm{u^0 - x^*}^2\right).
\end{align*}
It is left to use the definitions of $\bar{L}$ and $\theta_{\min}.$
\end{proof}

\subsection{General convex case}
Let us use Theorem~\ref{theorem:main_theorem} to analyze the general convex case ($\mu$ can possibly be equal to zero):
\begin{theorem}
  \label{theorem:main_theorem_general_convex}
  Suppose that Assumptions~\ref{ass:lipschitz_constant}, \ref{ass:workers_lipschitz_constant}, \ref{ass:convex}, \ref{ass:unbiased_compressors} hold. Let
\begin{align}
  \liptconstraint,
\end{align}
$\beta = \frac{1}{\omega + 1},$ $\theta_{\min} = \frac{1}{4}\min\left\{1, \frac{\alpha}{p}, \frac{\tau}{p}, \frac{\beta}{p}\right\},$ $h_i^0 = \nabla f_i(z^{0})$ for all $i \in [n],$ $w^0 = u^0,$ $k^{0} = \nabla f(z^0),$ $v^{0} = \nabla f(z^0),$ and $\Gamma_0 \in [1, \nicefrac{\bar{L}}{L}].$
Then Algorithm~\ref{alg:bi_diana} returns $\varepsilon$-solution, i.e., $\Exp{f(z^{T})}- f(x^*) \leq \varepsilon,$ after
\begin{equation}
\begin{aligned}
  \label{eq:main_theorem_general_convex}
    T = 
    \begin{cases}
      \Theta\left(\frac{1}{p \theta_{\min}} \log \frac{\bar{L} \norm{z^0 - x^*}^2}{\Gamma_0 \varepsilon}\right), &\frac{\bar{L} \norm{z^0 - x^*}^2}{\varepsilon} < \frac{1}{p \theta_{\min}^2} \\
      \Theta\left(\max\left\{\frac{1}{p \theta_{\min}} \log \frac{1}{\Gamma_0 p \theta_{\min}^2}, 0\right\} + Q\sqrt{\frac{\norm{z^0 - x^*}^2}{\varepsilon}}\right), & \textnormal{otherwise}
    \end{cases}
\end{aligned}
\end{equation}
iterations,
where
$$Q \eqdef \Theta\left(\max\left\{\sqrt{\frac{L}{\alpha p}}, \sqrt{\frac{L}{\alpha \tau}}, \sqrt{\frac{\sqrt{L L_{\max}} (\omega + 1)\sqrt{\omega \tau}}{\alpha \sqrt{n}}}, \sqrt{\frac{\sqrt{L L_{\max}} \sqrt{\omega + 1} \sqrt{\omega \tau}}{\alpha \sqrt{p} \sqrt{n}}}, \sqrt{\frac{L_{\max} \omega (\omega + 1)^2 p}{n}}, \sqrt{\frac{L_{\max} \omega}{n p}}\right\}\right).$$
\end{theorem}

\begin{remark}
  One can see that the optimal $\Gamma_0$ in \eqref{eq:main_theorem_general_convex} from Theorem~\ref{theorem:main_theorem_general_convex} equals $\Gamma_0 = \nicefrac{\bar{L}}{L}.$ But the dependence on $\Gamma_0$ is under the logarithm, so if the dependence on the logarithm is not critical, one can take any $\Gamma_0 \in [1, \nicefrac{\bar{L}}{L}].$
\end{remark}

\begin{proof}
All conditions from Theorem~\ref{theorem:main_theorem} are satisfied. Let us sum the inequality \eqref{eq:main_theorem} for $t = 0$ to $T - 1:$
\begin{align*}
  &\Gamma_{T}\left(\Exp{f(z^{T}) - f(x^*)} + \kappa \Exp{\frac{1}{n}\sum_{i=1}^n \norm{h_i^{T} - \nabla f_i(z^{T})}^2} + \nu_{T} \Exp{\norm{w^{T} - u^{T}}^2}\right.\\
    &\quad  \left.+ \rho \Exp{\norm{k^{T} - \nabla f(z^{T})}^2} + \lambda \Exp{\norm{v^{T} - \nabla f(z^{T})}^2}\right)\\
    &\leq \Gamma_{0} \Bigg( \Exp{f(z^{0}) - f(x^*)} + \kappa \Exp{\frac{1}{n}\sum_{i=1}^n \norm{h_i^0 - \nabla f_i(z^{0})}^2}  + \nu_{0} \Exp{\norm{w^0 - u^0}^2} \\
    &\qquad\qquad + \rho \Exp{\norm{k^{0} - \nabla f(z^0)}^2} + \lambda \Exp{\norm{v^{0} - \nabla f(z^0)}^2} \Bigg)\\
    &\quad+ \left(\frac{\bar{L} + \Gamma_{0} \mu}{2} \Exp{\norm{u^0 - x^*}^2} - \frac{\bar{L} + \Gamma_{T} \mu}{2} \Exp{\norm{u^{T} - x^*}^2}\right).
\end{align*}
Using the initial conditions, the non-negativity of the terms, and $\Gamma_0 \leq \nicefrac{\bar{L}}{L}$, we get
\begin{align*}
  \Gamma_{T}\Exp{f(z^{T}) - f(x^*)} &\leq \Gamma_{0} \left(f(z^{0}) - f(x^*)\right) + \frac{\bar{L} + \Gamma_{0} \mu}{2} \norm{u^0 - x^*}^2 \\
  &\leq \frac{\bar{L}}{L}\left(f(z^{0}) - f(x^*)\right) + \bar{L} \norm{z^0 - x^*}^2.
\end{align*}
Using the $L$--smoothness and $\mu \leq L \leq \bar{L},$ we have
\begin{align*}
  \Gamma_{T}\Exp{f(z^{T}) - f(x^*)} &\leq \bar{L} \norm{z^0 - x^*}^2 + \bar{L} \norm{z^0 - x^*}^2 \leq 2\bar{L} \norm{z^0 - x^*}^2.
\end{align*}
Using Lemma~\ref{lemma:learning_rates}, we have
\begin{align*}
  \begin{cases}
    \frac{\Gamma_0}{2} \exp \left(t p \theta_{\min}\right) \Exp{f(z^{T}) - f(x^*)} \leq 2\bar{L} \norm{z^0 - x^*}^2, &t < \bar{t} \\
    \frac{p (t - \bar{t})^2}{16} \Exp{f(z^{T}) - f(x^*)} \leq 2\bar{L} \norm{z^0 - x^*}^2, &t \geq \bar{t},
  \end{cases}
\end{align*}
where $\bar{t} \eqdef \max\left\{\left\lceil\frac{1}{p \theta_{\min}} \log \frac{1}{2 \Gamma_0 p \theta_{\min}^2}\right\rceil, 0\right\}.$ The last inequalities guarantees that Algorithm~\ref{alg:bi_diana} returns $\varepsilon$-solution after \eqref{eq:main_theorem_general_convex}
iterations.

\end{proof}

\subsection{Choosing optimal parameters}
\label{sec:opt_param}

\COROLLARYMAINREALISTIC*

\begin{proof}
    We implicitly assume that $p \in (0, 1]$ and $\tau \in (0, 1].$ Using \eqref{eq:abstract_comm_complexity} and \eqref{eq:main_complexity}, we have
    \begin{align*}
      \argmin_{p, \tau} \max\limits_{L_{\max} \in [L, n L]} \mathfrak{m}^{r}_{{\scriptscriptstyle \textnormal{new}}} = \argmin_{p, \tau} \max\limits_{L_{\max} \in [L, n L]} \widetilde{\Theta} \left((1 - r)K_{\omega} T + r \left(K_{\alpha} + p d\right) T\right).
    \end{align*}
    Note that only $T$ depends on $L_{\max}$ and $\tau.$ We have
    \begin{align*}
      &\min_{\tau} \max\limits_{L_{\max} \in [L, n L]} T \\
      &\overset{\eqref{eq:strong_complexity}}{=} \min_{\tau} \max\limits_{L_{\max} \in [L, n L]} \widetilde{\Theta}\Bigg(\max\Bigg\{\sqrt{\frac{L}{\alpha p \mu}},\sqrt{\frac{L}{\alpha \tau \mu}},\sqrt{\frac{\sqrt{L L_{\max}} (\omega + 1) \sqrt{\omega \tau}}{\alpha \sqrt{n} \mu}},\\
      & \qquad\quad \qquad \qquad \qquad \sqrt{\frac{\sqrt{L L_{\max}} \sqrt{\omega + 1} \sqrt{\omega \tau}}{\alpha \sqrt{p} \sqrt{n} \mu}},\sqrt{\frac{L_{\max} \omega (\omega + 1)^2 p}{n \mu}}, \sqrt{\frac{L_{\max} \omega}{n p \mu}}, \frac{1}{\alpha}, \frac{1}{\tau}, (\omega + 1), \frac{1}{p}\Bigg\}\Bigg) \\
      &= \min_{\tau} \widetilde{\Theta}\Bigg(\max\Bigg\{\sqrt{\frac{L}{\alpha p \mu}},\sqrt{\frac{L}{\alpha \tau \mu}},\sqrt{\frac{L (\omega + 1) \sqrt{\omega \tau}}{\alpha \mu}},\sqrt{\frac{L \sqrt{\omega + 1} \sqrt{\omega \tau}}{\alpha \sqrt{p} \mu}},\sqrt{\frac{L \omega (\omega + 1)^2 p}{\mu}}, \sqrt{\frac{L \omega}{p \mu}}, \\
      & \qquad\quad \qquad \qquad \frac{1}{\alpha}, \frac{1}{\tau}, (\omega + 1), \frac{1}{p}\Bigg\}\Bigg) \\
    \end{align*}
    The last term attains the minimum when 
    \begin{align}
    \label{eq:opt_tau}
    \tau = \min\left\{\frac{1}{\omega + 1}, \frac{p^{1/3}}{(\omega + 1)^{2/3}}\right\}.
    \end{align} Therefore, we get
    \begin{eqnarray}
    \begin{aligned}
      \label{eq:t_prime}
      &T' \eqdef \min_{\tau} \max\limits_{L_{\max} \in [L, n L]} T \\
      &= \widetilde{\Theta}\Bigg(\max\Bigg\{\sqrt{\frac{L}{\alpha p \mu}},\sqrt{\frac{L (\omega + 1)}{\alpha \mu}},\sqrt{\frac{L (\omega + 1)^{2/3}}{\alpha p^{1/3} \mu}},\sqrt{\frac{L \omega (\omega + 1)^2 p}{\mu}}, \sqrt{\frac{L \omega}{p \mu}}, \frac{1}{\alpha}, (\omega + 1), \frac{1}{p}\Bigg\}\Bigg),
    \end{aligned}
    \end{eqnarray}
    where we use that $$\frac{1}{\tau} = \max\left\{\omega + 1, \frac{(\omega + 1)^{2/3}}{p^{1/3}}\right\} \leq \max\left\{\omega + 1, \frac{2}{3}\left(\omega + 1\right) + \frac{1}{3 p}\right\} = \Theta\left(\max\left\{\left(\omega + 1\right), \frac{1}{p}\right\}\right).$$
    It is left to find 
    \begin{align}
      \argmin_{p} \left(\min_{\tau}\max\limits_{L_{\max} \in [L, n L]} \mathfrak{m}^{r}_{{\scriptscriptstyle \textnormal{new}}}\right) &= \argmin_{p} \widetilde{\Theta} \left((1 - r)K_{\omega} T' + r \left(K_{\alpha} + p d\right) T'\right) \nonumber\\
      &= \argmin_{p} \widetilde{\Theta} \left(A \times T' + B \times p T'\right), \label{eq:last_argmin_p}
    \end{align}
    where $A \eqdef (1 - r)K_{\omega} + r K_{\alpha} \geq 0$ and $B \eqdef r d \geq 0.$ Note that $A$ and $B$ do not depend on $p.$ 
    If $p \geq \nicefrac{A}{B},$ then $\widetilde{\Theta} \left(A T' + B p T'\right) = \widetilde{\Theta} \left(B p T'\right).$ The term $p T'$ is non-decreasing function w.r.t. $p.$ For $p \geq \nicefrac{A}{B},$ it means that an optimal point $p = \nicefrac{A}{B}.$ Thus the argmin \eqref{eq:last_argmin_p} is equivalent to $$\argmin_{p \in Q_0} \widetilde{\Theta} \left(A \times T' + B \times p T'\right) = \argmin_{p \in Q_0} \widetilde{\Theta} \left(T'\right),$$ where $Q_0 \eqdef \left\{p\,\middle|\, p \leq \frac{A}{B}\right\}.$
    Thus, we have
    \begin{align*}
      &\argmin\limits_{p \in Q_0} \widetilde{\Theta}\Bigg(\max\Bigg\{\sqrt{\frac{L}{\alpha p \mu}},\sqrt{\frac{L (\omega + 1)}{\alpha \mu}},\sqrt{\frac{L (\omega + 1)^{2/3}}{\alpha p^{1/3} \mu}},\sqrt{\frac{L \omega (\omega + 1)^2 p}{\mu}}, \sqrt{\frac{L \omega}{p \mu}}, \frac{1}{\alpha}, (\omega + 1), \frac{1}{p}\Bigg\}\Bigg).
    \end{align*}
    The next observation is that this argmin is non-decreasing when $p \geq \nicefrac{1}{\omega + 1}.$ It means that the minimum attains at some point $p \in Q_1 \eqdef \{p\,|\, p \leq \nicefrac{1}{\omega + 1}, p \in Q_0\}.$ Using this information, we can eliminate the redundant terms (for instance, $A\sqrt{\frac{L (\omega + 1)}{\alpha \mu}} \leq A\sqrt{\frac{L}{\alpha p \mu}}$ for $p \in Q_1$) and get an equivalent argmin
    \begin{align*}
      &\argmin\limits_{p \in Q_1}\widetilde{\Theta}\Bigg(\max\Bigg\{\sqrt{\frac{L}{\alpha p \mu}}, \sqrt{\frac{L \omega}{p \mu}}, \frac{1}{\alpha},  \frac{1}{p}\Bigg\}\Bigg),
    \end{align*}
    The last term attains the minimum at 
    \begin{align}
      \label{eq:opt_p}
      p = \min\left\{\frac{1}{\omega + 1}, \frac{A}{B}\right\}.
    \end{align}
    Using \eqref{eq:opt_p} and \eqref{eq:opt_tau}, an optimal $\tau$ is
    \begin{align}
      \label{eq:opt_tau_2}
      \tau = \frac{p^{1/3}}{(\omega + 1)^{2/3}}.
    \end{align}
    We substitute \eqref{eq:opt_p} and \eqref{eq:opt_tau_2} to \eqref{eq:strong_complexity} and obtain \eqref{eq:realistic_compl}. We use Lemma~\ref{lemma:lipt_constants} to eliminate the redundant terms. Note that $\mu^r_{\omega, \alpha} \eqdef \nicefrac{B}{A}.$
    
Using \eqref{eq:abstract_comm_complexity} and \eqref{eq:opt_p}, we have
    \begin{align*}
      \mathfrak{m}^{r}_{{\scriptscriptstyle \textnormal{realistic}}} &= (1 - r)K_{\omega} T_{{\scriptscriptstyle \textnormal{realistic}}} + r \left(K_{\alpha} + p d\right) T_{{\scriptscriptstyle \textnormal{realistic}}} + d\\
     &\leq (1 - r)K_{\omega} T_{{\scriptscriptstyle \textnormal{realistic}}} + r \left(K_{\alpha} + \frac{(1 - r)K_{\omega} + r K_{\alpha}}{r d} d\right) T_{{\scriptscriptstyle \textnormal{realistic}}} + d \\
     &= 2 (1 - r)K_{\omega} T_{{\scriptscriptstyle \textnormal{realistic}}} + 2 r K_{\alpha} T_{{\scriptscriptstyle \textnormal{realistic}}} + d.
    \end{align*}
    Note that
    \begin{align*}
      \mathfrak{m}^{r}_{{\scriptscriptstyle \textnormal{realistic}}} &= (1 - r)K_{\omega} T_{{\scriptscriptstyle \textnormal{realistic}}} + r \left(K_{\alpha} + p d\right) T_{{\scriptscriptstyle \textnormal{realistic}}} + d\\
     &\geq (1 - r)K_{\omega} T_{{\scriptscriptstyle \textnormal{realistic}}} + r K_{\alpha} T_{{\scriptscriptstyle \textnormal{realistic}}} + d.
    \end{align*}

\end{proof}

\COROLLARYMAINSTRONG*

\begin{proof}
  We implicitly assume that $p \in (0, 1]$ and $\tau \in (0, 1].$ We start the proof as in Theorem~\ref{cor:realistic}. Using \eqref{eq:abstract_comm_complexity} and \eqref{eq:main_complexity}, we have
    \begin{align*}
      \argmin_{p, \tau} \mathfrak{m}^{r}_{{\scriptscriptstyle \textnormal{new}}} = \argmin_{p, \tau} \widetilde{\Theta} \left((1 - r)K_{\omega} T + r \left(K_{\alpha} + p d\right) T\right).
    \end{align*}
  Unlike Theorem~\ref{cor:realistic}, we know the ratio $\nicefrac{L_{\max}}{L},$ thus
  \begin{align*}
    &\min_{\tau} T \\
    &\overset{\eqref{eq:strong_complexity}}{=} \min_{\tau} \widetilde{\Theta}\Bigg(\max\Bigg\{\sqrt{\frac{L}{\alpha p \mu}},\sqrt{\frac{L}{\alpha \tau \mu}},\sqrt{\frac{\sqrt{L L_{\max}} (\omega + 1) \sqrt{\omega \tau}}{\alpha \sqrt{n} \mu}},\\
    & \qquad\quad \qquad \qquad \quad \sqrt{\frac{\sqrt{L L_{\max}} \sqrt{\omega + 1} \sqrt{\omega \tau}}{\alpha \sqrt{p} \sqrt{n} \mu}},\sqrt{\frac{L_{\max} \omega (\omega + 1)^2 p}{n \mu}}, \sqrt{\frac{L_{\max} \omega}{n p \mu}}, \frac{1}{\alpha}, \frac{1}{\tau}, (\omega + 1), \frac{1}{p}\Bigg\}\Bigg)
  \end{align*}
  The last term attains the minimum when 
  \begin{align}
  \label{eq:opt_tau_know}
  \tau = \min\left\{1, \left(\frac{L n}{L_{\max}}\right)^{1/3} \min\left\{\frac{1}{\omega + 1}, \frac{p^{1/3}}{(\omega + 1)^{2/3}}\right\}\right\}.
  \end{align}
  Therefore, we get
    \begin{eqnarray}
    \begin{aligned}
      \label{eq:t_prime_know}
      &T' \eqdef \min_{\tau} T \\
      &= \widetilde{\Theta}\Bigg(\max\Bigg\{\sqrt{\frac{L}{\alpha p \mu}},\sqrt{\frac{L^{2/3} L_{\max}^{1/3} (\omega + 1)}{\alpha n^{1/3} \mu}}, \sqrt{\frac{L^{2/3} L_{\max}^{1/3} (\omega + 1)^{2/3}}{\alpha n^{1/3} p^{1/3} \mu}}, \sqrt{\frac{L_{\max} \omega (\omega + 1)^2 p}{n \mu}}, \sqrt{\frac{L_{\max} \omega}{n p \mu}}, \frac{1}{\alpha}, \omega + 1, \frac{1}{p}\Bigg\}\Bigg).
    \end{aligned}
    \end{eqnarray}
  It is left to find 
  \begin{align}
    \argmin_{p} \left(\min_{\tau} \mathfrak{m}^{r}_{{\scriptscriptstyle \textnormal{new}}}\right) &= \argmin_{p} \widetilde{\Theta} \left((1 - r)K_{\omega} T' + r \left(K_{\alpha} + p d\right) T'\right) \nonumber\\
    &= \argmin_{p} \widetilde{\Theta} \left(A \times T' + B \times p T'\right), \label{eq:last_argmin_p_know}
  \end{align}
  where $A \eqdef (1 - r)K_{\omega} + r K_{\alpha} \geq 0$ and $B \eqdef r d \geq 0.$ Note that $A$ and $B$ do not depend on $p.$ If $p \geq \nicefrac{A}{B},$ then $\widetilde{\Theta} \left(A T' + B p T'\right) = \widetilde{\Theta} \left(B p T'\right).$ The term $p T'$ is non-decreasing function w.r.t. $p.$ For $p \geq \nicefrac{A}{B},$ it means that an optimal point $p = \nicefrac{A}{B}.$ Thus the argmin \eqref{eq:last_argmin_p_know} is equivalent to $$\argmin_{p \in Q_0} \widetilde{\Theta} \left(A \times T' + B \times p T'\right) = \argmin_{p \in Q_0} \widetilde{\Theta} \left(T'\right),$$ where $Q_0 \eqdef \left\{p\,\middle|\, p \leq \frac{A}{B}\right\}.$
  Next, we have
  \begin{align}
    &\argmin_{p \in Q_0} \widetilde{\Theta} \left(T'\right) \nonumber\\
    &=\argmin\limits_{p \in Q_0} \nonumber\\
    &\widetilde{\Theta}\Bigg(\max\Bigg\{\sqrt{\frac{L}{\alpha p \mu}}, \sqrt{\frac{L^{2/3} L_{\max}^{1/3} (\omega + 1)}{\alpha n^{1/3} \mu}}, \sqrt{\frac{L^{2/3} L_{\max}^{1/3} (\omega + 1)^{2/3}}{\alpha n^{1/3} p^{1/3} \mu}}, \sqrt{\frac{L_{\max} \omega (\omega + 1)^2 p}{n \mu}}, \sqrt{\frac{L_{\max} \omega}{n p \mu}}, \frac{1}{\alpha}, (\omega + 1), \frac{1}{p}\Bigg\}\Bigg).\label{eq:proof_aux_p_monot}
  \end{align}

  For $p \geq \left(\frac{L n}{L_{\max}}\right)^{1/3} \frac{1}{\omega + 1},$ using Lemma~\ref{lemma:lipt_constants}, we have $p \geq \frac{1}{\omega + 1},$
  \begin{align*}
    &\sqrt{\frac{L}{\alpha p \mu}} \leq \sqrt{\frac{L^{2/3} L_{\max}^{1/3}  (\omega + 1)}{\alpha n^{1/3} \mu}},\\
    &\sqrt{\frac{L^{2/3} L_{\max}^{1/3} (\omega + 1)^{2/3}}{\alpha n^{1/3} p^{1/3} \mu}} \leq \sqrt{\frac{L^{2/3} L_{\max}^{1/3} (\omega + 1)}{\alpha n^{1/3} \mu}}\\
    &\sqrt{\frac{L_{\max} \omega}{n p \mu}} \leq \sqrt{\frac{L_{\max} \omega (\omega + 1)^2 p}{n \mu}}\\
    &\frac{1}{p} \leq \left(\frac{L_{\max}}{L n}\right)^{1/3} (\omega + 1) \leq (\omega + 1).\\
  \end{align*}
  It means that for $p \geq \left(\frac{L n}{L_{\max}}\right)^{1/3} \frac{1}{\omega + 1},$ the argmin \eqref{eq:proof_aux_p_monot} is equivalent to
  \begin{align*}
    &\argmin\limits_{p \in Q_0} \widetilde{\Theta}\Bigg(\max\Bigg\{\sqrt{\frac{L^{2/3} L_{\max}^{1/3} (\omega + 1)}{\alpha n^{1/3} \mu}}, \sqrt{\frac{L_{\max} \omega (\omega + 1)^2 p}{n \mu}}, \frac{1}{\alpha}, (\omega + 1)\Bigg\}\Bigg).
  \end{align*}

  Since all terms are non-increasing functions of $p,$ the minimum is attained at a point $p = \left(\frac{L n}{L_{\max}}\right)^{1/3} \frac{1}{\omega + 1}$ for all $p \geq \left(\frac{L n}{L_{\max}}\right)^{1/3} \frac{1}{\omega + 1}.$ Let us define $$Q_1 \eqdef \left\{p\,\middle|\, p \leq \left(\frac{L n}{L_{\max}}\right)^{1/3} \frac{1}{\omega + 1}, p \in Q_0\right\}.$$ The last observation means that the argmin \eqref{eq:proof_aux_p_monot} is equivalent to
  \begin{align}
    &\argmin\limits_{p \in Q_1} \nonumber\\
    &\widetilde{\Theta}\Bigg(\max\Bigg\{\sqrt{\frac{L}{\alpha p \mu}},\sqrt{\frac{L^{2/3} L_{\max}^{1/3} (\omega + 1)}{\alpha n^{1/3} \mu}}, \sqrt{\frac{L^{2/3} L_{\max}^{1/3} (\omega + 1)^{2/3}}{\alpha n^{1/3} p^{1/3} \mu}}, \sqrt{\frac{L_{\max} \omega (\omega + 1)^2 p}{n \mu}}, \sqrt{\frac{L_{\max} \omega}{n p \mu}}, \frac{1}{\alpha}, (\omega + 1), \frac{1}{p}\Bigg\}\Bigg) \nonumber\\
    &=\argmin\limits_{p \in Q_1} \widetilde{\Theta}\Bigg(\max\Bigg\{\sqrt{\frac{L}{\alpha p \mu}}, \sqrt{\frac{L_{\max} \omega (\omega + 1)^2 p}{n \mu}}, \sqrt{\frac{L_{\max} \omega}{n p \mu}}, \frac{1}{\alpha}, (\omega + 1), \frac{1}{p}\Bigg\}\Bigg), \label{eq:argmin_max_p}
  \end{align}

  where we eliminate the redundant terms using the additional information $p \leq \left(\frac{L n}{L_{\max}}\right)^{1/3} \frac{1}{\omega + 1}.$ In particular,
  \begin{align*}
    \sqrt{\frac{L^{2/3} L_{\max}^{1/3} (\omega + 1)}{\alpha n^{1/3} \mu}} \leq \sqrt{\frac{L}{\alpha p \mu}} \textnormal{ and } \sqrt{\frac{L^{2/3} L_{\max}^{1/3} (\omega + 1)^{2/3}}{\alpha n^{1/3} p^{1/3} \mu}} \leq \sqrt{\frac{L}{\alpha p \mu}}.
  \end{align*}
  for all $p \leq \left(\frac{L n}{L_{\max}}\right)^{1/3} \frac{1}{\omega + 1}.$ 
  Without the condition $p \in Q_1,$ the argmin \eqref{eq:argmin_max_p} attains the minimum at a point 
  \begin{align*}
    p = \max\left\{\frac{1}{\omega + 1}, \left(\frac{L n}{L_{\max}}\right)^{1/2}\frac{1}{\sqrt{\alpha} (\omega + 1)^{3/2}}\right\}.
  \end{align*}
  Considering the condition $p \in Q_1$ and $\mu^r_{\omega, \alpha} \eqdef \nicefrac{B}{A},$ we have
  \begin{align}
    \label{eq:opt_p_know}
    p = \min\left\{1, \frac{1}{\mu^r_{\omega, \alpha}}, \left(\frac{L n}{L_{\max}}\right)^{1/3} \frac{1}{\omega + 1}, \max\left\{\frac{1}{\omega + 1}, \left(\frac{L n}{L_{\max}}\right)^{1/2}\frac{1}{\sqrt{\alpha} (\omega + 1)^{3/2}}\right\}\right\}.
  \end{align}
  It is left carefully to substitute \eqref{eq:opt_p_know} to 
  \begin{align*}
    \widetilde{\Theta}\Bigg(\max\Bigg\{\sqrt{\frac{L}{\alpha p \mu}}, \sqrt{\frac{L_{\max} \omega}{n p \mu}}, \frac{1}{\alpha}, (\omega + 1), \frac{1}{p}\Bigg\}\Bigg)
  \end{align*}
  and obtain \eqref{eq:optimistic_compl}.
  The proof of $\mathfrak{m}^{r}_{{\scriptscriptstyle \textnormal{realistic}}} = \widetilde{\Theta}\left(\left((1 - r)K_{\omega} + r K_{\alpha}\right) T_{{\scriptscriptstyle \textnormal{realistic}}}\right)$ is the same as in Theorem~\ref{cor:realistic}.
\end{proof}

\subsection{Comparison with \algname{EF21 + DIANA}}
\label{sec:compare_ef21}
\THEOREMBETTEREF*

\begin{proof}
  Using the inequality of arithmetic and geometric means, i.e., $\sqrt{x y} \leq \frac{x + y}{2}$ for all $x, y \geq 0,$ and $L \geq \mu,$ we have
\begin{align*}
  \mathfrak{m}^{r}_{{\scriptscriptstyle \textnormal{realistic}}} &= \widetilde{\Theta}\Bigg(K_{\omega, \alpha}^r \Bigg(\sqrt{\frac{L (\omega + 1)}{\alpha \mu}} +  \sqrt{\frac{L_{\max} \omega (\omega + 1)}{n \mu}} + \sqrt{\frac{L \mu^r_{\omega, \alpha}}{\alpha \mu}} + \sqrt{\frac{L_{\max} \omega \mu^r_{\omega, \alpha}}{n \mu}} + \frac{1}{\alpha} + \omega + \mu^r_{\omega, \alpha}\Bigg) + d\Bigg)\\
  &= \widetilde{\cO}\Bigg(K_{\omega, \alpha}^r \Bigg(\frac{L}{\alpha \mu} + \frac{L_{\max} \omega}{n \mu} + \omega + \sqrt{\frac{L \mu^r_{\omega, \alpha}}{\alpha \mu}} + \sqrt{\frac{L_{\max} \omega \mu^r_{\omega, \alpha}}{n \mu}} + \mu^r_{\omega, \alpha}\Bigg) + d\Bigg).
\end{align*}
From the definition of $\mu^r_{\omega, \alpha} \eqdef \nicefrac{r d}{K_{\omega, \alpha}^r},$ we get
\begin{align*}
  \mathfrak{m}^{r}_{{\scriptscriptstyle \textnormal{realistic}}} = \widetilde{\cO}\Bigg(&K_{\omega, \alpha}^r \Bigg(\frac{L}{\alpha \mu} + \frac{L_{\max} \omega}{n \mu} + \omega\Bigg) + \Bigg(\sqrt{\frac{L K_{\omega, \alpha}^r \times r d}{\alpha \mu}} + \sqrt{\frac{L_{\max} \omega K_{\omega, \alpha}^r \times r d}{n \mu}} + rd\Bigg) + d\Bigg).
\end{align*}
Using the inequality of arithmetic and geometric means again and $r \leq 1$, we obtain
\begin{align*}
  \mathfrak{m}^{r}_{{\scriptscriptstyle \textnormal{realistic}}} = \widetilde{\cO}\Bigg(&K_{\omega, \alpha}^r \Bigg(\frac{L}{\alpha \mu} + \frac{L_{\max} \omega}{n \mu} + \omega\Bigg) + K_{\omega, \alpha}^r\Bigg(\frac{L}{\alpha \mu} + \frac{L_{\max} \omega}{n \mu}\Bigg) + d\Bigg).
\end{align*}
The last equality means that $\mathfrak{m}^{r}_{{\scriptscriptstyle \textnormal{realistic}}} = \widetilde{\cO}\left(\mathfrak{m}^{r}_{{\scriptscriptstyle \textnormal{EF21-P + DIANA}}}\right)$ for all $r \in [0, 1].$
\end{proof}

\subsection{Comparison with \algname{AGD}}

\THEOREMBETTERAGD*

\begin{proof}
  Consider that $r \in [0, 1/2],$ then $K_{\omega} = K,$ $K_{\alpha} = \min\{\lceil\nicefrac{1 - r}{r} K\rceil, d\}.$ Therefore, we have $$K_{\omega, \alpha}^r \eqdef (1 - r)K_{\omega} + r K_{\alpha} \leq (1 - r)K + r \lceil\nicefrac{1 - r}{r} K\rceil \leq 3 (1 - r) K$$ and 
  $$K_{\omega, \alpha}^r \eqdef (1 - r)K_{\omega} + r K_{\alpha} \geq (1 - r)K.$$
  Using this observation, we obtain
  $\mu^r_{\omega, \alpha} \eqdef \frac{rd}{K_{\omega, \alpha}^r} \leq \frac{rd}{(1 - r) K} \leq \frac{d}{K}.$ Note that $\alpha \geq \nicefrac{K_{\alpha}}{d}$ and $\omega \leq \nicefrac{d}{K_{\omega}} - 1$ for Top$K$ and Rand$K$.
  Thus $\alpha \geq \min\left\{\frac{(1 - r) K}{r d}, 1\right\}$ and $\omega \leq \frac{d}{K} - 1.$ We substitute the bounds to \eqref{eq:acc_diana_total} and obtain
  \begin{align*}
    \mathfrak{m}^{r}_{{\scriptscriptstyle \textnormal{realistic}}} = \widetilde{\cO}\Bigg((1 - r) K \Bigg(\sqrt{\frac{d}{K}}\sqrt{\frac{L}{\mu}} + \frac{d}{K}\sqrt{\frac{r L}{(1 - r)\mu}} + \frac{d}{K}\sqrt{\frac{L_{\max}}{n \mu}} + \frac{r d}{(1 - r) K} + \frac{d}{K} + \frac{rd}{(1 - r) K}\Bigg) + d\Bigg).
  \end{align*}
  Since $r \in [0, 1/2]$ and $K \leq d,$ one can easily show that
  \begin{align*}
    \mathfrak{m}^{r}_{{\scriptscriptstyle \textnormal{realistic}}} = \widetilde{\cO}\Bigg(&d \sqrt{\frac{L}{\mu}} + d \sqrt{\frac{L_{\max}}{n \mu}} + d\Bigg).
  \end{align*}
  It is left to use Lemma~\ref{lemma:lipt_constants}, to get $\mathfrak{m}^{r}_{{\scriptscriptstyle \textnormal{realistic}}} = \widetilde{\cO}\left(d \sqrt{\frac{L}{\mu}}\right) = \widetilde{\cO}\left(\mathfrak{m}_{{\scriptscriptstyle \textnormal{AGD}}}\right)$ for all $r \in [0, 1/2].$
  
  Assume that $r \in (1/2, 1],$ then $K_{\omega} = \min\{\lceil\nicefrac{r}{1 - r} K \rceil, d\}$ and $K_{\alpha} = K.$ Using the same reasoning, we have 
  $$K_{\omega, \alpha}^r \eqdef (1 - r)K_{\omega} + r K_{\alpha} \leq (1 - r) \lceil\nicefrac{r}{1 - r} K \rceil + r K \leq 3 r K,$$
  $$K_{\omega, \alpha}^r \eqdef (1 - r)K_{\omega} + r K_{\alpha} \geq r K,$$
  $$\mu^r_{\omega, \alpha} \eqdef \frac{rd}{K_{\omega, \alpha}^r} \leq \frac{d}{K},$$
  $$\alpha \geq \frac{K}{d} \textnormal{ and } \omega \leq \max\left\{\frac{(1 - r) d}{r K}, 1\right\} - 1.$$ By substituting these inequalities to \eqref{eq:acc_diana_total}, we obtain
  \begin{align*}
    \mathfrak{m}^{r}_{{\scriptscriptstyle \textnormal{realistic}}} = \widetilde{\cO}\Bigg(r K \Bigg(\frac{d}{K}\sqrt{\frac{L}{\mu}} + \frac{d}{K}\sqrt{\frac{(1 - r) L_{\max} \omega}{r n \mu}} + \frac{d}{K} + \frac{(1 - r)d}{r K} + \frac{d}{K}\Bigg) + d\Bigg)
  \end{align*}
  Using Lemma~\ref{lemma:lipt_constants}, one can easily show that $\mathfrak{m}^{r}_{{\scriptscriptstyle \textnormal{realistic}}} = \widetilde{\cO}\left(d \sqrt{\frac{L}{\mu}}\right) = \widetilde{\cO}\left(\mathfrak{m}_{{\scriptscriptstyle \textnormal{AGD}}}\right)$ for all $r \in (1/2, 1].$
  \end{proof}

\section{Auxillary Inequalities For $\bar{L}$}
\newcommand{\constlipt}{c}
We now prove useful bounds for $\bar{L}.$
\begin{lemma}[Auxillary Inequalities]
  \label{lemma:aux_ineq_L}
  Assume that the constraint \eqref{eq:constr_lipts} hold, and a constant $c = \liptconstraintconst.$ Then

  \begin{tabularx}{1.1\linewidth}{XXX}
  \begin{equation}
    \bar{L} \geq \constlipt \frac{L_{\max} \omega p^2}{\beta^2 n} \label{eq:lipt:max_2}
  \end{equation}
  &
  \begin{equation}
    \bar{L} \geq \constlipt \frac{L_{\max} \omega p}{\beta n} \label{eq:lipt:max_1}
  \end{equation}
  &
  \begin{equation}
    \bar{L} \geq \constlipt \frac{L_{\max} \omega}{n} \label{eq:lipt:max}
  \end{equation}
  \end{tabularx}
  \begin{tabularx}{1.1\linewidth}{XXX}
    \begin{equation}
      \bar{L} \geq \constlipt\frac{\sqrt{L L_{\max}} p \sqrt{\omega \tau}}{\alpha \beta \sqrt{n}} \label{eq:lipt:l_l_max_2}
    \end{equation}
    &
    \begin{equation}
      \bar{L} \geq \constlipt\frac{\sqrt{L L_{\max}} \sqrt{p} \sqrt{\omega \tau}}{\alpha \sqrt{\beta} \sqrt{n}} \label{eq:lipt:l_l_max_1}
    \end{equation}
    &
    \begin{equation}
      \bar{L} \geq \constlipt\frac{\sqrt{L L_{\max}} \sqrt{p} \sqrt{\omega \tau}}{\alpha \sqrt{n}} \label{eq:lipt:l_l_max_p}
    \end{equation}
  \end{tabularx}
  \begin{tabularx}{1.1\linewidth}{XXX}
    \begin{equation}
      \bar{L} \geq \constlipt\frac{\widehat{L} p \sqrt{\omega \tau}}{\alpha \beta \sqrt{n}} \label{eq:lipt:hat_2}
    \end{equation}
    &
    \begin{equation}
      \bar{L} \geq \constlipt\frac{\widehat{L} \sqrt{p} \sqrt{\omega \tau}}{\alpha \sqrt{\beta} \sqrt{n}} \label{eq:lipt:hat_1}
    \end{equation}
    &
    \begin{equation}
      \bar{L} \geq \constlipt\frac{\widehat{L} \sqrt{p} \sqrt{\omega \tau}}{\alpha \sqrt{n}} \label{eq:lipt:hat_p}
    \end{equation}
  \end{tabularx}
  \begin{tabularx}{1.1\linewidth}{XXX}
    \begin{equation}
      \bar{L} \geq \constlipt\frac{\widehat{L} p \sqrt{\omega}}{\sqrt{\alpha} \beta \sqrt{n}} \label{eq:lipt:hat_alpha_2}
    \end{equation}
    &
    \begin{equation}
      \bar{L} \geq \constlipt\frac{\widehat{L} \sqrt{p} \sqrt{\omega}}{\sqrt{\alpha} \sqrt{\beta} \sqrt{n}} \label{eq:lipt:hat_alpha_1}
    \end{equation}
    &
  \end{tabularx}
  \begin{tabularx}{1.1\linewidth}{XXX}
    \begin{equation}
      \bar{L} \geq \constlipt\frac{\widehat{L} p \sqrt{\omega}}{\beta \sqrt{n}} \label{eq:lipt:hat_no_alpha_2}
    \end{equation}
    &
    \begin{equation}
      \bar{L} \geq \constlipt\frac{\widehat{L} \sqrt{p \omega}}{\sqrt{\beta n}} \label{eq:lipt:hat_no_alpha_1}
    \end{equation}
    &
  \end{tabularx}
  \begin{tabularx}{1.1\linewidth}{XXX}
    \begin{equation}
      \bar{L} \geq \constlipt\frac{L}{\alpha} \label{eq:lipt:plain}
    \end{equation}
    &
    \begin{equation}
      \bar{L} \geq \constlipt\frac{L p}{\alpha \tau} \label{eq:lipt:plain_p_alpha}
    \end{equation}
    &
    \begin{equation}
      \bar{L} \geq \constlipt L \label{eq:lipt:plain_no_alpha}
    \end{equation}
  \end{tabularx}
  \begin{tabularx}{1.1\linewidth}{XXX}
    \begin{equation}
      \bar{L} \geq \constlipt \left(\frac{L \widehat{L}^2 \omega p^4}{\alpha^2 \beta^2 n \tau^2}\right)^{1/3} \label{eq:lipt:double_lipt_2}
    \end{equation}
    &
    \begin{equation}
      \bar{L} \geq \constlipt \left(\frac{L \widehat{L}^2 \omega p^3}{\alpha^2 \beta n \tau^2}\right)^{1/3} \label{eq:lipt:double_lipt_1}
    \end{equation}
    &
  \end{tabularx}
\end{lemma}

\begin{proof}
  The inequalities \eqref{eq:lipt:max_2} and \eqref{eq:lipt:max} follow from \eqref{eq:constr_lipts}. 
  The inequality \eqref{eq:lipt:max_1} follows from \eqref{eq:lipt:max_2} and \eqref{eq:lipt:max}:
  \begin{align*}
    \constlipt \frac{L_{\max} \omega p}{\beta n} \leq \constlipt \frac{L_{\max} \omega}{n} \left(\frac{1}{2} \times \frac{p^2}{\beta^2} + \frac{1}{2} \times 1^2\right) = \frac{\constlipt}{2} \times \frac{L_{\max} \omega p^2}{\beta^2 n} + \frac{\constlipt}{2} \times \frac{L_{\max} \omega}{n} \leq \bar{L}.
  \end{align*}
  The inequalities \eqref{eq:lipt:l_l_max_2} and \eqref{eq:lipt:l_l_max_1} follow from \eqref{eq:constr_lipts}. The inequality \eqref{eq:lipt:l_l_max_p} follows from \eqref{eq:lipt:l_l_max_1} and $\beta \in (0, 1]$:
  \begin{align*}
    \bar{L} \geq \constlipt\frac{\sqrt{L L_{\max}} \sqrt{p} \sqrt{\omega \tau}}{\alpha \sqrt{\beta} \sqrt{n}} \geq \constlipt\frac{\sqrt{L L_{\max}} \sqrt{p} \sqrt{\omega \tau}}{\alpha \sqrt{n}}.
  \end{align*}
  Using Lemma~\ref{lemma:lipt_constants}, \eqref{eq:lipt:l_l_max_2}, \eqref{eq:lipt:l_l_max_1}, and \eqref{eq:lipt:l_l_max_p}, the inequalities \eqref{eq:lipt:hat_2}, \eqref{eq:lipt:hat_1}, and \eqref{eq:lipt:hat_p} follow from
  \begin{align*}
    &\bar{L} \geq \constlipt\frac{\sqrt{L L_{\max}} p \sqrt{\omega \tau}}{\alpha \beta \sqrt{n}} \geq \constlipt\frac{\widehat{L} p \sqrt{\omega \tau}}{\alpha \beta \sqrt{n}},\\
    &\bar{L} \geq \constlipt\frac{\sqrt{L L_{\max}} \sqrt{p} \sqrt{\omega \tau}}{\alpha \sqrt{\beta} \sqrt{n}} \geq \constlipt\frac{\widehat{L} \sqrt{p} \sqrt{\omega \tau}}{\alpha \sqrt{\beta} \sqrt{n}},\\
    &\bar{L} \geq \constlipt\frac{\sqrt{L L_{\max}} \sqrt{p} \sqrt{\omega \tau}}{\alpha \sqrt{n}} \geq \constlipt\frac{\widehat{L} \sqrt{p} \sqrt{\omega \tau}}{\alpha \sqrt{n}}.
  \end{align*}
  Next, using Lemma~\ref{lemma:lipt_constants}, and $\frac{x+y}{2} \geq \sqrt{x y}$ for all $x, y \geq 0,$ the inequality \eqref{eq:lipt:hat_alpha_2} follows from
  \begin{align*}
    \constlipt\frac{\widehat{L} p \sqrt{\omega}}{\sqrt{\alpha} \beta \sqrt{n}} \leq \constlipt\frac{\sqrt{L L_{\max}} p \sqrt{\omega}}{\sqrt{\alpha} \beta \sqrt{n}} \leq \frac{\constlipt}{2} \times \frac{L}{\alpha} + \frac{\constlipt}{2} \times \frac{L_{\max} p^2 \omega}{\beta^2 n} \overset{\eqref{eq:constr_lipts}}{\leq} \bar{L}.
  \end{align*}
  The inequality \eqref{eq:lipt:hat_alpha_1} follows from
  \begin{align*}
    \constlipt\frac{\widehat{L} \sqrt{p} \sqrt{\omega}}{\sqrt{\alpha} \sqrt{\beta} \sqrt{n}} \leq \constlipt\frac{\sqrt{L L_{\max}} \sqrt{p} \sqrt{\omega}}{\sqrt{\alpha} \sqrt{\beta} \sqrt{n}} \leq \frac{\constlipt}{2} \times \frac{L}{\alpha}+ \frac{\constlipt}{2} \times \frac{L_{\max} p \omega}{\beta n} \overset{\eqref{eq:constr_lipts}, \eqref{eq:lipt:max_1}}{\leq} \bar{L}.
  \end{align*}
  The inequalities \eqref{eq:lipt:hat_no_alpha_2} and \eqref{eq:lipt:hat_no_alpha_1} follow from \eqref{eq:lipt:hat_alpha_2}, \eqref{eq:lipt:hat_alpha_1}, and $\alpha \in (0, 1]:$
  \begin{align*}
    &\bar{L} \geq \constlipt\frac{\widehat{L} p \sqrt{\omega}}{\sqrt{\alpha} \beta \sqrt{n}} \geq \constlipt\frac{\widehat{L} p \sqrt{\omega}}{\beta \sqrt{n}}, \\
    &\bar{L} \geq \constlipt\frac{\widehat{L} \sqrt{p} \sqrt{\omega}}{\sqrt{\alpha} \sqrt{\beta} \sqrt{n}} \geq \constlipt\frac{\widehat{L} \sqrt{p} \sqrt{\omega}}{\sqrt{\beta} \sqrt{n}}.
  \end{align*}
  The inequalities \eqref{eq:lipt:plain} and \eqref{eq:lipt:plain_p_alpha} follow from \eqref{eq:constr_lipts}, and \eqref{eq:lipt:plain_no_alpha} follows from \eqref{eq:lipt:plain} and $\alpha \in (0, 1].$
  Using Lemma~\ref{lemma:lipt_constants}, and $\frac{x+y+z}{3} \geq (x y z)^{1/3}$ for all $x, y, z\geq 0,$ the inequalities \eqref{eq:lipt:double_lipt_2} and \eqref{eq:lipt:double_lipt_1} follows from
  \begin{align*}
    &\constlipt \left(\frac{L \widehat{L}^2 \omega p^4}{\alpha^2 \beta^2 n \tau^2}\right)^{1/3} \leq \constlipt \left(\frac{L^2 L_{\max} \omega p^4}{\alpha^2 \beta^2 n \tau^2}\right)^{1/3} \leq \frac{\constlipt}{3} \times \frac{L p}{\alpha \tau} + \frac{\constlipt}{3} \times \frac{L p}{\alpha \tau} + \frac{\constlipt}{3} \times \frac{L_{\max} \omega p^2}{\beta^2 n} \overset{\eqref{eq:constr_lipts}}{\leq} \bar{L}, \\
    &\constlipt \left(\frac{L \widehat{L}^2 \omega p^3}{\alpha^2 \beta n \tau^2}\right)^{1/3} \leq \constlipt \left(\frac{L^2 L_{\max} \omega p^3}{\alpha^2 \beta n \tau^2}\right)^{1/3} \leq \frac{\constlipt}{3} \times \frac{L p}{\alpha \tau} + \frac{\constlipt}{3} \times \frac{L p}{\alpha \tau}+ \frac{\constlipt}{3} \times \frac{L_{\max} \omega p}{\beta n} \overset{\eqref{eq:constr_lipts}, \eqref{eq:lipt:max_1}}{\leq} \bar{L}.
  \end{align*}
\end{proof}

\section{Proof of Lemma~\ref{lemma:parameters} (\nameref{lemma:parameters})}
\label{sec:lemma_parameters}

We use the notations from the proof of Theorem~\ref{theorem:main_theorem}.
\LEMMAPARAMETERS*

\begin{proof}
  The inequalities \eqref{eq:parameters_task} are equivalent to 
  \begin{equation*}
    \begin{gathered}
      \frac{8 \theta_{t+1}^2}{\alpha} \left(\frac{pL}{2} + \kappa 4 p \left(1 + \frac{p}{\beta}\right) \widehat{L}^2 + \rho 4 p L^2 + \lambda \left(2 p \left(1 + \frac{2 p}{\tau}\right) L^2 + \frac{4 p \tau^2 \omega \widehat{L}^2}{n}\right) \right) \leq \nu_{t}, \\
      \nu_{t} \frac{8}{\alpha p} \left(\frac{\gamma_{t+1}}{\bar{L} + \Gamma_{t+1} \mu}\right)^2 \leq \rho, \\
      \rho\frac{8 p}{\tau} \leq \lambda, \\
      \frac{8 p \omega}{n \bar{L} \beta} + \nu_{t} \frac{8 \omega}{n \beta} \left(\frac{\gamma_{t+1}}{\bar{L} + \Gamma_{t+1} \mu}\right)^2 + \lambda \frac{8 \tau^2 \omega}{n \beta}\leq \kappa.
    \end{gathered}
    \end{equation*}
  Let us take 
  \begin{align}
    \lambda \eqdef \rho\frac{8 p}{\tau} 
    \label{eq:lambda_sol}
  \end{align}
    to ensure that the third inequality holds. It left to find the parameters such that
  \begin{equation*}
  \begin{gathered}
      \frac{8 \theta_{t+1}^2}{\alpha} \left(\frac{pL}{2} + \kappa 4 p \left(1 + \frac{p}{\beta}\right) \widehat{L}^2 + \rho 4 p L^2 + \rho\frac{8 p}{\tau} \left(2 p \left(1 + \frac{2 p}{\tau}\right) L^2 + \frac{4 p \tau^2 \omega \widehat{L}^2}{n}\right) \right) \leq \nu_{t}, \\
      \nu_{t} \frac{8}{\alpha p} \left(\frac{\gamma_{t+1}}{\bar{L} + \Gamma_{t+1} \mu}\right)^2 \leq \rho, \\
      \frac{8 p \omega}{n \bar{L} \beta} + \nu_{t} \frac{8 \omega}{n \beta} \left(\frac{\gamma_{t+1}}{\bar{L} + \Gamma_{t+1} \mu}\right)^2 + \rho\frac{8 p}{\tau} \cdot \frac{8 \tau^2 \omega}{n \beta}\leq \kappa.
  \end{gathered}
  \end{equation*}
  Let us take 
  \begin{align}
    \nu_{t} &\eqdef \theta_{t+1}^2 \widehat{\nu}(\kappa, \rho),
    \label{eq:nu_sol}
  \end{align}
  where we additionally define 
  \begin{align}
    \label{eq:nu_sol_hat}
    \widehat{\nu} \equiv \widehat{\nu}(\kappa, \rho) &\eqdef \frac{8}{\alpha} \left(\frac{pL}{2} + \kappa 4 p \left(1 + \frac{p}{\beta}\right) \widehat{L}^2 + \rho 4 p L^2 + \rho\frac{8 p}{\tau} \left(2 p \left(1 + \frac{2 p}{\tau}\right) L^2 + \frac{4 p \tau^2 \omega \widehat{L}^2}{n}\right) \right),
  \end{align}
    to ensure that the first inequality holds. It left to find the parameters $\kappa$ and $\rho$ such that
  \begin{equation*}
    \begin{gathered}
      \widehat{\nu}(\kappa, \rho) \frac{8}{\alpha p} \left(\frac{\gamma_{t+1} \theta_{t+1}}{\bar{L} + \Gamma_{t+1} \mu}\right)^2 \leq \rho, \\
      \frac{8 p \omega}{n \bar{L} \beta} + \widehat{\nu}(\kappa, \rho) \frac{8 \omega}{n \beta} \left(\frac{\gamma_{t+1} \theta_{t+1}}{\bar{L} + \Gamma_{t+1} \mu}\right)^2 + \rho\frac{8 p}{\tau} \cdot \frac{8 \tau^2 \omega}{n \beta}\leq \kappa.
    \end{gathered}
    \end{equation*}
    Using Lemma~\ref{lemma:learning_rates}, we have $\frac{\gamma_{t+1} \theta_{t+1}}{\bar{L} + \Gamma_{t+1} \mu} \leq \frac{\gamma_{t+1} \theta_{t+1}}{\bar{L} + \Gamma_{t} \mu} \leq \frac{1}{\bar{L}}$, so it is sufficient to show that stronger inequalities hold:
    \begin{align}
        \widehat{\nu}(\kappa, \rho) \frac{8}{\alpha p \bar{L}^2} \leq \rho, \label{eq:rho_ineq} \\
        \frac{8 p \omega}{n \bar{L} \beta} + \widehat{\nu}(\kappa, \rho) \frac{8 \omega}{n \beta \bar{L}^2} + \rho\frac{8 p}{\tau} \cdot \frac{8 \tau^2 \omega}{n \beta}\leq \kappa. \label{eq:kappa_ineq}
    \end{align}
    
    From this point all formulas in this lemma are generated by the script from Section~\ref{sec:jupyter_notebook} (see Section 4 in Section~\ref{sec:jupyter_notebook}). We use the SymPy library \citep{sympy}.

    Using the definition of $\widehat{\nu},$ the left hand side of \eqref{eq:kappa_ineq} equals
    \begin{dmath}
      \formulakappaexpand
    \end{dmath}
    where we grouped the terms w.r.t. $\kappa.$ Let us take $\bar{L}$ such that the bracket is less or equal to $1 / 2.$ We define the constraints in Section~\ref{sec:sym_comp_kapap_const}.
    Therefore, \eqref{eq:kappa_ineq} holds if
    \begin{align}
      \formulakappasol
    \end{align}
    Using the definition of $\widehat{\nu}$ and $\kappa,$ the left hand side of \eqref{eq:rho_ineq} equals
    \begin{dmath}
      \formularhoexpand
    \end{dmath}
    where we grouped the terms w.r.t. $\rho.$ Let us take $\bar{L}$ such that the bracket is less or equal to $1 / 2.$ We define the constraints in Section~\ref{sec:sym_comp_rho_const}.
    Therefore, \eqref{eq:rho_ineq} holds if
    \begin{align}
      \formularhosol
    \end{align}
    Finally, under the constraints from Sections~\ref{sec:sym_comp_kapap_const} and \ref{sec:sym_comp_rho_const} on $\bar{L}$, the choices of parameters \eqref{eq:rho_sol}, \eqref{eq:kappa_sol}, \eqref{eq:nu_sol} and \eqref{eq:lambda_sol} insure that \eqref{eq:parameters_task} holds.
\end{proof}

\section{Proof of Lemma~\ref{lemma:negative_parameters} (\nameref{lemma:negative_parameters})}
\label{sec:negative_parameters}

We use the notations from the proof of Theorem~\ref{theorem:main_theorem}.
\LEMMANEGATIVETERMS*

\begin{proof}
  Since $p \geq 0$ and $D_f(z^t, y^{t+1}) \geq 0$ for all $t \geq 0,$ the inequality \eqref{eq:lemma:negative_parameters:bregman} is satisfied if 
  \begin{equation*}
    \begin{aligned}
      &\frac{4 \omega L_{\max}}{n \bar{L}} + \kappa 8 \left(1 + \frac{p}{\beta}\right) L_{\max} + \nu_{t} \left(\frac{\gamma_{t+1}}{\bar{L} + \Gamma_{t+1} \mu}\right)^2\left(\frac{4 \omega L_{\max}}{p n} + \frac{8L}{p \alpha}\right) + \rho 8 L +  \\
    &\qquad\qquad+ \lambda \left(4 \left(1 + \frac{2 p}{\tau}\right) L + \frac{8 \tau^2 \omega L_{\max}}{n}\right) + \theta_{t+1} - 1 \leq 0.
    \end{aligned}
    \end{equation*}
    Note that $\theta_{t+1} \leq \frac{1}{4}$ for all $t \geq 0.$ Therefore, it is sufficient to show that
    \begin{equation*}
    \begin{aligned}
        &\frac{4 \omega L_{\max}}{n \bar{L}} + \kappa 8 \left(1 + \frac{p}{\beta}\right) L_{\max} + \nu_{t} \left(\frac{\gamma_{t+1}}{\bar{L} + \Gamma_{t+1} \mu}\right)^2\left(\frac{4 \omega L_{\max}}{p n} + \frac{8L}{p \alpha}\right) + \rho 8 L +  \\
      &\qquad\qquad+ \lambda \left(4 \left(1 + \frac{2 p}{\tau}\right) L + \frac{8 \tau^2 \omega L_{\max}}{n}\right) \leq \frac{3}{4}.
    \end{aligned}
    \end{equation*}
    In the view of \eqref{eq:nu_sol}, we have to show that
    \begin{equation*}
      \begin{aligned}
          &\frac{4 \omega L_{\max}}{n \bar{L}} + \kappa 8 \left(1 + \frac{p}{\beta}\right) L_{\max} + \widehat{\nu} \left(\frac{\gamma_{t+1} \theta_{t+1}}{\bar{L} + \Gamma_{t+1} \mu}\right)^2\left(\frac{4 \omega L_{\max}}{p n} + \frac{8L}{p \alpha}\right) + \rho 8 L +  \\
        &\qquad\qquad+ \lambda \left(4 \left(1 + \frac{2 p}{\tau}\right) L + \frac{8 \tau^2 \omega L_{\max}}{n}\right) \leq \frac{3}{4}.
    \end{aligned}
    \end{equation*}
    Using Lemma~\ref{lemma:learning_rates}, we have $\frac{\gamma_{t+1} \theta_{t+1}}{\bar{L} + \Gamma_{t+1} \mu} \leq \frac{\gamma_{t+1} \theta_{t+1}}{\bar{L} + \Gamma_{t} \mu} \leq \frac{1}{\bar{L}}$, so it is sufficient to show that
    \begin{equation}
      \begin{aligned}
          &\frac{4 \omega L_{\max}}{n \bar{L}} + \kappa 8 \left(1 + \frac{p}{\beta}\right) L_{\max} + \frac{\widehat{\nu}}{\bar{L}^2} \left(\frac{4 \omega L_{\max}}{p n} + \frac{8L}{p \alpha}\right) + \rho 8 L +  \\
        &\qquad\qquad+ \lambda \left(4 \left(1 + \frac{2 p}{\tau}\right) L + \frac{8 \tau^2 \omega L_{\max}}{n}\right) \leq \frac{3}{4}. \label{eq:bregman_where_to_sub}
    \end{aligned}
    \end{equation}
    
    From this point all formulas in this lemma are generated by the script from Section~\ref{sec:jupyter_notebook} (see Section 5 in Section~\ref{sec:jupyter_notebook}).

    Let us substitute \eqref{eq:rho_sol}, \eqref{eq:kappa_sol}, \eqref{eq:lambda_sol}, and \eqref{eq:nu_sol_hat} to the last inequality and obtain the inequality from Section~\ref{sec:sym_comp_bregman}.
    The conditions from Section~\ref{sec:sym_comp_bregman_const} insure that the inequality from Section~\ref{sec:sym_comp_bregman} holds. It left to prove \eqref{eq:lemma:negative_parameters:norm}. Since $p \geq 0,$ $\ExpSub{t}{\norm{u^{t+1} - u^t}^2} \geq 0$ and $\theta_{t+1}^2 \geq 0$ for all $t \geq 0,$ the inequality \eqref{eq:lemma:negative_parameters:norm} holds if 
    \begin{align}
      &\frac{4}{\bar{L}} \left(\frac{L}{2} + \kappa 4 \left(1 + \frac{p}{\beta}\right) \widehat{L}^2 + \rho 4 L^2 + \lambda \left(2 \left(1 + \frac{2 p}{\tau}\right) L^2 + \frac{4 \tau^2 \omega \widehat{L}^2}{n}\right) \right) \leq 1. \label{eq:dist_where_to_sub}
    \end{align}
    Let us substitute \eqref{eq:rho_sol}, \eqref{eq:kappa_sol} and \eqref{eq:lambda_sol} to the last inequality and obtain the inequality from Section~\ref{sec:sym_comp_dist}. The inequality from Section~\ref{sec:sym_comp_dist} holds if $\bar{L}$ satisfy the inequalities from Section~\ref{sec:sym_comp_dist_const}.
\end{proof}

\section{Symbolically Computed Constraints for $\bar{L}$ Such That The Term w.r.t. $\kappa$ is less or equal $\nicefrac{1}{2}$ in \eqref{eq:kappa_expand}}
\label{sec:sym_comp_kapap_const}
  \formulakappaconst

\section{Symbolically Computed Constraints for $\bar{L}$ Such That The Term w.r.t. $\rho$ is less or equal $\nicefrac{1}{2}$ in \eqref{eq:rho_expand}}
\label{sec:sym_comp_rho_const}
  \formularhoconst

\section{Symbolically Computed Expression \eqref{eq:bregman_where_to_sub}}
\label{sec:sym_comp_bregman}
\begin{align*}
  \formulabregman
\end{align*}

\section{Symbolically Computed Constraints for $\bar{L}$ Such That The Inequality from Section~\ref{sec:sym_comp_bregman} Holds}
\label{sec:sym_comp_bregman_const}
  \formulabregmanconst

\section{Symbolically Computed Expression \eqref{eq:dist_where_to_sub}}
\label{sec:sym_comp_dist}
\begin{align*}
  \formuladist
\end{align*}

\section{Symbolically Computed Constraints for $\bar{L}$ Such That The Inequality from Section~\ref{sec:sym_comp_dist} Holds}
\label{sec:sym_comp_dist_const}
  \formuladistconst

\section{Symbolical Check That The Constraints from Sections~\ref{sec:sym_comp_kapap_const}, \ref{sec:sym_comp_rho_const}, \ref{sec:sym_comp_bregman_const} and \ref{sec:sym_comp_dist_const} Follow From The Constraint \eqref{eq:constr_lipts}}
\label{sec:sym_comp_check}

Note that the inequalities from Lemma~\ref{lemma:aux_ineq_L} follow from \eqref{eq:constr_lipts}. Therefore, the inequalities from Sections~\ref{sec:sym_comp_kapap_const}, \ref{sec:sym_comp_rho_const}, \ref{sec:sym_comp_bregman_const} and \ref{sec:sym_comp_dist_const} follow from \eqref{eq:constr_lipts}, if they follow from the inequalities from Lemma~\ref{lemma:aux_ineq_L}. We now present it\footnote{"($a$) follows from ($b$),($c$)" means that one should use ($b$) and ($c$) to get ($a$). It is possible that $b = c$, thus one should apply $b$ (= $c$) two times.}. These results are checked and generated using the script in Section~\ref{sec:jupyter_notebook} (see Section 6 in Section~\ref{sec:jupyter_notebook})

\begin{align*}
  \formulafollows
\end{align*}



# Jupyter Notebook for Symbolic Computations

## 1 Import Necessary Libraries

```python
import os

from IPython.display import display

import sympy
from sympy import Symbol

from utils import FileWriter, ConstraintsAggregator
from utils import get_factors, get_term, search_factors,↵
    ↪latex_repres_of_inequality
```

## 2 Initialize a File for Results

```python
file_path = '../paper/result.txt'
if os.path.exists(file_path):
    os.remove(file_path)
fw = FileWriter(file_path)
```

## 3 Initialize Symbols From the Paper

```python
rho = Symbol('rho', nonnegative=True)
kappa = Symbol('kappa', nonnegative=True)
lmbda = Symbol('lambda', nonnegative=True)

omega = Symbol('omega', nonnegative=True)
p = Symbol('p', positive=True)
tau = Symbol('tau', positive=True)
beta = Symbol('beta', positive=True)
alpha = Symbol('alpha', positive=True)
n = Symbol('n', positive=True)

L = Symbol('L', positive=True)
L_hat = Symbol(r'\hat{L}', positive=True)
```

```
L_max = Symbol(r'L_{\max}', positive=True)

L_bar = Symbol(r'\bar{L}', positive=True)
```

# 4 Assistant for "First Symbolically Computed" Lemma

## 4.1 Calculate Expressions

```
[ ]: nu_hat_exp = (8 / alpha) * (p * L / 2
                                 + kappa * 4 * p * (1 + p / beta) * L_hat**2
                                 + rho * 4 * p * L**2
                                 + rho * (8 * p / tau) * (2 * p * (1 + 2 * p / tau) *
      ↪L**2
                                                         + 4 * p * tau**2 * omega *
      ↪L_hat**2 / n))
     nu_hat = Symbol(r'\hat{\nu}', nonnegative=True)
     # The left hand sides of the inequalities
     rho_lhs = nu_hat * (8 / (alpha * p * L_bar ** 2))
     kappa_lhs = ((8 * p * omega) / (n * L_bar * beta)
                  + nu_hat * (8 * omega) / (n * beta * L_bar**2)
                  + rho * (8 * p) / tau * (8 * tau ** 2 * omega) / (n * beta))
```

## 4.2 Display Them

```
[ ]: display(nu_hat_exp)
     display(rho_lhs)
     display(kappa_lhs)
```

## 4.3 Symbolically Calculate The Steps From The Proof

```
[ ]: constraints_agg = ConstraintsAggregator()

     rho_lhs = rho_lhs.subs(nu_hat, nu_hat_exp)
     kappa_lhs = kappa_lhs.subs(nu_hat, nu_hat_exp)

     # Group Terms w.r.t. kappa
     kappa_lhs_poly = sympy.poly(kappa_lhs, kappa)
     kappa_coef = kappa_lhs_poly.all_coeffs()
     kappa_lhs = kappa_lhs.expand().collect(kappa)
     fw.write(sympy.latex(kappa_lhs) + ", \label{eq:kappa_expand}")

     # Find Conditions When The Coefficients Near Kappa <= 1/2
     terms = sympy.expand(kappa_coef[0]).args
     latex_string = constraints_agg.add_constraints(terms)
     fw.write(latex_string)
```

```python
# Define kappa
kappa_solution = (2 * kappa_coef[1]).expand()
fw.write("\kappa \eqdef " + sympy.latex(kappa_solution) + ". \label{eq:
    ↪kappa_sol}")

# Group Terms w.r.t. rho
rho_lhs = rho_lhs.subs(kappa, kappa_solution)
rho_lhs = rho_lhs.expand().collect(rho)
fw.write(sympy.latex(rho_lhs) + ", \label{eq:rho_expand}")

# Find Conditions When The Coefficients Near rho <= 1/2
rho_lhs_poly = sympy.poly(rho_lhs, rho)
rho_coef = rho_lhs_poly.all_coeffs()
terms = sympy.expand(rho_coef[0]).args
latex_string = constraints_agg.add_constraints(terms)
fw.write(latex_string)

# Define rho
rho_solution = (2 * rho_coef[1]).expand()
fw.write("\\rho \eqdef " + sympy.latex(rho_solution) + ". \label{eq:rho_sol}")
```

# 5  Assistant for "Second Symbolically Computed" Lemma

## 5.1  First Inequality

```python
[ ]: bregman_coef = ((4 * omega * L_max) / (n * L_bar)
                      + kappa * 8 * (1 + p / beta) * L_max
                      + (nu_hat / L_bar**2) * (4 * omega * L_max / (p * n) + 8 * L /␣
    ↪(p * alpha))
                      + rho * 8 * L
                      + lmbda * (4 * (1 + (2 * p) / tau) * L + 8 * tau**2 * omega *␣
    ↪L_max / n))
    display(bregman_coef)

    lmbda_solution = rho * 8 * p / tau

    # Substitue all known expressions to the term
    bregman_coef = \
        bregman_coef.subs(nu_hat, nu_hat_exp).subs(lmbda, lmbda_solution).
    ↪subs(kappa, kappa_solution).subs(rho, rho_solution)
    bregman_coef = bregman_coef.expand().simplify().expand()
    latex_string = latex_repres_of_inequality(bregman_coef)
    fw.write(latex_string)

    # Find Conditions When bregman_coef is Less or Equal <= 1/2
    terms = bregman_coef.args
```

```
latex_string = constraints_agg.add_constraints(terms)
fw.write(latex_string)
```

## 5.2 Second Inequality

```
[ ]: dist_coef = (4 / L_bar) * (L / 2
                               + kappa * 4 * (1 + p / beta) * L_hat ** 2
                               + rho * 4 * L**2
                               + lmbda * (2 * (1 + 2 * p / tau) * L**2 + 4 * tau**2 *↵
     ↪omega * L_hat**2 / n))
     display(dist_coef)

     dist_coef = \
         dist_coef.subs(lmbda, lmbda_solution).subs(kappa, kappa_solution).subs(rho,↵
     ↪rho_solution)
     dist_coef = dist_coef.expand().simplify().expand()
     latex_string = latex_repres_of_inequality(dist_coef, rhs="1")
     fw.write(latex_string)
     terms = dist_coef.args
     latex_string = constraints_agg.add_constraints(terms)
     fw.write(latex_string)
```

# 6 Check That The Constraints Follow From The Inequalities from The "Auxillary Inequalities" Lemma.

## 6.1 The Inequalities from The "Auxillary Inequalities" Lemma:

```
[ ]: from collections import OrderedDict
     _contstant = 1
     proposals = OrderedDict([
         ("eq:lipt:max_2",        (L_max * omega * p**2) / (beta**2 * n)),
         ("eq:lipt:max_1",        (L_max * omega * p) / (beta * n)),
         ("eq:lipt:max",          (L_max * omega) / (n)),
         ("eq:lipt:l_l_max_2",    (sympy.sqrt(L * L_max) * p * sympy.sqrt(omega *↵
     ↪tau)) / (alpha * beta * sympy.sqrt(n))),
         ("eq:lipt:l_l_max_1",    (sympy.sqrt(L * L_max) * sympy.sqrt(p * omega *↵
     ↪tau)) / (alpha * sympy.sqrt(beta * n))),
         ("eq:lipt:l_l_max_p",    (sympy.sqrt(L * L_max) * sympy.sqrt(p * omega *↵
     ↪tau)) / (alpha * sympy.sqrt(n))),
         ("eq:lipt:hat_2",        (L_hat * p * sympy.sqrt(omega * tau)) / (alpha *↵
     ↪beta * sympy.sqrt(n))),
         ("eq:lipt:hat_1",        (L_hat * sympy.sqrt(p * omega * tau)) / (alpha *↵
     ↪sympy.sqrt(beta * n))),
         ("eq:lipt:hat_p",        (L_hat * sympy.sqrt(p * omega * tau)) / (alpha *↵
     ↪sympy.sqrt(n))),
```

```python
    ("eq:lipt:hat_alpha_2",   (L_hat * p * sympy.sqrt(omega)) / (beta * sympy.
↪sqrt(alpha * n))),
    ("eq:lipt:hat_alpha_1",   (L_hat * sympy.sqrt(p * omega)) / (sympy.sqrt(beta␣
↪* alpha * n))),
    ("eq:lipt:hat_no_alpha_2",(L_hat * p * sympy.sqrt(omega)) / (beta * sympy.
↪sqrt(n))),
    ("eq:lipt:hat_no_alpha_1",(L_hat * sympy.sqrt(p * omega)) / (sympy.sqrt(beta␣
↪* n))),
    ("eq:lipt:plain",         (L) / (alpha)),
    ("eq:lipt:plain_p_alpha", (L * p) / (alpha * tau)),
    ("eq:lipt:plain_no_alpha",(L)),
    ("eq:lipt:double_lipt_2", sympy.cbrt((L * L_hat**2 * omega * p**4) /␣
↪(alpha**2 * beta**2 * n * tau**2))),
    ("eq:lipt:double_lipt_1", sympy.cbrt((L * L_hat**2 * omega * p**3) /␣
↪(alpha**2 * beta * n * tau**2))),
])
const = 660508
for k, proposal in proposals.items():
    proposals[k] = (const * proposal / L_bar).expand()
proposals_factors = [get_factors(proposal) for _, proposal in proposals.items()]
```

## 6.2   Search The Right Inequalities for Each Constraint

```python
# Takes ~1 hour on a laptop
import tqdm
constraints = constraints_agg.get_constraints()
constraints_inequalities = []
for i, factor in enumerate(tqdm.tqdm(constraints)):
    num_of_base_factors_to_use = -int(factor[L_bar])
    path = search_factors(factor, num_of_base_factors_to_use, proposals_factors,
                          pos_hints=[L, L_max, L_hat, omega, p],
                          neg_hints=[alpha, n])
    constraints_inequalities.append(path)
    assert path is not None
```

```python
text = constraints_agg.prepare_text_from_proposals(proposals,␣
↪constraints_inequalities)
fw.write(text)
```



\newpage
\subsection{File utils.py}

\definecolor{codegreen}{rgb}{0,0.6,0}
\definecolor{codegray}{rgb}{0.5,0.5,0.5}
\definecolor{codepurple}{rgb}{0.58,0,0.82}
\definecolor{backcolour}{rgb}{0.95,0.95,0.92}

\lstdefinestyle{mystyle}{
    backgroundcolor=\color{backcolour},   
    commentstyle=\color{codegreen},
    keywordstyle=\color{magenta},
    numberstyle=\tiny\color{codegray},
    stringstyle=\color{codepurple},
    basicstyle=\ttfamily\footnotesize,
    breakatwhitespace=false,         
    breaklines=true,                 
    captionpos=b,                    
    keepspaces=true,                 
    numbers=left,                    
    numbersep=5pt,                  
    showspaces=false,                
    showstringspaces=false,
    showtabs=false,                  
    tabsize=2,
    linewidth=17cm
}

\lstset{style=mystyle}

\begin{lstlisting}[language=Python]
import os
from collections import defaultdict
from copy import copy

import sympy


class _defaultdictwithconst(defaultdict):
    def __init__(self, *args, **kwargs):
        super(_defaultdictwithconst, self).__init__(*args, **kwargs)
        self.const = None
    
    def __copy__(self):
        obj = super(_defaultdictwithconst, self).__copy__()
        obj.const = self.const
        return obj


def get_factors(term):
    """ Converts a sympy expression with the sympy.Mul type to a dictionary with factors.
        Ex: A^2 / B^(1/2) -> {A: 2, B: -1/2}
    Args:
        param1: sympy expression (sympy.Mul).
    """
    factors = _defaultdictwithconst(int)
    factors.const = 1
    const_assigned = False
    assert isinstance(term, sympy.Mul)
    for el in term.args:
        free_symbols = list(el.free_symbols)
        if len(free_symbols) == 0:
            assert not const_assigned
            const_assigned = True
            factors.const = int(el)
            continue
        assert len(free_symbols) == 1
        power = el.as_powers_dict()[free_symbols[0]]
        assert isinstance(power, sympy.Rational)
        factors[free_symbols[0]] = power
    return factors


def get_term(factors):
    """ The inverse function to the get_factors function: converts factors to a term
        Ex: {A: 2, B: -1/2} -> A^2 / B^(1/2)
    Args:
        param1: dictionary with factors
    """
    result = factors.const
    for k, v in factors.items():
        result = result * k ** v
    return result


def search_factors(factors, num_of_base_factors_to_use, base_factors, pos_hints=[], neg_hints=[]):
    """ Checks if it possible to decompose 'factors' using num_of_base_factors factors from 'base_factors' 
    Args:
        param1: dictionary with factors
        param2: number of base factors to use
        param3: list of dictionaries with factors
        param4 and param5: hints to the algorithm that helps to improve performance
    Returns:
        A list with indices (with repetitions) of 'base_factors' that compose 'factors'.
        If it not possible, then the function returns None.
    """
    def _check_hints(factors_):
        for k in pos_hints:
            assert factors_[k] >= 0
        for k in neg_hints:
            assert factors_[k] <= 0
    _check_hints(factors)
    for base in base_factors:
        _check_hints(base)
    def _search_factors(factors, num_of_base_factors_to_use, path, base_factors):
        if num_of_base_factors_to_use == 0:
            return factors.const <= 1 and all([factors[k] == 0 for k in factors])
        for index, choice in enumerate(base_factors):
            factor_choice = copy(factors)
            factor_choice.const = factor_choice.const / choice.const
            for k, v in choice.items():
                factor_choice[k] = factor_choice[k] - v
            skip = False
            for k in pos_hints:
                if factor_choice[k] < 0:
                    skip = True
                    break
            for k in neg_hints:
                if factor_choice[k] > 0:
                    skip = True
                    break
            if not skip and _search_factors(factor_choice, num_of_base_factors_to_use - 1, path, base_factors):
                path.append(index)
                return True
        return False
    path = []
    if _search_factors(factors, num_of_base_factors_to_use, path, base_factors):
        return path
    else:
        return None


class FileWriter(object):
    def __init__(self, file_path):
        self._file_path = file_path
        assert not os.path.exists(file_path)
    
    def write(self, text):
        with open(self._file_path, "a") as fd:
            fd.write("{}\n".format(text))


class ConstraintsAggregator(object):
    def __init__(self):
        self._constraints = []
        
    def get_constraints(self):
        return self._constraints
    
    def add_constraints(self, terms):
        denom_constant = self._denom_constant(terms)
        terms_global_index = []
        for term in terms:
            terms_global_index.append(len(self._constraints))
            self._constraints.append(get_factors(term * denom_constant))
        return self._prepare_text(terms, terms_global_index)
    
    def prepare_text_from_proposals(self, proposals, constraints_inequalities):
        # assert len(constraints_inequalities) == len(self._constraints)
        text = ''
        proposals_labels = list(proposals.keys())
        for global_index, ineq in enumerate(constraints_inequalities):
            constraint_label = self._get_label(global_index)
            text += "& \eqref{{{}}}".format(constraint_label) + " \\textnormal{ follows from } "
            for i, index_proposal in enumerate(ineq):
                text += "\eqref{{{}}}".format(proposals_labels[index_proposal])
                if i == len(ineq) - 1:
                    text += "."
                else:
                    text += ","
            text += "\\\\"
        return text
    
    def _denom_constant(self, terms):
        return 2 * len(terms)
    
    def _get_label(self, global_index):
        return "eq:sympy:constraints:{}".format(global_index)
    
    def _prepare_text(self, terms, terms_global_index):
        text = ''
        denom_constant = self._denom_constant(terms)
        num_per_column = 3
        for i, (global_index, term) in enumerate(zip(terms_global_index, terms)):
            if i % num_per_column == 0:
                text += r"\begin{tabularx}{1.2\linewidth}{XXX}"
            text += r"\begin{equation}"
            text += sympy.latex(term) + " \leq \\frac{{1}}{{{}}}".format(denom_constant)
            text += "\label{{{}}}".format(self._get_label(global_index))
            text += r"\end{equation}"
            if i % num_per_column == (num_per_column - 1):
                text += r"\end{tabularx} "
            else:
                text += "&"
        def ceildiv(a, b):
            return -(a // -b)
        for i in range(len(terms), ceildiv(len(terms), num_per_column) * num_per_column):
            if i % num_per_column == (num_per_column - 1):
                text += r"\end{tabularx} "
            else:
                text += "&"
        return text


def latex_repres_of_inequality(coef, rhs="\\frac{{3}}{{4}}"):
    latex_string = "&"
    for i, term in enumerate(coef.args):
        latex_string = latex_string + sympy.latex(term)
        if i == len(coef.args) - 1:
            latex_string += " \leq " + rhs
        else:
            if i % 3 == 2:
                latex_string += " \\\\ & +"
            else:
                latex_string += " + "
    return latex_string
\end{lstlisting}

\newpage
\section{Experiments}
\label{sec:experiments}
\subsection{Setup}

We now conduct experiments on the practical logistic regression task with LIBSVM datasets \citep{chang2011libsvm} (under the 3-clause BSD license). The experiments were implemented in Python 3.7.9. The distributed environment was emulated on machines with Intel(R) Xeon(R) Gold 6248 CPU @ 2.50GHz. In each plot we show the relations between the total number of coordinates transmitted from and to the server and function values. The parameters of the algorithms are taken as suggested by the corresponding theory, except for the stepsizes that we fine-tune from a set $\{2^i\,|\,i \in [-20, 20]\}$. For \algname{2Direction}, we use parameters from Theorem~\ref{cor:realistic} and finetune the step size $\bar{L}.$

We solve the logistic regression problem:
$$f_i(x_1, \dots, x_{c}) \eqdef -\frac{1}{m} \sum_{j=1}^m\log\left(\frac{\exp\left(a_{ij}^\top x_{y_{ij}}\right)}{\sum_{y=1}^c \exp\left(a_{ij}^\top x_{y}\right)}\right), $$
where $x_1, \dots, x_{c} \in \R^{d}$, $c$ is the number of unique labels,
$a_{ij} \in \R^{d}$ is a feature of a sample on the $i$\textsuperscript{th} worker, $y_{ij}$ is a corresponding label and $m$ is the number of samples located on the $i$\textsuperscript{th} worker. The Rand$K$ compressor is used to compress information from the workers to the server, the Top$K$ compressor is used to compress information from the server to the workers. The performance of algorithms is compared on CIFAR10 \citep{krizhevsky2009learning} (\# of features $ = 3072$, \# of samples equals $\num[group-separator={,}]{50000}$), and \textit{real-sim} (\# of features $ = 20958$, \# of samples equals $\num[group-separator={,}]{72309}$) datasets.

\subsection{Results}

In Figure~\ref{fig:real-sim}, \ref{fig:cifar10_10} and \ref{fig:cifar10_100} we provide empirical communication complexities of our experiments. For each algorithm, we show the three best experiments. The experiments are collaborative with our theory. \algname{2Direction} enjoys faster convergence rates than \algname{EF21-P + DIANA} and \algname{AGD}.

\begin{figure}[h]
  \centering
  \begin{subfigure}{.49\textwidth}
    \centering
    \includegraphics[width=\textwidth]{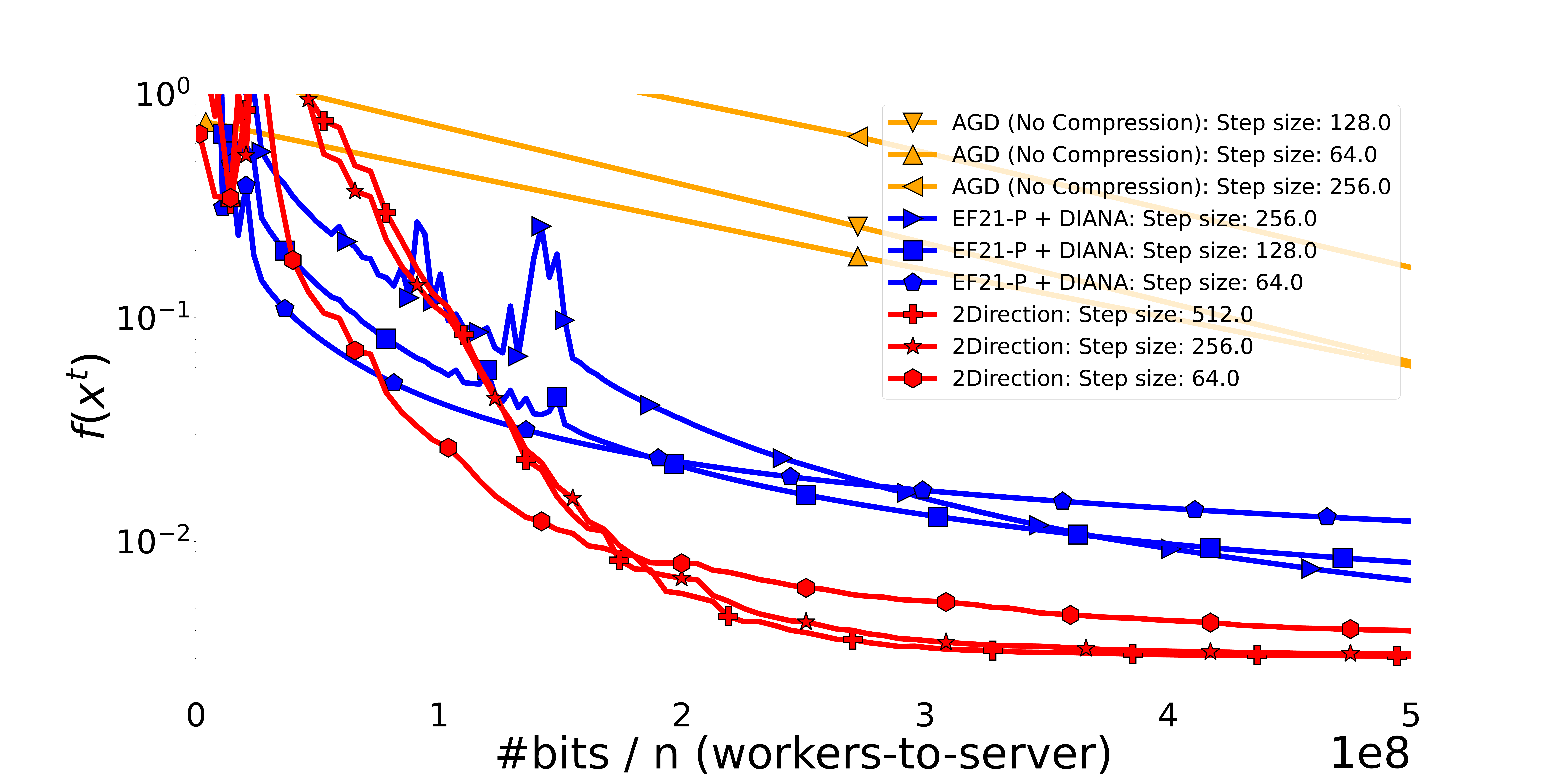}
  \end{subfigure}
  \begin{subfigure}{.49\textwidth}
      \centering
      \includegraphics[width=\textwidth]{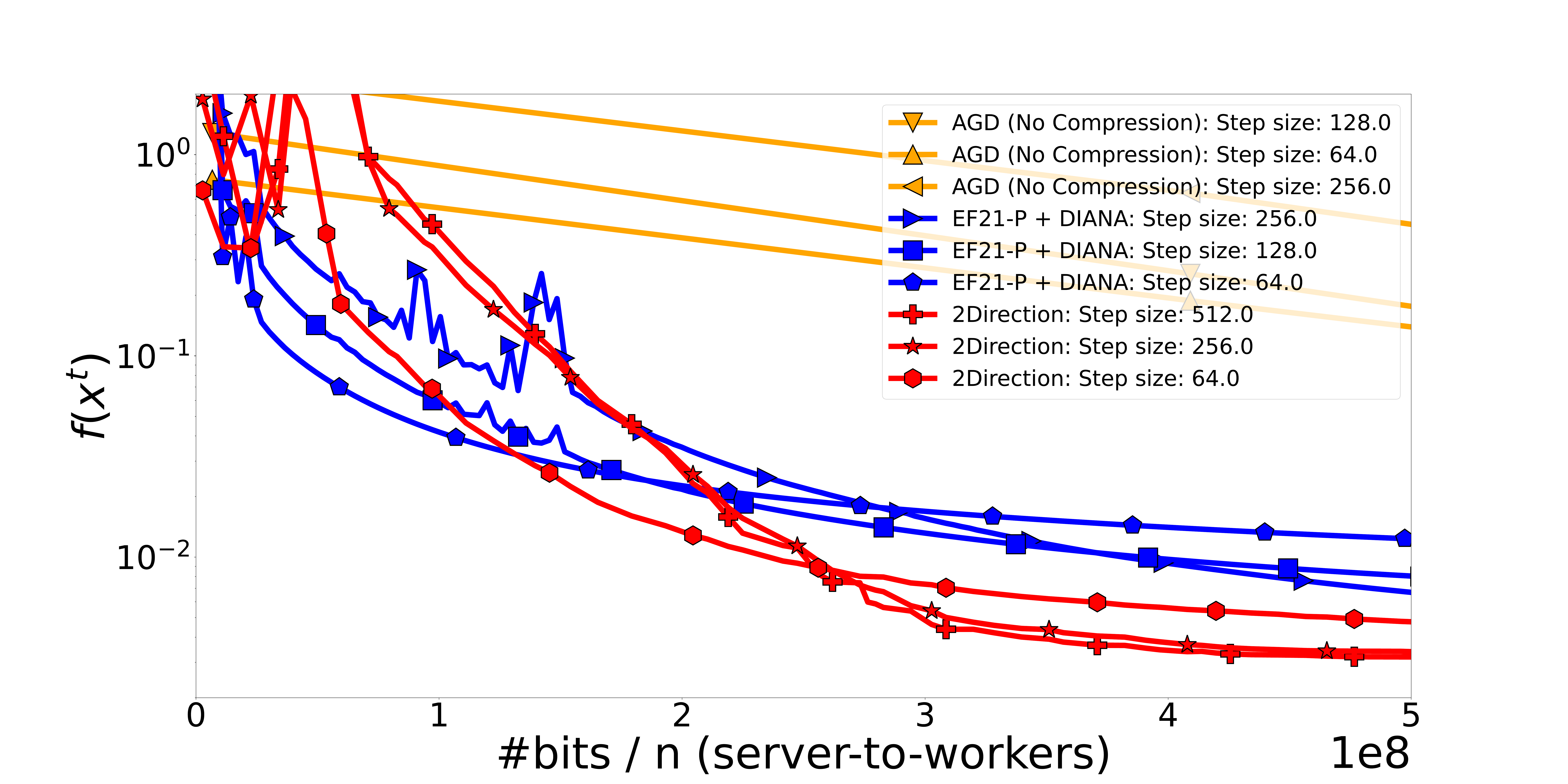}
    \end{subfigure}\hfill
  \caption{Logistic Regression with \textit{real-sim} dataset. \# of workers $n = 100.$ $K = 1000$ in all compressors.}
  \label{fig:real-sim}
\end{figure}

\begin{figure}[h]
  \centering
  \begin{subfigure}{.49\textwidth}
    \centering
    \includegraphics[width=\textwidth]{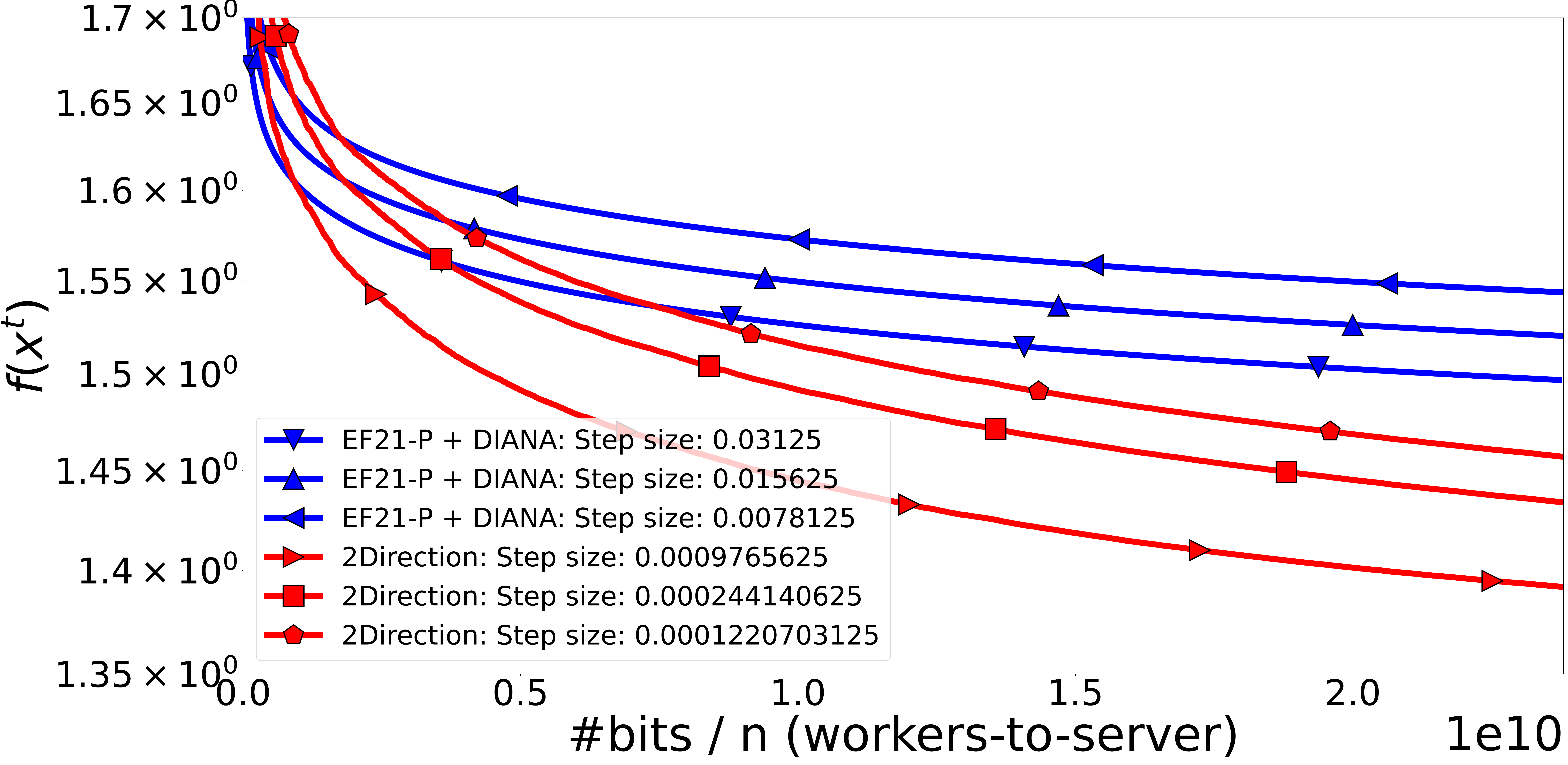}
  \end{subfigure}
  \begin{subfigure}{.49\textwidth}
      \centering
      \includegraphics[width=\textwidth]{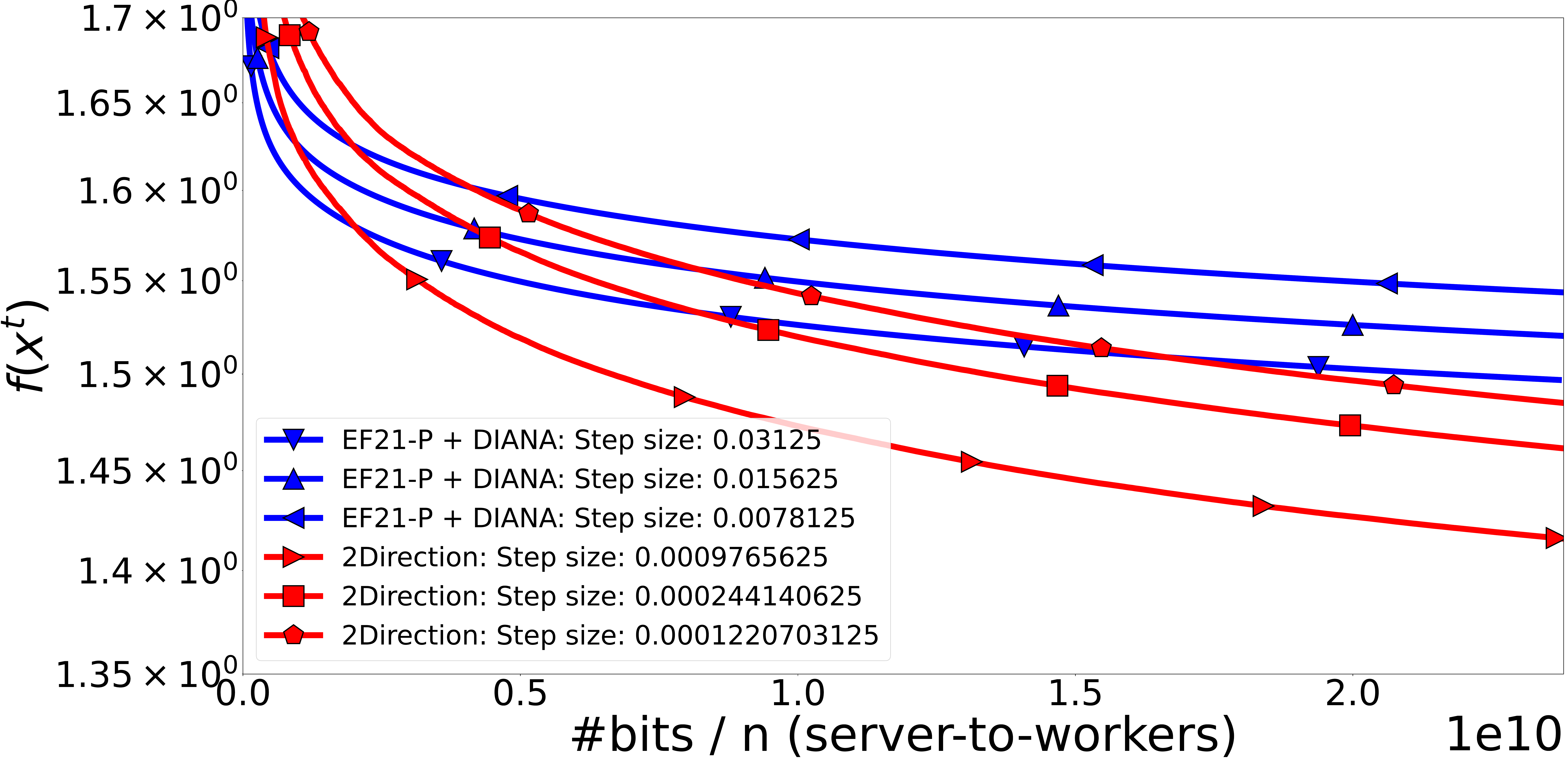}
    \end{subfigure}\hfill
  \caption{Logistic Regression with \textit{CIFAR10} dataset. \# of workers $n = 10.$ $K = 1000$ in all compressors.}
  \label{fig:cifar10_10}
\end{figure}

\begin{figure}[h]
  \centering
  \begin{subfigure}{.49\textwidth}
    \centering
    \includegraphics[width=\textwidth]{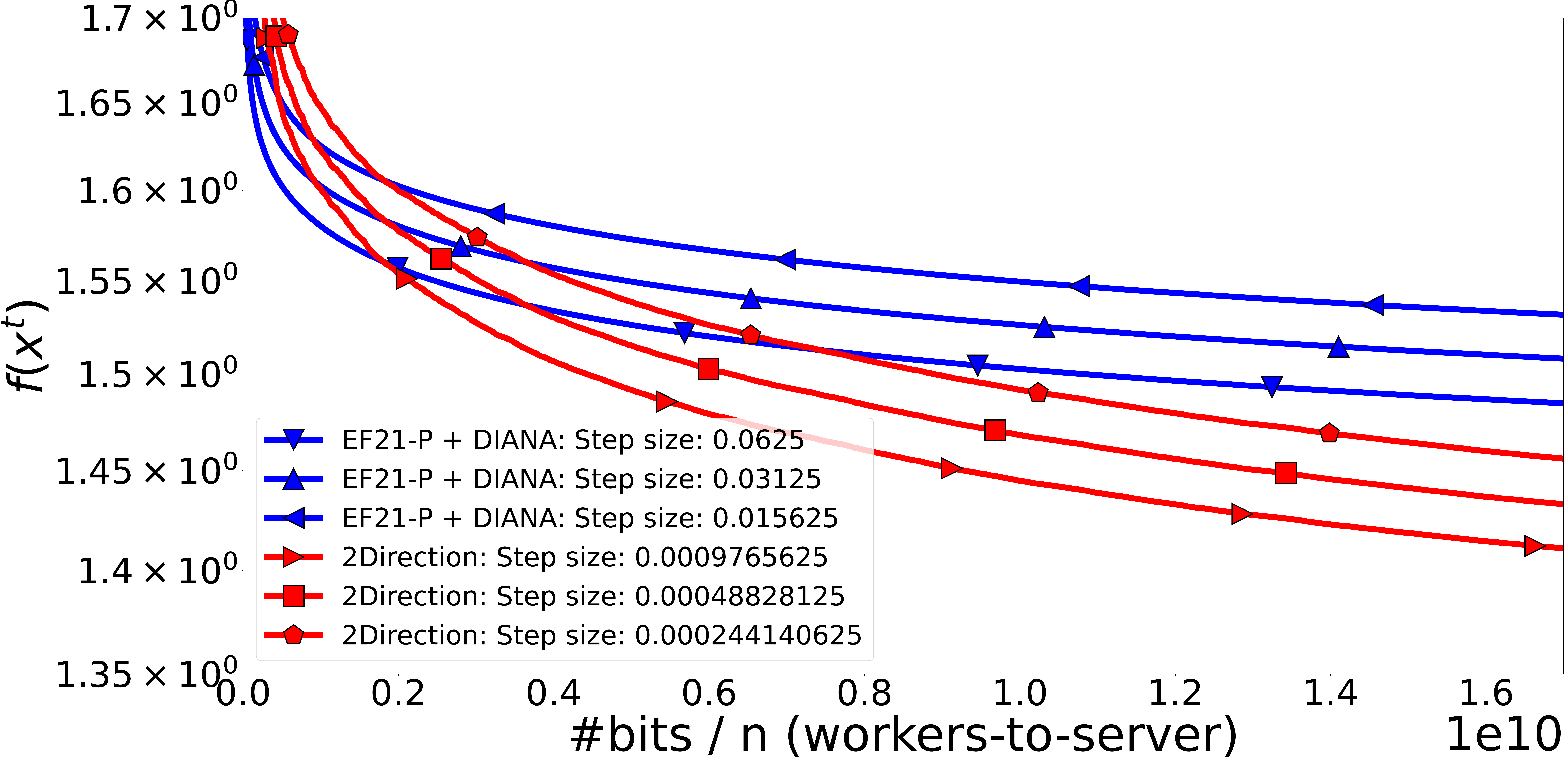}
  \end{subfigure}
  \begin{subfigure}{.49\textwidth}
      \centering
      \includegraphics[width=\textwidth]{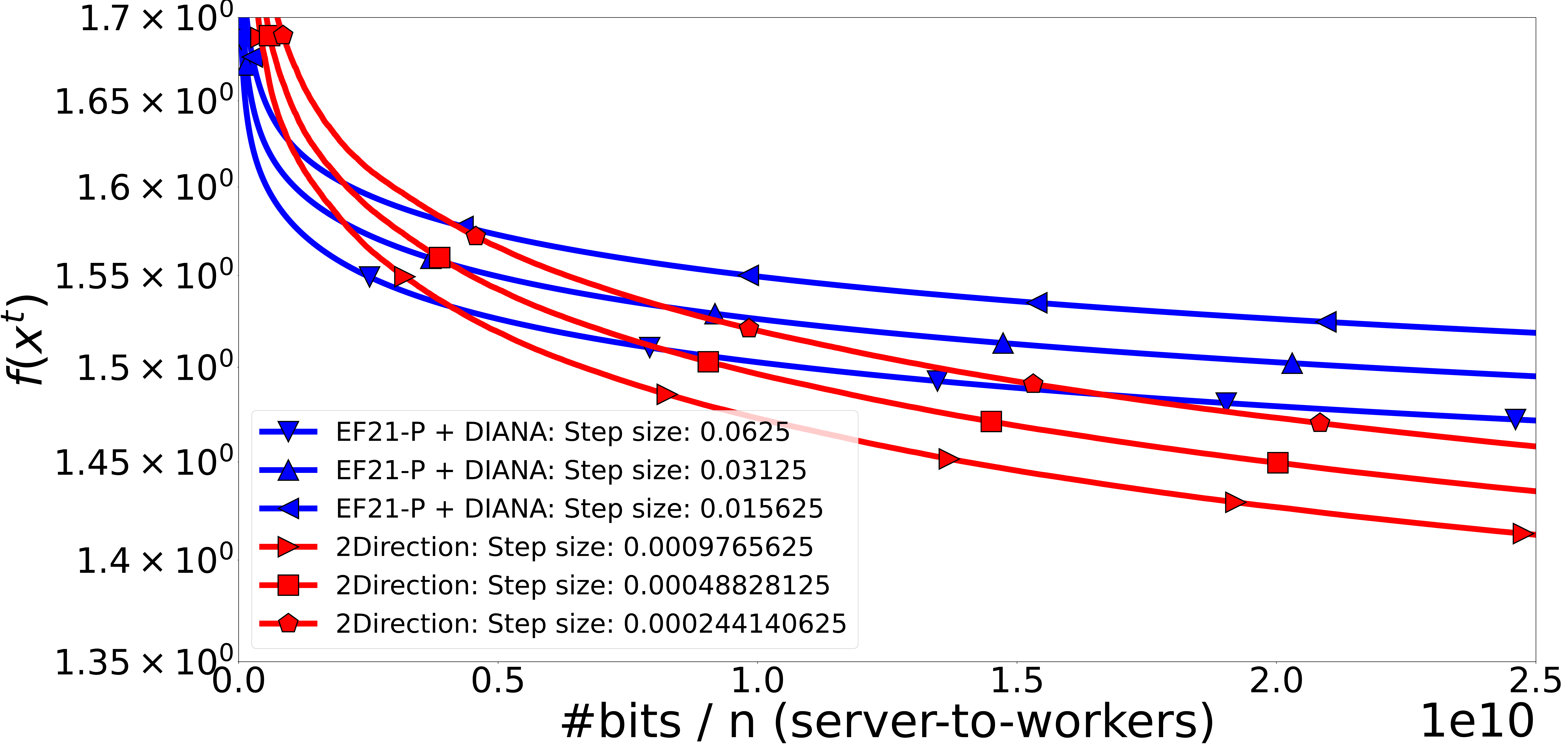}
    \end{subfigure}\hfill
  \caption{Logistic Regression with \textit{CIFAR10} dataset. \# of workers $n = 100.$ $K = 1000$ in all compressors.}
  \label{fig:cifar10_100}
\end{figure}

\newpage
\section{Convergence Rate of \algnamebig{CANITA} obtained by \citet{li2021canita}}
\label{sec:canita}
In their Equation (54), \citet{li2021canita} derive the following bound for their \algname{CANITA} method:
\begin{align*}
  \Exp{F^{T+1}} \leq \cO\left(\max\left \{\frac{(1 + \omega)^3}{T^3}, \frac{(1 + b)(\beta + 3/2) L}{T^2}\right \}\right).
\end{align*}
In the regime when $\omega \geq n,$ choosing $b = \omega$ and $\beta = \Theta\left(\frac{\omega}{n}\right)$ in their Equation (10) gives
\begin{align*}
  \Exp{F^{T+1}} &\leq \cO\left(\max\left \{\frac{(1 + \omega)^3}{T^3}, \frac{(1 + b)(\beta + 3/2) L}{T^2}\right \}\right) \\
  &= \cO\left(\max\left \{\frac{(1 + \omega)^3}{T^3}, \frac{\omega (\omega / n + 3/2) L}{T^2}\right \}\right) \\
  &= \cO\left(\max\left \{\frac{(1 + \omega)^3}{T^3}, \frac{\omega^2 L}{n T^2} \right \}\right).
\end{align*}
This means that the correct convergence rate of the \algname{CANITA} method \citep{li2021canita} is 
\begin{align}
  \label{eq:canita_rate}
  T = \begin{cases}
   \Theta\left(\frac{\omega}{\varepsilon^{1/3}} + \frac{\omega}{\sqrt{n}} \sqrt{\frac{L}{\varepsilon}}\right), & \omega \geq n, \\
   \Theta\left(\frac{\omega}{\varepsilon^{1/3}} + \left(1 + \frac{\omega^{3/4}}{n^{1/4}}\right) \sqrt{\frac{L}{\varepsilon}}\right), & \omega < n.
  \end{cases}
\end{align}
Comparing this result with our Theorem~\ref{theorem:main_theorem_general_convex} describing the convergence of our method \algname{2Direction}, one can see that in the low accuracy regimes (in particular, when $\frac{\omega}{\varepsilon^{1/3}}$ dominates in \eqref{eq:canita_rate}), our result improves $\Theta\left(\frac{1}{\varepsilon^{1/3}}\right)$ to at least $\Theta\left(\log \frac{1}{\varepsilon}\right).$ However, the dependence $\Theta\left(\log \frac{1}{\varepsilon}\right)$ should not be overly  surprising as it was observed by \citet{lan2019unified} already, albeit in a somewhat different context.

\section{Comparison with \algnamebig{ADIANA}}
\label{sec:comp_adiana}
We now want to check that our rate \eqref{eq:optimistic_compl} restores the rate from \citep{ADIANA}. Since \algname{ADIANA} only compresses from the workers to the server, let us take $r = 0,$ the identity compressor operator $\cC^{P}(x) = x$ for all $x \in \R^d,$ which does not perform compression, and, as in \citep{ADIANA}, consider the optimistic case, when $L_{\textnormal{max}} = L$. For this compressor, we have $\alpha = 1$ in \eqref{eq:biased_compressor}. Note that $\mu^r_{\omega, \alpha} = 0.$ Thus the iteration complexity \eqref{eq:optimistic_compl} equals
\begin{align}
  T^{\textnormal{optimistic}} &= \widetilde{\Theta}\Bigg(\max\Bigg\{\sqrt{\frac{L}{\mu}}, \sqrt{\frac{L (\omega + 1)}{n^{1/3} \mu}}, \sqrt{\frac{L (\omega + 1)^{3/2}}{\sqrt{n} \mu}}, \sqrt{\frac{L \omega (\omega + 1)}{n \mu}}, (\omega + 1)\Bigg\}\Bigg) \nonumber \\
  &= \widetilde{\Theta}\Bigg(\max\Bigg\{\sqrt{\frac{L}{\mu}}, \sqrt{\frac{L (\omega + 1)^{3/2}}{\sqrt{n} \mu}}, \sqrt{\frac{L \omega (\omega + 1)}{n \mu}}, (\omega + 1)\Bigg\}\Bigg), \label{eq:adiana_rate}
\end{align}
where we use Young's inequality: $\sqrt{\frac{L}{\mu}}\sqrt{\frac{(\omega + 1)}{n^{1/3}}} \leq \sqrt{\frac{L}{\mu}}\sqrt{\frac{1}{3} \times 1^3 + \frac{2}{3} \frac{(\omega + 1)^{3/2}}{\sqrt{n}}}.$ Without the server-to-worker compression, Algorithm~\ref{alg:bi_diana} has the same iteration \eqref{eq:adiana_rate} and communication complexity as \citep{ADIANA}.

\end{document}